\documentclass[a4paper,11pt, reqno]{amsart}

\usepackage{style}

\crefname{claim}{Claim}{Claims}
\crefname{assumption}{Assumption}{Assumptions}
\usepackage{tikz-cd}

\title[weighted Four-Dimensional Fano Hypersurfaces of K3 Type]{
Weighted Four-Dimensional Fano Hypersurfaces of K3 Type
}

\author{Valeria Bertini}
\address[Valeria Bertini]{Dipartimento di Matematica Federigo Enriques, Università Statale di Milano. Via Cesare Saldini 50, 20133, Milano, Italy}
\email{valeria.bertini@unimi.it}

  \author{Francesco Antonio Denisi}
  \address[F.\ A.\ Denisi]{Université Paris Cité and Sorbonne Université, CNRS, IMJ-PRG, F-75013 Paris, France}
\email{denisi@imj-prg.fr}

\author{Enrico Fatighenti}
\address[Enrico Fatighenti]{Dipartimento di Matematica, Università di Bologna. Piazza di Porta San Donato 5, 40126, Bologna, Italy}
\email{enrico.fatighenti@unibo.it}

\author{Annalisa Grossi}
\address[Annalisa Grossi]{Dipartimento di Matematica, Università di Bologna. Piazza di Porta San Donato 5, 40126, Bologna, Italy}
\email{annalisa.grossi3@unibo.it}

\date{\today}

\makeatletter
\@namedef{subjclassname@2020}{%
 \textup{2020} Mathematics Subject Classification}
\makeatother


\usepackage{framed}

\usepackage[backend = biber,
    style = alphabetic,
    citestyle = alphabetic,
    maxbibnames=10,
    giveninits=true,
    doi=false,
    url=false,
    maxalphanames=99
]{biblatex}
\usepackage{hyperref}
\hypersetup{
	colorlinks=true,
	linkcolor=red,
	urlcolor=teal,
	citecolor=blue}

\theoremstyle{plain}

\theoremstyle{definition}

\newcommand{\N}{\mathbf{N}}
\newcommand{\Z}{\mathbf{Z}}
\newcommand{\Q}{\mathbf{Q}}

\newcommand{\C}{\mathbf{C}}
\newcommand{\A}{\mathbf{A}}

\renewcommand{\P}{\mathbf{P}}
\renewcommand{\OO}{\mathcal{O}}

\begin{document}

\begin{abstract}
We study weighted Fano fourfolds of K3 type realized as hypersurfaces in weighted projective spaces. Under the additional assumption that the singular locus has dimension at most one, we prove that only finitely many such families exist. We provide the complete list and analyze their rationality properties, as well as their singularities.
\end{abstract}

\maketitle

\tableofcontents

\section{Introduction}

Fano varieties of K3 type (abbreviated as FK3) form a remarkable class of smooth projective varieties that lie at the intersection of several major areas of algebraic geometry. Their name reflects the fact that, at their core, they exhibit a Hodge-theoretic structure closely resembling that of a K3 surface.

The prototypical example of a FK3 is the classical cubic fourfold, defined as the zero locus of a cubic polynomial in $\P^5$. This variety has been the subject of intense investigation, giving rise to a wide range of deep questions—most notably concerning its rationality and the explicit construction of associated hyperk\"ahler manifolds. In layman's terms, a smooth cubic fourfold is conjectured to be rational if and only if it has a K3 associated to it in the Hodge theoretical or derived categorical sense - see \cite{AT14} for the connection between the two worlds. In the derived category setting, this conjecture is widely known as \emph{Kuznetsov's conjecture}. There are many excellent surveys addressing these questions: as an example, we mention \cite{Deb22,Has16,Kuz16}.

Beyond the case of the cubic fourfold, one can pose analogous questions for other FK3 varieties, such as Gushel--Mukai fourfolds and many others; see, for instance, \cite{Fat22} for a collection of relevant results. In recent years, the third-named author and several collaborators have worked to produce new examples of FK3 varieties; see, for example, \cite{BFMT21}.

To enlarge the class of known examples, it is natural to slightly relax the definition of Fano varieties—for instance, by allowing mild singularities such as cyclic quotient singularities. A convenient and fruitful strategy is to consider hypersurfaces in five-dimensional weighted projective spaces. For this broader class of varieties, one can formulate the same guiding questions as in the smooth case, suitably adapted to the singular setting. In particular, one may investigate their rationality properties and ask whether an appropriate version of Kuznetsov’s conjecture still applies. 
Another compelling direction is to explore whether these weighted FK3 varieties might lead to the construction of singular analogues of hyperkähler manifolds, and hence of singular analogues of Enriques manifolds. These topics have attracted growing attention in recent years (see, e.g.\, \cite{BCS24, BGMM24, brakkee2024, DRTX24,FM21, Men22}). 
In this paper, we focus on the first question, leaving the second for future work.

Working with (even mildly) singular Fano fourfolds naturally presents some drawbacks, most notably in understanding their geography. Indeed, unless one imposes strong constraints on the type of singularities—such as terminality, which guarantees finiteness by Birkar’s celebrated results—one immediately obtains infinitely many families of weighted Fano hypersurfaces; see, for instance, the works of Johnson--Kollár, and Brown--Kasprzyk \cite{JK01,Brown_Kasprzyk}, and the database \cite{altinokgraded} for many lists of examples.

On the other hand, restricting solely to the terminal case risks excluding several examples of particular interest for our purposes. In this work, we adopt a hybrid approach: we consider quasi-smooth Fano fourfold hypersurfaces whose singular locus has dimension at most one.
This condition is not sufficient, in general, to ensure finiteness. However, when combined with the additional constraints required for a Fano variety to be of K3 type, it does yield the boundedness of the families of FK3 fourfold hypersurfaces satisfying our singularity assumption. Moreover, our restriction on the singular locus is optimal from the boundedness perspective: as observed in \cref{rmk:unbounded}, it is not difficult to construct infinitely many FK3 fourfold hypersurfaces whose singular locus has dimension two.

The purpose of this paper is twofold. On the one hand, we prove—by purely combinatorial means—the boundedness of FK3 hypersurfaces under our singularity assumption, and we provide the complete list of such families. On the other hand, we investigate the birational geometry of the resulting examples, with particular emphasis on their rationality and the structure of their singular loci.

\medskip

We now describe in detail the content of our paper. In \cref{section FK3 weighted fourfolds} we start by providing some basic facts and definitions on weighted hypersurfaces, $\mathbf{Q}$-Fano fourfolds and FK3 varieties. As a first result, we address the finiteness of FK3 fourfolds under our assumption on the singular locus. 

To prove this result, we proceed as follows. We start from a weighted Fano fourfold hypersurface $X_d \subseteq \P(a_0, \ldots, a_5)$. We observe that imposing the K3-type condition, together with the constraint on the singular locus, yields strong combinatorial restrictions on the admissible weights $a_0, \ldots, a_5$; see \cref{thm ai+a5=d}. Then, in \cref{lemma the K3 is quasi smooth}, \cref{lemma the K3 is well-formed}, and in \cref{cor the k3 is a k3}, we show how to associate to each $X_d$ a quasi-smooth, well-formed weighted K3 surface that is not a linear cone — of which there are finitely many (the \emph{famous 95} list by Reid, see \cite{reid}).

The complete list of FK3 weighted fourfold hypersurfaces satisfying our singularity assumption can thus be constructed starting from the list of K3 surfaces, and then adding weights in such a way that the dimension of the singular locus does not exceed one.

We summarize our main results in the following theorem. From now on, a weighted hypersurface is always assumed to be quasi-smooth, not a linear cone and well-formed, see \cref{subsec weighted hypersurf and FK3}.

\begin{theorem}[cf.\ \cref{thm ai+a5=d}, \cref{thm list complete}, \cref{rationality}] \label{main:thm}
\ 
\begin{itemize}
    \item There exist exactly 244 families of FK3 fourfold hypersurfaces $X_d \subseteq \P(a_0, \ldots, a_5)$ whose singular locus has dimension at most 1. These are listed in \cref{table fourfolds}.
    \item Every FK3 weighted fourfold hypersurface $X_d \subseteq \P(a_0, \ldots, a_5)$ with singular locus of dimension at most 1 that is not a cubic fourfold is rational. 
\end{itemize}
\end{theorem}

In \cref{section: the examples}, we provide the complete list and give a complete geometric description of all the examples. In \cref{prop blow-up}, we explicitly describe a birational map from each $X_d$ to a weighted projective space $\P^4$. In fact, we identify the (weighted) blow up of $X_d$ at the point dual to the coordinate $x_5$ with the blow up of $\P(a_0, \ldots, a_4)$ along a naturally associated complete intersection K3 surface. This proves the rationality of $X_d$, and also shows that the K3-type Hodge structure and the Calabi--Yau 2 (CY2) subcategory in the derived category of perfect complexes are commutative, as discussed in \cref{subse:geometricdescription}, consistently with what is expected or happens in other known examples of fourfolds of K3 type.

In \cref{section:sing}, we provide a detailed analysis of the singularity types of the families we have described. As a corollary, in \cref{cor:terminality} we list 197 families of terminal weighted Fano (FK3) fourfolds. Finally, in \cref{sec:others}, we examine in detail two special families of examples, which are included in \cref{table fourfolds from del Pezzo}. These do not belong to the sub-case we have focused on, as their singular loci have dimension 2. These examples were first described by Laza, Pearlstein, and Zhang in \cite{LPZ18}, as part of a broader classification. The same tools for studying rationality and the commutativity of the Hodge structure that we develop in this paper (including the \emph{double suspension trick}, see \cref{remark double suspension}) apply to all the elements in their list except these two. In particular, we can neither conclude that the general elements of these two families are rational, nor explain their K3 structure via an actual K3 surface. Therefore, we conclude the paper by discussing a geometric construction that applies to certain special elements of these families, and we end with a brief discussion on how these cases might be relevant to a possible extension of Kuznetsov’s rationality conjecture for cubic fourfolds.

\subsection*{Acknowledgments}
We are grateful to Alessandro Frassineti, Lucas Li Bassi, and Fabio Tanturri for useful discussions and comments. This research has been partially funded by the European Union - NextGenerationEU under the
National Recovery and Resilience Plan (PNRR) - Mission 4 Education and research - Component 2
From research to business - Investment 1.1 Notice Prin 2022 - DD N. 104 del 2/2/2022, from title “Symplectic varieties: their interplay with Fano manifolds and derived categories”, proposal code 2022PEKYBJ – CUP J53D23003840006. The first-named author has been partially funded by the PRIN 2020 2020KKWT53003 - ‘Curves, Ricci flat Varieties and their Interactions’, and wishes to thank the MIUR Excellence Department of Mathematics, University of Milano. The second-named author is supported by the European Research Council (ERC) under the European Union's Horizon 2020 research and innovation programme (ERC-2020-SyG-854361-HyperK). The first, second, and fourth-named authors are grateful to the European Research Council (ERC) for supporting a visiting period during the ‘Junior Trimester Program: "Algebraic geometry: derived categories, Hodge theory, and Chow groups’, organized in collaboration with the aforementioned programme.  The first, third and fourth authors are members of the INDAM-GNSAGA group.

\section{FK3 weighted fourfolds}\label{section FK3 weighted fourfolds}

\subsection{Weighted hypersurfaces and FK3}\label{subsec weighted hypersurf and FK3}

Given $(\underline a)=(a_0,...,a_n)\in \Z_{>0}^{n+1}$, define $S(\underline a)$ to be the graded ring $\C[x_0,...,x_n]$ graded by $\deg x_i=a_i$. The weighted projective space $\P(\underline a)=\P(a_0,...,a_n)$ is defined as \[\P(\underline a):=\Proj(S(\underline a)).\] 

Observe that $\Proj(S(a_0,...,a_n))\simeq \Proj(S(ka_0,...,ka_n))$, as the corresponding graded rings are isomorphic. Furthermore, if $a_0,...,a_n$ are positive integers with no common factors and $q=\text{gcd}(a_1,...,a_n)$, then $\Proj(S(a_0,...,a_n))\simeq\Proj(a_0,a_1/q,...,a_n/q)$, see \cite[Lemma 5.7]{Fletcher00}. Because of that, we will always assume that the weighted projective space $\P(a_0,...,a_n)$ is \textit{well-formed}, i.e.\  that $\gcd(a_0,...,\hat a_i,...,a_n)=1$ for every $i=0,..,n$.

For a closed subvariety $X\subseteq \P(\underline a)$, we denote by $C^*_X$ the punctured affine cone over $X$, i.e.\ $C^*_X=p^{-1}X$ for $p:\A^{n+1}\smallsetminus \{0\}\rightarrow \P(\underline a)$ the quotient map. 
\medskip

In what follows, a \textit{weighted hypersurface} $X_d\subseteq \P(\underline a)$, i.e.\ 
\[X_d=V(f_d)=\Proj (S(\underline a)/(f_d))\] 
with $\deg(f_d)=d$ in $S(\underline a)$, is assumed to be \textit{quasi-smooth}, i.e.\ $C^*_X$ is smooth, \textit{not a linear cone}, i.e.\ $d\neq a_i$ for all $i$, and \textit{well-formed}, i.e.\ $\P(\underline a)$ is well-formed and $X_d\cap \P(\underline a)^{\mathrm{sing}}$ has codimension at least 2 in $X_d$. Being quasi-smooth implies that the singularities of $X$ can only arise from those of the ambient space, i.e.\ the ones coming from the $\C^*$-action on $C_X$. Not being a linear cone limits redundancy, as it excludes hypersurfaces isomorphic to a weighted projective space or to a cone in a weighted projective space.

\medskip
These hypotheses on the hypersurface translate into numerical conditions on the weights $a_i$ and the degree $d$. Indeed, for $X_d=V(f_d)\subseteq \P(\underline a)$ being well-formed is equivalent to the condition:
\[\gcd(a_0,...,\hat a_i, ..., a_n)=1 \ \ \forall i \ \ \ \text{and} \ \ \ \gcd(a_0,...,\hat a_i,..., \hat a_j,...,a_n)|d \ \ \forall i,j.\]

Furthermore, being quasi-smooth and not a linear cone for a weighted hypersurface translates into the following criterion.

\begin{proposition}{\cite[Theorem 8.1]{Fletcher00}}\label{prop quasismooth}
    Let $X_d=V(f_d)\subseteq \P(\underline a)$ be a general weighted hypersurface of degree $d$. Then $X_d$ is quasi-smooth, not a linear cone, if and only if the following numerical conditions are satisfied: $a_i\neq d$ for all $i$, and each nonempty subset $I\subseteq \{0,...,n\}$, either 
    \begin{enumerate}
        \item $d=\sum_{i\in I}n_i a_i$, with coefficients $n_i\in \N$, or
        \item there are at least $|I|$ indices $j\notin I$ such that $d-a_j=\sum_{i\in I}n_i a_i$, with coefficients $n_i\in \N$.
    \end{enumerate}
    In particular, $d>a_i$ for all $i$.
\end{proposition}

We are interested in weighted hypersurfaces that are Fano varieties.

\begin{definition}\label{def:Fano}
A Fano variety is a normal, $\mathbf{Q}$-factorial projective variety $X$ with at worst klt singularities, such that $-K_X$ is ample.
\end{definition}

In literature, such singular analogues of Fano varieties are often referred as $\mathbf{Q}$-Fano.

By \cite[Theorem 3.1.6]{Dolgachev82} being quasi-smooth for $X_d=V(f_d)\subseteq\P(\underline a)$ ensures having quotient singularities, which are klt, for example, by \cite[Corollary 3.12]{Prokhorov01}. Moreover, algebraic varieties with only quotient singularities are $\mathbf{Q}$-factorial, for example, by \cite[Proposition 5.15]{KM98}.
Then a weighted hypersurface is Fano as soon as the anti-canonical divisor is ample.  Note that for a weighted hypersurface $X_d\subseteq\P(\underline a)$ the adjunction formula holds \cite[Theorem 3.3.4]{Dolgachev82}: 
\begin{equation}\label{eq adjuction}
    \OO(K_{X_d})\simeq \mathcal O_{X_d}\Bigl(d-\sum_{i=0}^n a_i\Bigl),
\end{equation}
hence $X_d\subseteq\P(\underline a)$ is Fano as soon as $d<\sum_{i=0}^n a_i$. The well-formed hypothesis is necessary for the adjunction formula to hold (see \cite[Examples 6.15 (i),(ii)]{Fletcher00} for some counterexamples).

On a quasi-smooth variety $X$ is it possible to define the notion of a pure $\Q$-Hodge structure,  see Steenbrink \cite[Theorem~1.12]{Ste76}. Consider in fact the smooth locus $j:X_0\hookrightarrow X$ and $\Omega^{[p]}_X:=j_* \Omega^p_{X_0}$. Then we can define $H^{p,q}(X)$ as in the smooth case and moreover $H^{p,q}(X) \cong H^q\left(X, \Omega^{[p]}_X\right)$. The Hodge decomposition takes then the form $$H^k(X,\C) = \bigoplus_{p+q=k}   H^q\left(X, \Omega^{[p]}_X\right).$$

It therefore makes sense to formulate the following definition.

\begin{definition}[Cf. Definition 1.1 in \cite{FatMon21}]
    A weighted hypersurface $X_d\subseteq \P(\underline a)$ of dimension $2t$ is a \textit{Fano of K3 type} (in short, \textit{FK3}) if it is a Fano variety and $H^{2t}(X,\C)$ is a Hodge structure of K3 type, i.e.\
    \begin{itemize}
        \item $\max\{q-p \mid H^{p,q}(X_d) \neq 0\}=2$,
        \item $h^{t-1,t+1}=1$.
\end{itemize}
\end{definition}

The Hodge structure of a weighted hypersurface $X_d=V(f_d)\subseteq\P(\underline a)$ is computed by the Jacobian ring $R_{f_d}$ of $f_d$, i.e.\
\[
R_{f_d}=S(\underline a)/J_{f_d}
\]
where $J_{f_d}=(\partial f_d/\partial x_i)_{i=0,..,n}$ is the Jacobian ideal of $f_d$.

\begin{theorem}[Griffiths' Residue Theorem]{\cite[Theorem 4.3.2]{Dolgachev82}}\label{thm Hodge structure of Xd}
Let $X_d=V(f_d)\subseteq \P(\underline a)$ be a weighted hypersurface of degree $d$ and dimension $2t$. Then
\[
h^{2t-j,j}_{\mathrm{prim}}(X_d)=\dim (R_{f_d})_{(j+1)d-\sum_{i=0}^{2t+1} a_i}.
\] 
\end{theorem}

\begin{corollary}
A weighted hypersurface $X_d=V(f_d)\subseteq \P(\underline a)$ of dimension $2t$ has a K3 structure in degree $2t$ if and only if $\sum_{i=0}^{2t+1}a_i=d\cdot t$. Furthermore, we have
\[
h^{t,t}_{\mathrm{prim}}(X_d)= \dim (R_{f_d})_{d}.
\]      
\end{corollary}

We conclude that a weighted hypersurface $X_d\subseteq \P(\underline a)$ of dimension $2t$ is a FK3 if and only if $\sum_{i=0}^{2t+1}a_i=d\cdot t$ and $2t\ge 4$. In the case $2t=2$, the weighted hypersurface $X_d$ is an actual K3 surface.

Notice that this numerical condition is of course exactly the same as the one appearing in the derived categorical context in \cite[Corollary 4.2]{Kuz19}. We will come back later to the comparison between our setting and this one.

\begin{definition}\label{def:singularK3}
 A K3 surface is a normal projective surface $S$ with at worst du Val (hence, by \cite[Theorem 4.20]{KM98}, canonical) singularities such that $h^1(S,\mathcal{O}_S)=0$, and $K_S\sim0$. 
\end{definition}

The weighted hypersurface $X_d$ of dimension 2 has du Val singularities by \cite[Corollary 2.9]{ETW23}. Note that the minimal resolution of a K3 surface is a smooth K3 surface, for example, by \cite[Lemma 3.2]{LMM24}. 

\medskip

Finally, given a weighted hypersurface $X_d\subseteq\P(\underline a)$, assume that $\codim_{X_d}X_d^{\mathrm{sing}}\ge 3$, i.e.\ that
\[\gcd(a_0,...,\hat a_i,\hat a_j,...,a_5)=1 \ \ \ \text{and} \ \ \  \gcd(a_0,\dots,\hat a_i,\hat a_j,\hat a_k,\dots,a_5)| d.\]
There is a dual reason to ask for this somewhat arbitrary condition on the dimension of the singular locus. The first one is about a boundedness property. In fact, this condition on the dimension of the singular locus is weaker than terminality, an assumption which would guarantee the existence of a finite number of families of $\mathbf{Q}$-Fano varieties with this hypothesis, see \cref{rmk:unbounded} and also \cref{table sigularities} for a detailed analysis on the type of singularities of our hypersurfaces. Surprisingly, the FK3 condition together with the dimension of the singular locus provides the boundedness. Notice that if we remove the FK3 hypothesis, we can find an infinite number of families of such Fano fourfolds; for example by considering $\P(1^4,k)$ (i.e.\ the cones over the Veronese varieties $\nu_k(\P^3)$), with $k\in \mathbf{Z}_{>0}$. Also, increasing the dimension of the singular locus creates immediately infinitely many families of FK3, see again \cref{rmk:unbounded}. Therefore, the condition on the dimension of the singular locus is at the same time simple and optimal from the boundedness point of view.

The second, and equally important reason, to ask for this bound on the dimension of the singular locus, is that, under this hypothesis, Fano varieties admit a good deformation theory. As in \cite[A.4]{CK99, Sch71}, the deformation theory of our families is controlled by $\textrm{Ext}^1(\Omega^{[1]}_X, \mathcal{O}_X)$, since $H^0(\mathcal{E}xt^1(\Omega^{[1]}_X, \mathcal{O}_X))$ vanishes. Moreover, in our FK3 setting, we can show that $\textrm{Ext}^1(\Omega^{[1]}_X, \mathcal{O}_X) \cong H^1(T_X)$ is actually isomorphic to $H^1(T_S)$, where $S$ is the K3 naturally attached to $X$ via the birational transformation of  \cref{prop blow-up}, see \cref{rmk:defo}.

\subsection{The dimension 4 case} Let $X_d\subseteq \P(\underline a)=\P(a_0,...,a_5)$ be a FK3 weighted fourfold such that $\codim_{X_d}X_d^{\mathrm{sing}}\ge 3$, i.e.\ such that $\dim X_d^{\mathrm{sing}}\le 1$; we assume that the weights are ordered as $a_0\le...\le a_5$. Following \cref{subsec weighted hypersurf and FK3}, our assumptions translate into the following numerical conditions.

\begin{assumption}\label{conditions}The integers $a_i,d$ satisfy the following relations.
\begin{enumerate} 
    \item  $\gcd(a_0,...,\hat a_i,...,a_5)=1$ and \(\gcd(a_0,\dots,\hat a_i,\hat a_j,\dots,a_5)| d\) (well-formedness);
    \item $a_0\le...\le a_5<d$;
    \item $\sum a_i=2d$ (FK3 condition);
    \item \cref{prop quasismooth} is satisfied (quasi-smoothness and being not a linear cone);
    \item  $\gcd(a_0,...,\hat a_i,\hat a_j,...,a_5)=1$ and \(\gcd(a_0,\dots,\hat a_i,\hat a_j,\hat a_k,\dots,a_5)| d\) ($\dim X_d^{\mathrm{sing}}\le 1$).
\end{enumerate}
\end{assumption}
Observe that \cref{conditions}.(1) is implied by \cref{conditions}.(5), but we keep \cref{conditions}.(1) as in what follows it will happen to drop the dimension hypothesis on the singular locus.

\medskip

Our main goal in this section is to prove the boundedness of the FK3 weighted fourfolds with the limitation on the singular locus. We do not assume any extra divisibility condition which would make the combinatorics much easier to handle (in particular, Equation (A.8) in \cite{LPZ18} does not hold in our case).

However, it turns out that our hypotheses impose strong constraints on the weights, as expressed by the following result.

\begin{theorem}\label{thm ai+a5=d}
    Let \(X_d \subseteq \P(a_0, \ldots, a_5)\) be a FK3 weighted fourfold such that $\dim X_d^{\mathrm{sing}}\le 1$. If  $a_5>1$, then there exists an index $i\in\{0,\dots,4\}$ such that $a_i+a_5=d$.
\end{theorem}

The proof of the theorem is based on the following \cref{prop a4 divide d,prop a4=a5=d/2,prop d<=2a5}. Their proofs consist of elementary but involved computations in $a_i,d$, given \cref{conditions}; we write them in full detail because of the relevance of \cref{thm ai+a5=d} in the perspective of the current paper.

\begin{proposition}\label{prop d<=2a5}
    Let \(X_d \subseteq \P(a_0, \ldots, a_5)\) be a FK3 weighted fourfold such that $\dim X_d^{\mathrm{sing}}\le 1$. If  $a_5>1$, then \(d \leq 2a_5\).
\end{proposition}
\begin{proof}
Assume by contradiction that \(d > 2a_5\). We know that \(2d= a_0+ \ldots + a_5\leq 6a_5\), then \(d \leq 3a_5\), so we know that \begin{equation}\label{eq}
2a_5 < d \leq 3a_5.
\end{equation} 
We assume the notation in \cref{prop quasismooth}. By the quasi-smoothness hypothesis, if \(I=\{5\}\) we have \(a_5|d\) or \(a_5|(d-a_j)\) for an index \(j \neq 5\).
\begin{enumerate}

\item If \(a_5|d\) then \(d=ma_5\) with \(m \in \Z_{\geq2}\) which implies by \eqref{eq} that \(m=3\). Then \(a_0 + \ldots + a_5=6a_5\) implies \(a_i=a_5\) for all \(i\), which contradicts \cref{conditions}.(1) as $a_5>1$.

\item If \(a_5|(d-a_j)\) then \(d=a_j+ma_5\) with \(m \in \Z_{\geq 1}\). If \(m=1\), then \(d=a_j+a_5 \leq 2a_5\), which is not possible by \eqref{eq}.
If \(m \geq 3\), then \(d=ma_5+a_j>3a_5\), which contradicts \eqref{eq}. Then we have \(m=2\) and \(d=2a_5+a_j\), with \(a_j \leq a_5\).
As a consequence, we have \(a_0 + \ldots a_5= 4a_5+2a_j\) which means \(a_0 + \ldots a_4=3a_5+2a_j\). Since by \cref{conditions}.(1) and $a_5>1$ one has \(a_2+a_3+a_4 < 3a_5\), it turns out that \(a_0+a_1 > 2a_j\) for an index \(j \in \{0, \ldots, 4\}\) which gives \(a_j=a_0\). This implies that \(d=2a_5+a_0\) and furthermore
\begin{equation}\label{eq 0<1}
    a_0<a_1
\end{equation}
\begin{equation}\label{eq 4<5}
    a_4<a_5.
\end{equation}
The first inequality follows from $a_0+a_1>2a_0$. For the second one, assume by contradiction $a_4=a_5$; by quasi-smoothness, if $I=\{4,5\}$ then either $d=na_5=2a_5+a_0$, which is impossible, or there exists an index $j\neq 0,4,5$ such that $d-a_j=na_5$ for $n\ge 1$. If $n=1$ then $d=a_5+a_j=2a_5+a_0$, which is a contradiction, if $n\ge2$ then $2a_5\le na_5= d-a_j<d-a_0=2a_5$ by \eqref{eq 0<1}, contradiction.

From \(d=2a_5+a_0\) we have $2d=4a_5+2a_0=a_0 + \dots +a_5$, thus $3a_5+a_0=a_1+a_2+a_3+a_4$ and $(a_5-a_4)+(a_5-a_3)+(a_5-a_2)=a_1-a_0$. The addenda on the left are strictly positive by \eqref{eq 4<5}, hence
\begin{equation}\label{eq strict pos}
    a_1-a_0>a_5-a_4,\hspace{5mm} a_1-a_0>a_5-a_3,\hspace{5mm} a_1-a_0>a_5-a_2.
\end{equation}
\begin{claim}
    In the assumptions above, $d=n_ia_i$ with $n_i\in\Z_{\ge3}$, for $i=3,4$.
\end{claim}
\begin{proof}
We apply the quasi-smoothness criterion on $I=\{i,5\}$, $i=3,4$. We have the following possibilities.
\begin{enumerate}
    \item $d=na_i+ma_5$, for $n,m\in \mathbf Z_{\ge 0}$. 
    
    If $m\ge2$ then $2a_5+a_0=d=na_i+(m-2)a_2+2a_5$ hence $na_i+(m-2)a_5=a_0$, impossible as $m-2\ge0$ and $a_0<a_i$ by \cref{conditions}.(5) and \eqref{eq 4<5}: $a_0=a_i$ would imply $a_0=\dots=a_3=1$, hence $4+a_4+a_5=2d=4a_5+2$ which gives $a_4+2=3a_5>3a_4$, impossible as it would imply $a_4<1$. 

    If $m=1$ then $2a_5+a_0=d=na_i+a_5$, hence $na_i=a_0+a_5$ and $n\ge2$. Then $(n-1)a_i-a_0=a_5-a_i<a_1-a_0$ by \eqref{eq strict pos}, hence $a_i\le(n-1)a_i<a_1$, which is a contradiction.

    We conclude that $m=0$ and $d=na_i$. From $na_i=d=2a_5+a_0>2a_5$ follows $n\ge 3$, which is our claim.
    \item $d-a_j=na_i+ma_5$, for $n,m\in\mathbf Z_{\ge 0}$ and some $j\neq 0,4,5$. We want to show that this case cannot appear.
    
    By \eqref{eq 0<1} one has $a_0<a_j$, hence $2a_5=d-a_0>d-a_j=na_i+ma_5\ge ma_5$, which gives $m<2$.

    If $m=1$ then $d-a_j=na_i+a_5$, hence $n\ge 1$: if $n=0$ then $d-a_j=a_5=d-a_0$ which contradicts \eqref{eq 0<1} and $j\neq 0$. We have $a_0+\dots + a_5=2d=2na_i+2a_5+2a_j$ hence $a_0+a_1+a_2+a_k=(2n-1)a_i+a_5+2a_j$, where $k=4,3$ if $i=3,4$ respectively. From $n\ge1$ and \eqref{eq 4<5} we have $(2n-1)a_i+a_5\ge a_i+a_5>a_2+a_k$, hence $a_0+a_1>2a_j$ which contradicts \eqref{eq 0<1} and $j\neq 0$.

    We conclude $m=0$ and $d-a_j=na_i$, for $n\ge 1$. If $n=1$ then $a_i+a_j=d=2a_5+a_0>(a_i+a_j)+a_0$, which is a contradiction, hence $n\ge 2$. We have $a_0+\dots +a_5=2d=na_i+a_j+2a_5+a_0$ hence $a_1+a_2+a_k-a_j=a_5+(n-1)a_i$, where $k=4,3$ if $i=3,4$ respectively. We obtain $(a_1-a_j)+a_2+a_k=a_5+(n-1)a_i$, where $a_1-a_j\le 0$ as $j\neq 0$, hence $a_2+a_k\ge (n-1)a_i+a_5\ge a_i+a_5$, which is a contradiction by \eqref{eq 4<5}.
\end{enumerate}
\end{proof}
We conclude that there exists $k,n\in \mathbf Z_{\ge 3}$ such that $d=ka_4=na_3$. Hence $a_0+\dots+a_5=2d=ka_4+na_3$ and we obtain the following impossible inequality: $a_4+2a_3\le (k-2)a_4+(n-1)a_3=a_0+a_1+a_2+(a_5-a_4)<2a_1+a_2\le 2a_3+a_4$, where we used \eqref{eq strict pos} in the strict inequality.
\end{enumerate}
We find that the quasi-smoothness criterion is not satisfied for $I=\{5\}$, hence $d\le 2a_5$.
\end{proof}

\begin{proposition}\label{prop a4=a5=d/2}
    Let \(X_d \subseteq \P(a_0, \ldots, a_5)\) be a FK3 weighted fourfold such that $\dim X_d^{\mathrm{sing}}\le 1$. If $d=2a_5$ and $a_4$ divides $d$, then $a_4=a_5=\frac{d}{2}$.
\end{proposition}
\begin{proof}
If \(a_4|d\) then \(d=ka_4\) with \(k \in \Z_{\geq 2}\), as $a_4\le a_5$. Hence it holds \[a_0 + \ldots + a_4=2d-a_5=2ka_4-\frac{ka_4}{2}=\frac{3ka_4}{2}.\] By \cref{conditions}.(2) we have \(a_0+ \ldots + a_4 \leq 5a_4\), then 
\(\frac{3ka_4}{2} \leq 5a_4,\) which implies that \(3k \leq 10\), i.e.\ \(k\in \{2,3\}\) and consequently \(a_4=\frac{d}{2}\) or \(a_4=\frac{d}{3}\).  Assume by contradiction that $a_4=\frac{d}{3}$. 

\begin{claim}\label{claim 2a5=3a4=2}
    Let $X_d\subseteq \P(a_0,\dots,a_5)$ be a weighted FK3 satisfying \cref{conditions}.(1)-(4) and such that $d=2a_5=3a_4$. Then there exists an integer $k\ge 1$ such that \[(a_2,a_3,a_4,a_5,d)=(2k,2k,2k,3k,6k).\]
\end{claim}
\begin{proof}
We assume the notation in \cref{prop quasismooth}. By quasi-smoothness if $I=\{3\}$ we have the following possibilities:
\begin{enumerate}
    \item $d=ka_3$, with $k>2$ as $a_3\le a_4<a_5=\frac{d}{2}$. We have $a_5=\frac{k}{2}a_3$ and $a_4=\frac{k}{3}a_3$, hence 
    \[4a_3\ge a_0+a_1+a_2+a_3=2d-a_5-a_4=2ka_3-\frac{k}{2}a_3-\frac{k}{3}a_3=\frac{7}{6}ka_3\]
    hence $k\le\frac{24}{7}$, i.e.\ $k\le 3$; we conclude $k=3$, that is $a_3=a_4=\frac{d}{3}$.
    \item $d=ka_3+a_i$ for some $i\neq 3$ and $k\in \Z_{\ge1}$. If $i=5$ then $d=2ka_3$, that is a particular case of (a) above, hence $a_3=a_4=\frac{d}{3}$. 
    
    If $i<5$, assume by contradiction that $a_3<a_4$. We get $d=ka_3+a_i<ka_4+a_i=\frac{k}{3}d+a_i$, hence $\frac{d}{3}\ge a_i>\frac{3-k}{3}d$ which implies $k>2$. Furthermore $a_0+a_1+a_2+a_3=2d-a_5-a_4=2d-\frac{d}{2}-\frac{d}{3}=\frac{7}{6}d$, and from $a_0\le a_1\le a_2\le a_3$ we obtain $a_3\ge \frac{1}{4}(\frac{7}{6}d)=\frac{7}{24}d$, that gives $d<4a_3$ and hence $k\le 3$. We conclude that $k=3$ and $d=3a_3+a_i$. Finally $a_0+\dots +a_5=2d=3a_3+a_i+2a_5$ and $2a_5>a_4+a_5$ give $a_0+a_1+a_2>2a_3+a_i$, which is a contradiction for every $i$.
    
\end{enumerate}
We conclude that $a_3=a_4=\frac{d}{3}$. In order to argue on $a_2$ we apply the quasi-smoothness criterion on $I=\{2\}$, hence one of the following possibilities holds true:
\begin{enumerate}
    \item $d=ka_2$ for some $k>2$ (as $a_2<a_5$). We have $a_5=\frac{k}{2}a_2$ and $a_4=a_3=\frac{k}{3}a_2$, hence 
    \[3a_2\ge a_0+a_1+a_2=2d-a_5-a_4-a_3=2ka_2-\frac{k}{2}a_2-2\frac{k}{3}a_2=\frac{5}{6}ka_3\]
    hence $k\le\frac{18}{5}$, i.e.\ $k\le 3$; we conclude $k=3$, that is $a_2=a_3=a_4=\frac{d}{3}$.
    \item $d=ka_2+a_i$ for some $i\neq 2$ and $k\in \Z_{\ge1}$. If $i=5$ then $d=2ka_2$, that is a particular case of (a) above, hence $a_2=a_3=a_4=\frac{d}{3}$. 
    
    If $i<5$, assume by contradiction that $a_2<a_3$. We get $d=ka_2+a_i<ka_3+a_i=\frac{k}{3}d+a_i$, hence $\frac{d}{3}\ge a_i>\frac{3-k}{3}d$ which implies $k>2$. Furthermore $a_0+a_1+a_2=2d-a_5-a_4=-a_3=2d-\frac{d}{2}-2\frac{d}{3}=\frac{5}{6}d$, and from $a_0\le a_1\le a_2$ we obtain $a_2\ge \frac{1}{3}(\frac{5}{6}d)=\frac{5}{18}d$, that gives $d<4a_2$ and hence $k\le 3$. We conclude that $k=3$ and 
    \begin{equation}\label{eq d=3a2+ai}
        d=3a_2+a_i.
    \end{equation}
    We have $a_0+\dots +a_5=2d=3a_2+a_i+2a_5$, and $2a_5>a_4+a_5$ gives $a_0+a_1+a_3>2a_2+a_i$ which is a contradiction for $i=3,4$. 
    For the case $i=0,1$, we fix the notation
    \((a_3,a_4,a_5,d)=(2k,2k,3k,6k)\), hence $a_0+a_1+a_2=2d-a_5-a_4-a_3=5k$. We have $3a_2\ge a_0+a_1+a_2=5k$, hence from \eqref{eq d=3a2+ai} we obtain $6k\ge 5k+a_i$ and $a_i\le k$. Furthermore $a_0+a_1+a_2=5k=a_3+a_4+k$, hence $a_i=(a_3-a_j)+(a_4-a_2)+k$ where $j=1,0$ for $i=0,1$ respectively. If $a_2=a_3$ we are in case (1) above, otherwise $a_2<a_3$ implies $a_3-a_j>0$, $a_4-a_2>0$ and then $k<a_i\le k$, which is a contradiction.     
\end{enumerate}
We arrived to $a_2=a_3=a_4=\frac{d}{3}$, that implies the claim. 
\end{proof}
We arrived to $(a_2,a_3,a_4,a_5,d)=(2k,2k,2k,3k,6k)$ for some $k\ge 1$. To conclude, we apply \cref{conditions}.(5) and we obtain $k=1$, hence $d=6$ and $(\underline a)=(1,2^4,3)$, which contradicts again \cref{conditions}.(5). We find that this case is not possible and $a_4=a_5=\frac{d}{2}$.
\end{proof}

\begin{proposition}\label{prop a4 divide d}Let \(X_d \subseteq \P(a_0, \ldots, a_5)\) be a FK3 weighted fourfold such that $\codim_{X_d}X_d^{\mathrm{sing}}\ge 3$. If $d=2a_5$ then \(a_4\) divides $d$.
\end{proposition}
\begin{proof}
We assume the notation in \cref{prop quasismooth}. By quasi-smoothness if \(I=\{4\}\) we have \(a_4|d \) or \(a_4|d-a_j\) for an index \(j \neq 4\). Assume by contradiction that $a_4$ does not divide $d$, hence \(a_4|(d-a_j)\), i.e.\ 
\[d=ka_4+a_j,\hspace{3mm} j\neq4,\ k\in\Z_{\ge 1}.\] 
If \(j=5\) then $ka_4=\frac{d}{2}$ and $a_4$ divides $d$, that is our claim. Consider the case $j\le 3$. 

If $k=1$ then $a_4\ge a_j=d-a_4$, hence $2a_4\ge d=2a_5$ gives $a_4=a_5$ and $a_4$ divides $d$. If $k\ge 3$ then $a_0+\dots + a_5=2d=2a_5+ka_4+a_j$, hence $a_0+a_1+a_2+a_3=a_5+(k-1)a_4+a_j$; from $a_j\ge a_0$ we get $a_1+a_2+a_3\ge a_5+(k-1)a_4$, with $k-1\ge 2$, hence $a_1=\dots=a_5$, which is a contradiction with \cref{conditions}.(5) and $a_5>1$. 

We are left with the case $k=2$, i.e.\ \[d=2a_4+a_j,\hspace{3mm} j\le3.\] 
Observe that $a_j<a_4$, otherwise $a_4$ divides $d$. The integer $d$ is an even number by assumption, hence the equality above gives 
\begin{equation}\label{eq aj d a5}
a_j=2\alpha,\hspace{3mm}d=2(a_4+\alpha),\hspace{3mm}a_5=a_4+\alpha
\end{equation}
for some $\alpha\ge 1$. Furthermore, $a_0+\dots+a_4=2d-a_5=3a_4+3\alpha$, then $a_0+a_1+a_2+a_3=2a_4+3\alpha=2a_4+a_j+\alpha>a_1+a_2+a_3+\alpha$, as otherwise $a_2=a_3=a_4$ would give $a_4|d$ by \cref{conditions}.(5). We conclude that 
\begin{equation}\label{eq a0>alpha}
    a_0>\alpha.
\end{equation}

\begin{claim}
    In the above assumptions, $d=2a_4+a_j$ is not possible for $j=2,3$.
\end{claim}
\begin{proof}
Assume $j\in\{2,3\}$. We have $a_0+a_1+a_i=2d-a_5-a_4-a_j=2(2a_4+\alpha)-(a_4+\alpha)-a_4-2\alpha=2a_4+\alpha$, where $\{i,j\}=\{2,3\}$. This gives:
\begin{description}
    \item[$j=3$] using that $a_0+a_1+a_2\le 3a_3=6\alpha$, we get
    \begin{equation}\label{eq 2a4<=5alpha}
    2a_4\le 5\alpha;
    \end{equation}
    \item[$j=2$] using that $a_0+a_1+a_3< 2a_2+a_4=4\alpha+a_4$, we get
    \begin{equation}\label{eq a4<=4alpha}
    a_4 < 3\alpha.
    \end{equation}
\end{description}
The quasi-smoothness criterion on $I=\{j\}$ gives the following options.
\begin{enumerate}
    \item $d=ka_j$ for some $k>3$, as $ka_j=2a_4+a_j$ and $a_j<a_4$. Furthermore $(k-1)a_j=2a_4$ gives $a_4=(k-1)\alpha$, that combined with \eqref{eq 2a4<=5alpha} for $j=3$ and with \eqref{eq a4<=4alpha} for $j=2$ gives $k\le3$, which is a contradiction.
    \item $d=ka_j+a_i$ for some $i\neq j$ and $k\in \mathbf Z_{\ge 1}$, that combined with $d=2a_4+a_j$ gives $(k-1)a_j+a_i=2a_4$. \\
    If $i=5$ then $k\ge2$, as $k=1$ would give that $a_4$ divides $a_5$ hence $d$. We have $2(k-1)\alpha+(a_4+\alpha)=2a_4$, hence $a_4=(2k-1)\alpha$ that combined with \eqref{eq 2a4<=5alpha} for $j=3$ and with \eqref{eq a4<=4alpha} for $j=2$ gives $k\le 1$, which is a contradiction. \\
    If $i=4$ we have $(k-1)a_j+a_4=2a_4$, hence $(k-1)a_j=a_4$ and $k>2$ as $a_j<a_4$. Furthermore $a_4=2(k-1)\alpha$ that combined with \eqref{eq 2a4<=5alpha} for $j=3$ and with \eqref{eq a4<=4alpha} for $j=2$ gives $k\le 2$, which is a contradiction. \\
    If $i=3$ then $j=2$, and from $a_0+a_1+a_3=2a_4+\alpha=(k-1)a_2+a_3+\alpha=(2k-1)\alpha+a_3$ we get $(2k-1)\alpha=a_0+a_1\le 2a_2=4\alpha$. It follows $2k-1\le 4$ hence $k\le 2$, which is impossible as $2a_4=(k-1)a_2+a_3<ka_4$ hence $k>2$. \\
    If $i\le 2$ we have $2a_4=(k-1)a_j+a_i<ka_4$ hence $k>2$. We argue here differently for the case $j=2$ or 3. If $j=3$ we use that $a_0+a_1+a_2=2a_4+\alpha=(k-1)a_3+a_i+\alpha=(2k-1)\alpha+a_i$, that gives $(2k-1)\alpha=a_{i_1}+a_{i_2}\le 2a_3=4\alpha$, for $\{i,i_1,i_2\}=\{0,1,2\}$. It follows $2k-1\le 4$, hence $k\le 2$, which is a contradiction.
    For $j=2$ we have $i\le 1$ and $a_0+a_1+a_3=(2k-1)\alpha+a_i$ gives, for $j$ such that $\{i,j\}=\{0,1\}$, that $(2k-1)\alpha=a_j+a_3\le a_2+a_4=2\alpha+(k-1)\alpha+\frac{a_i}{2}\le (k+2)\alpha$, as $a_i\le a_2=2\alpha$. We conclude $k\le 3$ hence $k=3$. We obtain $2a_4=2a_2+a_i\le3a_2=6\alpha$, and from $a_j+a_3=5\alpha$ and $a_0\le a_2=2\alpha$ that $3\alpha\le a_3\le a_4$, hence $a_3=a_4=3\alpha$ and $\alpha$ divides $gcd(a_2,a_3,a_4,a_5)$, which means $\alpha=1$ by \cref{conditions}.(5). This gives $a_2=2$, $a_3=a_4=3$, $a_5=4$ and by \eqref{eq a0>alpha} we have $1=\alpha<a_0\le a_1\le a_2=2$ hence also $a_0=a_1=2$, which contradicts \cref{conditions}.(5). 
\end{enumerate}
We reached a contradiction with all possible options, hence the claim.
\end{proof}
\begin{claim}
   In the above assumptions, $d=2a_4+a_j$ is not possible for $j=0,1$.
\end{claim}
\begin{proof}
Assume $j=0,1$ and define $i$ to be such that $\{i,j\}=\{0,1\}$. Using \eqref{eq aj d a5} and $a_i+a_2+a_3=2d-a_5-a_4-a_j$ we get
\begin{equation}\label{eq ai+a2+a3}
    a_i+a_2+a_3=2a_4+\alpha.
\end{equation} 
As $a_i\ge a_0>\alpha$ by \eqref{eq a0>alpha}, we have $a_2<a_4$. 

The quasi-smoothness criterion on $I=\{j\}$ gives the following options.
\begin{enumerate}
\item $d=ka_j=2a_4+a_j$ for some $k\in\Z_{\ge 1}$, that gives $2a_4=(k-1)a_j=2(k-1)\alpha$ hence, using \eqref{eq aj d a5} and \eqref{eq ai+a2+a3}:
\begin{equation}\label{eq aj a4 a5 somma d}
    a_j=2\alpha,\hspace{3mm} a_4=(k-1)\alpha,\hspace{3mm} a_5=k\alpha,\hspace{3mm} a_i+a_2+a_3=(2k-1)\alpha,\hspace{3mm} d=2k\alpha.
\end{equation}
Note that $(k-1)\alpha=a_4>a_j=2\alpha$ gives $k>3$.  

We apply the quasi-smoothness criterion on $J=\{3,4\}$. 
\begin{enumerate}
    \item $d=na_3+ma_4$, for some $n,m\in\Z_{\ge 0}$. Observe that $2a_4+a_j=d=na_3+ma_4\ge na_j+ma_4$ gives $m\le 2$, otherwise $n=0$ hence $a_4$ divides $d$. 
    \begin{description}
        \item[$m=2$] this gives $na_3=a_j$, then $n=1$, $a_j=a_2=a_3$ and we conclude $\alpha=1$, hence $a_j=a_2=a_3=2$ and also $a_i=2$: this follows from $a_0=a_2=2$ when $j=0$, and from $1=\alpha<a_0\le a_1=2$ when $j=1$. This contradicts \cref{conditions}.(5).
        \item[$m=1$] this gives $na_3=a_4+a_j$, hence $n\ge 2$, and $na_3=(k+1)\alpha$. Using \eqref{eq aj a4 a5 somma d} we get $na_3+a_5=(2k+1)\alpha=a_0+a_1+a_2+a_3$, then $(n-1)a_3+a_5=a_0+a_1+a_2$ and $n\le 2$. We conclude $n=2$ and $2a_3=(k+1)\alpha$, hence $\frac{3}{2}(k+1)\alpha=3a_3\ge a_i+a_2+a_3=(2k-1)\alpha$ and $k\le 5$, that gives $k=4,5$. If $k=4$ then $2a_3=5\alpha$ implies that $\alpha$ is even, hence by \eqref{eq aj a4 a5 somma d} the same holds true for $a_j$, $a_4$, $a_5$ and (at least) one among $a_i$, $a_2$, $a_3$, which contradicts \cref{conditions}.(5). If $k=5$ then $a_i=a_2=a_3=3\alpha$, hence $\alpha=1$ and $gcd(a_i,a_2,a_3)=3\nmid 10=d$, which contradicts again \cref{conditions}.(5).
        \item[$m=0$] this gives $d=na_3=2a_4+a_j$, hence $n\ge 3$. Furthermore, using \eqref{eq aj a4 a5 somma d}: $na_3=2k\alpha=a_i+a_2+a_3+\alpha$ then $(n-1)a_3=a_i+a_2+\alpha<3\alpha$ (as $a_3\ge a_0>\alpha$) and we get $n<4$. We conclude $n=3$ and $3a_3=2k\alpha$, hence $a_3$ is even; as also $a_j$ and (at least) one among $a_4$, $a_5$ is even, by \cref{conditions}.(5) $a_2$ and $a_i$ are odd numbers, hence $a_i+a_2+a_3=(2k-1)\alpha$ is even, then $\alpha$ is even and both $a_4$ and $a_5$ are even number, which is again a contradiction with \cref{conditions}.(5).
    \end{description}
\item There exists $|J|=2$ indices $l\notin \{3,4\}$ such that $d-a_l=na_3+ma_4$, $n,m\in\Z_{\ge 0}$. The equation is satisfied for $l=j$ by taking $n=0$, $m=2$, hence the condition above is equivalent to $d-a_l=na_3+ma_4$ for some $l\in\{i,2,5\}$. 

If $l=5$ then $\frac{d}{2}=na_3+ma_4$, hence $d=2na_3+2ma_4$ and we are back to case (a) above. We are left with the case $l\le2$, $l\neq j$, i.e.\ $l\in\{i,2\}$. We define $l'$ to be such that $\{l,l'\}=\{i,2\}$. Note that under this notation the 4th equality in \eqref{eq aj a4 a5 somma d} reads as
\begin{equation}\label{eq al+al'+a3=2k-1 alpha}
    a_l+a_{l'}+a_3=(2k-1)\alpha.
\end{equation}
From $d=a_l+na_3+ma_4=2a_4+a_j$ we deduce $m\le 2$.
\begin{description}
    \item[$m=2$] in this case $n=0$, $a_l=a_j=2\alpha$ hence $\alpha=1$. From \eqref{eq aj a4 a5 somma d} we have that also one among $a_4$, $a_5$ is even, and from \eqref{eq al+al'+a3=2k-1 alpha} also one among $a_{l'}$ and $a_3$ is even. We obtain four even weights, which contradicts \cref{conditions}.(5).
    \item[$m=1$] in this case $2a_4+a_j=a_l+na_3+a_4$, i.e.\ $a_4+a_j=a_l+na_3=(k+1)\alpha$ and $n\ge 1$. Using \eqref{eq aj a4 a5 somma d} and \eqref{eq al+al'+a3=2k-1 alpha} we obtain  
    $a_{l'}+a_l+a_3+a_j=(2k+1)\alpha=a_l+na_3+a_5$, that gives $a_{l'}+a_j=(n-1)a_3+a_5$ and then $n=1$, $a_{l'}=(k-2)\alpha$ and $\alpha=1$. Then we have $a_l+a_3=k+1$, and from $k-2=a_{l'}\le a_3\le a_4=k-1$ we further obtain that either $(a_l,a_3)=(2,k-1)$ or $(a_l,a_3)=(3,k-2)$. 

    To conclude, we analyze what happens depending on $l$. If $(l,l')=(2,i)$ then $a_i=k-2\le a_2\le 3$, hence $k\le5$, i.e.\ $k=4$ or 5. If $k=4$ then the option $(a_2,a_3)=(3,k-2)=(3,2)$ 
    is not possible, hence $(a_i,a_j,a_2,a_3,a_4,a_5)=(2,2,2,3,3,4)$ which contradicts \cref{conditions}.(5). If $k=5$ then either $(a_i,a_j,a_2,a_3,a_4,a_5)=(3,2,2,4,4,5)$ or $(a_i,a_j,a_2,a_3,a_4,a_5)=(3,2,2,4,4,5)$ and $d=10$, both contradicting \cref{conditions}.(5).

    We are left with the case $(l,l')=(i,2)$ and 
    $(a_i,a_j,a_2,a_3,a_4,a_5)=(2,2,k-2,k-1,k-1,k)$ or $(a_i,a_j,a_2,a_3,a_4,a_5)=(3,2,k-2,k-2,k-1,k)$, hence by \cref{conditions}.(5) we get that $k$ is odd and $(a_0,a_1,a_2,a_3,a_4,a_5)=(2,3,k-2,k-2,k-1,k)$; furthermore if $k\equiv 2$ mod 3 then $3=gcd(2,k-2,k-2)$ does not divide $d=2k\equiv 1$ mod 3, contradicting \cref{conditions}.(5). We conclude $k\equiv 0$ mod 3 and $k>3$ odd, hence $k=6\beta+3$ for some $\beta\in\Z_{\ge 1}$. In order to exclude this last possibility, we apply the quasi smoothness criterion on $\{3\}$ and we get that $a_3|d$ or $a_3|d-a_j$ for some $j\neq 3$. If $a_3|d$ then $d=2(6\beta+3)=na_3=n(6\beta+1)$ for some $n\ge 1$, that gives $(12-6n)\beta=n-6$, which is impossible as the left term of the equality is positive for $n\le2$ and the right term for $n\ge 6$. For the remain case $a_3|d-a_j$ it is enough to write $d-a_j=na_3$ in terms of $\beta$ and $n$ as done before, writing appropriately $a_j$ accordingly with $(a_0,a_1,a_2,a_3,a_4,a_5)=(2,3,k-2,k-2,k-1,k)$; we leave to the reader to check that it gives a contradiction for any choice of $j\neq 3$, going exactly as the one done before.
    \item[$m=0$] in this case $2a_4+a_j=a_l+na_3$, that gives $n\ge 2$. Using \eqref{eq aj a4 a5 somma d} and \eqref{eq al+al'+a3=2k-1 alpha} we obtain $a_l+na_3=2k\alpha=a_{l'}+a_l+a_3+\alpha$ hence $(n-1)a_3=a_{l'}+\alpha$ and $n\le 2$. We conclude $n=2$ and $2a_4+a_j=a_l+2a_3$, hence $a_l$ is also even, together with $a_j$ and (at least) one among $a_4$ and $a_5$; \cref{conditions}.(5) implies that both $a_3$ and $a_{l'}$ are odd numbers, hence $a_l+a_{l'}+a_3=(2k-1)\alpha$ is even which means that $\alpha$ is even, that implies that both $a_4$ and $a_5$ are even numbers, which is impossible as it contradicts \cref{conditions}.(5).
\end{description}
\end{enumerate}
\item $d=a_h+ka_j=2a_4+a_j$, for some $h\neq j$ and $k\in\Z_{\ge1}$, that gives 
\begin{equation}\label{eq 2a4}
    2a_4=(k-1)a_j+a_h=2(k-1)\alpha+a_h
\end{equation}
If $h=5$ then using \eqref{eq aj d a5} we get $2a_4=2(k-1)\alpha+(a_4+\alpha)$, hence $a_4=(2k-1)\alpha$, $d=4k\alpha=2ka_j$ and $a_j|d$, i.e.\ we are back to case (1). If $h=4$ then using \eqref{eq aj d a5} we get $a_4=2(k-1)\alpha$, $d=2(2k-1)\alpha=(2k-1)a_j$ and $a_j|d$, i.e.\ we are again back to case (1). We arrived to $h\le 3$, i.e.\ $h\in\{i,2,3\}$, and from \eqref{eq 2a4} we have $k\ge 3$. Observe that \eqref{eq 2a4} gives that $a_h$ is an even number; as $a_j=2\alpha$ is also even, \cref{conditions}.(5) leads to a contradiction as soon as we find that (at least) two further weights are even numbers. We will use this argument many times in what follows.

We apply the quasi-smoothness criterion on $J=\{3,4\}$.
\begin{enumerate}
\item $d=na_3+ma_4$, for some $n,m\in\Z_{\ge0}$. Observe that $2a_4+a_j=d=na_3+ma_4\ge na_j+ma_4$ gives $m\le 2$, otherwise $n=0$ hence $a_4$ divides $d$. 
\begin{description}
    \item[$m=2$] in this case $d=na_3+2a_4=2a_4+a_j$ hence $na_3=a_j$, that gives $n=1$, $a_3=a_j$. If $h=3$, we obtain $d=a_3+ka_j=(k+1)a_j$, hence $a_j|d$ and we are back to the case (1) above. If $h\neq3$, i.e.\ $h\in \{i,2\}$, we call $h'$ the index such that $\{h,h'\}=\{i,2\}$; observe that we $a_3=a_j$ gives that also $a_3$ is an even number, hence we reach a contradiction as soon as we have a further even weight. If $\alpha$ is odd, then from \eqref{eq aj d a5} also one among $a_4$ and $a_5$ is indeed even, while if $\alpha$ is even, then by \eqref{eq ai+a2+a3} we have that  $a_h+a_{h'}+a_3$ is even. In any case, we have that this possibility leads to a contradiction.
    \item[$m=1$] in this case $d=na_3+a_4=2a_4+a_j$, hence $n\ge 2$ and $na_3=a_4+a_j=a_4+2\alpha$. From \eqref{eq ai+a2+a3} we have $na_3+a_4=2a_4+2\alpha=a_i+a_2+a_3+\alpha$, then $(n-1)a_3+a_4=a_i+a_2$ and $n<3$, that gives $n=2$. We then have $2a_3=a_4+a_j$, which implies that $a_4$ is even. Furthermore, if $\alpha$ is even then by \eqref{eq aj d a5} also $a_5$ is even, which is a contradiction, while if $\alpha$ is odd then from \eqref{eq ai+a2+a3} we have that $a_i+a_2+a_3$ is odd; as one of the weights in the sum is $a_h$, which is even, we conclude that one among the remaining ones in the sum is also even, which is a contradiction.
    \item[$m=0$] in this case $d=na_3=2a_4+a_j$, hence $n\ge 3$. Using \eqref{eq ai+a2+a3} we get $na_3=d=2a_4+2\alpha=a_i+a_2+a_3+\alpha$, hence $(n-1)a_3=a_i+a_2+\alpha<3a_3$ and $n<4$; we conclude $n=3$. We got $3a_3=d$, hence $a_3$ is an even number. If $h\neq 3$, we reach a contradiction, by arguing as in the case where $m=2$. If $h=3$ we have $d=3a_3=a_3+ka_j$, hence $2a_3=ka_j=2k\alpha$ and $a_3=k\alpha$, hence one among $k$ and $\alpha$ needs to be even, and $d=3k\alpha$. If $k=2k'$ is even, then $d=3k'a_j$, hence $a_j|d$ and we are back to case (1) above. If $k$ is odd then $\alpha=2\alpha'$ is even, and from \eqref{eq 2a4} we have $2a_4=(k-1)a_j+a_3=(3k-2)\alpha$, hence $a_j$, $a_3$, $a_4$ and (using \eqref{eq aj d a5}) $a_5$ are multiple of $\alpha'$, which gives $\alpha'=1$ by \cref{conditions}.(5) and $(a_j,a_3,a_4,a_5)=(4,2k,3k-2,3k)$. Hence by \eqref{eq ai+a2+a3} we have $a_i+a_2=2a_4+\alpha-a_3=4k-2$, and using $a_2\le a_3=2k$, hence $(a_i,a_2)=(2k-2,2k)$ or $(a_i,a_2)=(2k-1,2k-1)$. The first possibility would give that $a_i$ and $a_2$ are even numbers, which leads to a contradiction; for the second option, observe that if $(i,j)=(0,1)$ then it gives $2k-1\le 4$, which contradicts $k\ge 3$. We conclude that 
    \[(a_0,a_1,a_2,a_3,a_4,a_5,d)=(4,2k-1,2k-1,2k,3k-2,3k,6k).\]
    To exclude this last case, we apply the quasi-smoothness criterion on $\{1\}$. If $a_1|d$ then there exists $\beta\ge 1$ such that $6k=\beta(2k-1)$, i.e.\ $(2\beta-6)k=\beta$; this gives that $\beta$ is even, that $\beta\ge 3$ (as $k\ge 3$), and $2\beta-6<\beta$, i.e.\ $\beta<6$; we conclude $\beta=4$ hence $k=2$, which is not impossible. If $a_1|d-a_j$ for some $j\neq 1$, we argue case by case exactly as in the case before, using the explicit expressions of $a_j$ given above; we leave to the reader the check that this leads to a contradiction in all cases. 
\end{description}
\item There exist $|J|=2$ indices $l\notin \{3,4\}$ such that $d-a_l=na_3+ma_4$, $n,m\in\Z_{\ge 0}$. The equation is satisfied when $l = j$ by setting $n = 0$ and $m = 2$. Hence, the condition above is equivalent to $d - a_l = n a_3 + m a_4$ for some $l \in {i, 2, 5}$.

If $l=5$, then $\frac{d}{2}=na_3+ma_4$, hence $d=2na_3+2ma_4$ and we are back to case (a) above. We are left with the case $l\le2$, $l\neq j$, that is, $l\in\{i,2\}$. We define $l'$ as the number satisfying ${l, l'} = {i, 2}$. Note that under this notation the 4th equality in \eqref{eq ai+a2+a3} reads as
\begin{equation}\label{eq al+al'+a3=2a4+alpha}
    a_l+a_{l'}+a_3=2a_4+\alpha.
\end{equation}
Moreover, we recall that $h\in\{i,2,3\}=\{l,l',3\}$ and $a_h$, $a_j$ are even numbers.  From $d=a_l+na_3+ma_4=2a_4+a_j$ we deduce $m\le 2$.
\begin{description}
    \item[$m=2$] in this case $d=a_l+na_3=2a_4=2a_4+a_j$, that gives $n=0$ and $a_j=a_l$ and $a_l$ is also even. We analyze the two cases $h=l$ and $h\neq l$. 
    
    If $h\neq l$, i.e.\ $h\in\{l',3\}$, this gives that $a_l$, $a_h$, $a_j$ are three even weights; if $\alpha$ is odd then also one among $a_4$, $a_5$ is even, which is a contradiction, while if $\alpha$ is even then $a_l+a_{l'}+a_3=2a_4+\alpha$ gives that $a_l$, $a_{l'}$, $a_3$ are all even, which is again a contradiction together with $a_j$ even. 
    
    If $h=l$ we still have that $a_l$, $a_j$ are even weights, and moreover $d=a_l+2a_4=a_l+ka_j$ gives $2a_4=ka_i$, hence $a_4=k\alpha$ and by \eqref{eq aj d a5} that $a_5=(k+1)\alpha$; as $a_l=a_j=2\alpha$, this implies $\alpha=1$ by \cref{conditions}.(5), and then $a_4=k$, $a_5=k+1$, meaning that one among $a_4$ and $a_5$ is even. Furthermore, from \eqref{eq al+al'+a3=2a4+alpha} we deduce that $a_{l'}+a_3=2k-1$, with $a_{l'}\le a_3\le a_4=k$; we conclude that $(a_{l'},a_3)=(k-1,k)$, hence one among them is even, which is a contradiction.
    \item[$m=1$] in this case $d=a_l+na_3+a_4=2a_4+a_j$, which gives $n\ge1$; using \eqref{eq al+al'+a3=2a4+alpha} we have $2a_4+a_j=2a_4+2\alpha=a_l+a_{l'}+a_3+\alpha$, hence $(n-1)a_3+a_4=a_{l'}+\alpha<a_3+a_4$, (as $a_3>\alpha$) and $n<2$. We conclude $n=1$ and $a_4=a_{l'}+\alpha$. Moreover, using \eqref{eq 2a4}, we obtain $2a_4=2a_{l'}+2\alpha=2(k-1)\alpha+a_h$, whence $2a_{l'}=2(k-2)\alpha+a_h$. We analyze the two cases $h=l'$ and $h\neq l'$.

    If $h=l'$ we obtain $a_{l'}=2(k-2)\alpha$, $a_4=(2k-3)\alpha$, $a_5=(2k-2)\alpha$ that together with $a_j=2\alpha$ imply $\alpha=1$ by \cref{conditions}.(5). As $a_{l'}$, $a_5$, $a_j$ are even, we get a contradiction from \eqref{eq al+al'+a3=2a4+alpha}, as it gives that also one element between $a_3$ and $a_l$ is even.

    If $h\neq l'$, i.e.\ $h\in\{l,3\}$, we call $h'$ the index such that $\{h,h'\}=\{l,3\}$; with this notation \eqref{eq al+al'+a3=2a4+alpha} reads $a_{l'}+a_h+a_{h'}=2a_4+\alpha$, and $d=a_h+ka_j=a_l+a_3+a_4=a_4+a_h+a_h'$ gives $a_4=2k\alpha-a_{h'}$. Using these equalities, we can write every weight in terms of $a_{h'}$ and $\alpha$ as follows:
    \begin{align*}
        a_j&=2\alpha \\
        a_{l'}&=(2k-1)\alpha-a_{h'} \\
        a_h&=(2k+2)\alpha-2a_{h'} \\
        a_{h'}&=a_{h'} \\
        a_4&=2k\alpha-a_{h'}\\
        a_5&=(2k+1)\alpha-a_{h'},
    \end{align*}
    where we remind to the reader that $\{i,j\}=\{0,1\}$, $\{l,l'\}=\{i,2\}$,  $\{h,h'\}=\{l,3\}$. From $a_{h'}, a_h\le a_4$ we obtain \[2\alpha\le a_{h'}\le k\alpha.\]
    We are in one of the following cases.
    \begin{itemize}
        \item $(i,j)=(0,1)$, i.e.\ either $a_j\ge a_l$, if $l=i$, or $a_j\ge a_{l'}$, if $l'=i$. Let $l=i$, hence $l'=2$ and $\{h,h'\}=\{0,3\}$; if $h'=i=0$ then $2\alpha\le a_{h'}\le a_j=2\alpha$ and $a_{h'}=2\alpha$, that gives that all weights are multiple of $\alpha$ hence $\alpha=1$, $a_{h'=2}$ and $a_j$, $a_{h'}$, $a_h$, $a_4$ are all even, which is a contradiction; if $h=i=0$ then $a_h\le a_j$ gives $a_{h'}\ge k\alpha$, then $a_{h'}=k\alpha$ and again all weights are multiple of $\alpha$, hence $\alpha=1$, $(a_{l'},a_{h'},a_4,a_5)=\bigl((k-1)\alpha,k\alpha,k\alpha,(k+1)\alpha\bigl)$ and two weights among them are even, which is a contradiction. 
        \item $(i,j)=(1,0)$: in this case $a_0=2\alpha$, hence all other weights are greater or equal than $2\alpha$. If $\alpha$ is odd, then if $a_{h'}$ is even we have that $a_0$, $a_{h'}$, $a_h$, $a_4$ are even, if $a_{h'}$ is odd we have that $a_0$, $a_h$, $a_{l'}$, $a_5$ are even, which in any case give a contradiction. If $\alpha$ and $a_{h'}$ are even, then all weights are even, which is again a contradiction. We are left with the case $\alpha$ even and $a_{h'}$ odd; observe that in this case all weights but $a_0$, $a_h$ are odd, hence strictly bigger than $2\alpha$. To exclude this last case, we apply the quasi-smoothness criterion in \cref{prop quasismooth} on $\{l'\}$, which gives one of the following options.
        \begin{enumerate}
            \item $a_{l'}|d=2a_5$, hence $a_{l'}|a_5$ as $a_{l'}$ is odd; from $a_{l'}|a_5=a_{l'}+2\alpha$ we obtain $a_{l'}|2\alpha$, which is impossible as $a_{l'}>2\alpha$.
            \item $a_{l'}|d-a_s$ for some $s\neq l'$; we analyze case by case for all possible values of $s$, using the description of the weights in terms of $\alpha,a_{h'}$ given above.

            If $s=5$ we obtain $a_{l'}|a_5$, impossible as we are back in case (i). 

            If $s=4$ then $a_{l'}|d-a_4=2a_5-a_4=2(a_4+\alpha)-a_4=a_4+2\alpha=a_{l'}+3\alpha$, hence $a_{l'}|3\alpha$. Using that $a_{l'}$ is odd and $\alpha$ is even, we deduce that $a_{l'}|\frac{3}{2}\alpha$, which gives the contradiction $2\alpha<a_{l'}\le\frac{3}{2}\alpha$.

            If $s=h'$ then $a_{l'}|d-a_{h'}=2a_5-a_{h'}=2(a_{l'}+2\alpha)-a_{h'}=2a_{l'}+4\alpha-a_{h'}$. Observe that  from $a_{h'}>2\alpha$ we have $4\alpha-a_{h'}<2\alpha<a_{l'}$, hence $a_{l'}|2a_{l'}+4\alpha-a_{h'}<3a_{l'}$ gives $2a_{l'}+4\alpha-a_{h'}\in\{a_{l'},2a_{l'}\}$. If $2a_{l'}+4\alpha-a_{h'}=2a_{l'}$ then $a_{h'}=4\alpha$ and all weights are multiple of $\alpha$, which is a contradiction. If $2a_{l'}+4\alpha-a_{h'}=a_{l'}$ then $4\alpha-a_{h'}=-a_{l'}$, i.e.\ $a_{h'}=a_{l'}+4\alpha=(2k+3)\alpha-a_{h'}$ which gives the contradiction $(2k+3)\alpha=2a_{h'}<2k\alpha$.

            If $s=h$ then $a_{l'}|d-a_h=2a_5-a_h=2k\alpha$, hence $a_{l'}|k\frac{\alpha}{2}$ as $a_{l'}$ is odd, $\alpha$ is even. In particular $a_{l'}\le k\frac{\alpha}{2}$, i.e.\ $2(2k-1)\alpha-2a_{h'}\le k\alpha$ and $(3k-2)\alpha\le 2a_{h'}\le 2k\alpha$, giving $k\le 2$ which is a contradiction.

            Finally, if $s=0$ then $a_{l'}|d-a_0=2a_5-2\alpha=2(a_{l'}+\alpha)$ hence $a_{l'}|2\alpha$ which is impossible as $a_{l'}>2\alpha$.
        \end{enumerate}
    \end{itemize}
    \item[$m=0$] in this case $d=a_l+na_3=2a_4+a_j$, which gives $n\ge 2$; using \eqref{eq al+al'+a3=2a4+alpha} we have $2a_4+a_j=2a_4+2\alpha=a_l+a_{l'}+a_3+\alpha$, hence $(n-1)a_3=a_{l'}+\alpha< 2a_3$ (as $a_3>\alpha$), which gives $n\le2$. We conclude $n=2$, i.e.\ $a_3=a_{l'}+\alpha$ and, using again \eqref{eq al+al'+a3=2a4+alpha}, that $a_l+2a_{l'}=2a_4$, hence $a_l$ is even. Observe that we conclude as soon as we have that $\alpha$ is odd: this would imply that one among $a_4$ and $a_5$ is even because of \eqref{eq aj d a5}; if $h\neq l$ this is enough to conclude because also $a_l$ is even if $h=l$ then we conclude because one among $a_3$ and $a_{l'}$ is even by $a_3=a_{l'}+\alpha$. Assume $h\neq l$; we have that $a_l+a_{l'}+a_3=2a_4+\alpha$, where $a_l$ is even and one among $a_3$, $a_{l'}$ is $a_h$, hence also even; if $\alpha$ is even then $a_l$, $a_{l'}$, $a_3$, $a_j$ are all even and we get a contradiction, hence $\alpha$ is odd and we conclude as said. Finally, assume $h=l$; in this case $d=a_h+ka_j=a_l+2a_3$ gives $2a_3=ka_j=2k\alpha$, hence $a_3=k\alpha$ and $a_{l'}=(k-1)\alpha$. From \cref{conditions}.(5) we obtain that $\alpha|d=2a_4+2\alpha$, hence $\alpha|2a_4$ and either $a|a_4$, then $\alpha=1$ again by \cref{conditions}.(5), i.e.\ $\alpha$ odd, or $\alpha|2$. In this last case $\alpha=1$ or 2; the odd case is impossible, hence $\alpha=2$ and $a_j$, $a_{l}$, $a_{l'}$, $a_3$ are all even, which is again a contradiction.
\end{description}
\end{enumerate}
\end{enumerate}
We arrived at a contradiction with all possible options, hence the claim.
\end{proof}
We conclude that $d=ka_4+a_j$ is not verified, hence $a_4$ divides $d$. 
\end{proof}

\begin{proof}[Proof of \cref{thm ai+a5=d}]
By \cref{prop d<=2a5} we have that $d\le 2a_5$.
If \(d = 2a_5\), then \(a_4 + a_5 = d\) by \cref{prop a4=a5=d/2,prop a4 divide d}, hence the claim for $i=4$. If \(d < 2a_5\), then there exists \(i \in \{0, \dots, 4\}\) such that \(a_i + a_5 = d\): the polynomial \(f_d\) defining \(X_d\) must contain a monomial \(x_i x_5\) with \(i < 5\); otherwise, the punctured affine cone \(C_{X_d}^*\) over \(X_d\) would contain the point \((0, \dots, 0, 1) \in \A^6\) and would be singular at that point. We conclude that $a_i+a_5=d$.   
\end{proof}

\subsection{Finiteness of families of FK3 fourfolds} The main consequence of \cref{thm ai+a5=d} is the finiteness result \cref{thm list complete}. We start with some preliminary facts.

\begin{lemma}\label{lemma the K3 is quasi smooth}
    Let $X_d\subseteq\P(a_0,\dots,a_5)$ be a FK3 weighted fourfold such that $a_5>1$ and \(\dim X_d^{\mathrm{sing}}\le 1\), and call $i\in\{1,\dots,5\}$ an index such that $a_i+a_5=d$, given by \cref{thm ai+a5=d}. Then the general weighted surface \(S_d\subseteq\P(a_0,\dots,\hat a_i,\dots,a_4)\) is quasi-smooth and it is not a linear cone.
\end{lemma}
\begin{proof}
To simplify our notation, we assume, without loss of generality, that $i=4$ (hence $a_4+a_5=d$). To prove that statement, we apply the criterion in \cref{prop quasismooth}, and we refer to the notation therein in what follows. Any nonempty subset $I\subseteq\{0,\dots,3\}=:A$ satisfies (1) or (2) in \cref{prop quasismooth} as a subset of $\{0,\dots,5\}=:A'$, as $X_d$ is a quasi-smooth and not a linear cone. If $I$ satisfies (1) as a subset of $A'$, then the same equation gives (1) as a subset of $A$, hence $I\subseteq A$ satisfies the criterion. If $I$ satisfies (2) as a subset of $A'$, then there are two options: either no $j$ in the claim equals 4 or 5, hence $I$ satisfies (2) also as a subset of $A$, or at least one of the following equations holds:
\begin{enumerate}[(a)]
    \item $d-a_4=a_5=\sum_{i\in I}n_i a_i$
    \item $d-a_5=a_4=\sum_{i\in I}n_i a_i$
\end{enumerate}
We define $J'\subset A'$ to be the set $\{4,5\}$ if both (a) and (b) hold, and to be the set $\{k\}$ if only one of them holds, with $k=4$ or 5 if (b) or (a) holds respectively. Given $J:=J'\cup I\subset A'$, from the quasi-smoothness of $X_d$, we have again that one of the following options holds:
\begin{enumerate}
    \item $d=\sum_{k\in J'} m_k a_k+\sum_{i\in I}m_ia_i=\sum_{i\in I}m'_ia_i$, where the $m_i'$s are the coefficients obtained combining the equation with (a), (b) or both (a) and (b) above, depending on the definition of $J'$; we get that $I\subseteq A$ satisfies (1);
    \item there exist at least $|J|$ indices $j\notin J$ (hence $j\notin I$) such that $d-a_j=\sum_{k\in J'} m_k a_k+\sum_{i\in I}m_ia_i=\sum_{i\in I}m'_ia_i$, where again the $m_i'$s are the coefficients obtained combining the equation with (a), (b) or both (a) and (b) above, depending on the definition of $J$. Among these indices $j$, there are at least $|I|$ indices different from 4 and 5: if $|J|=|I|+1$ then there is at most one index $j=4,5\notin J$, hence all the remaining $|I|$ indices satisfy the claim, while if $|J|=|I|+2$ then $4,5\in J$, hence $j\neq 4,5$.  We obtain that $I\subseteq A$ satisfies (2).
\end{enumerate}
In any case, we conclude that $S_d$ is quasi-smooth and not a linear cone. 
\end{proof}
\begin{lemma}\label{lemma the K3 is well-formed}
Let $X_d\subseteq\P(a_0,\dots,a_5)$ be a FK3 weighted fourfold such that $a_5>1$ and \(\dim X_d^{\mathrm{sing}}\le 1\), and call $i\in\{1,\dots,5\}$ an index such that $a_i+a_5=d$, given by \cref{thm ai+a5=d}. Then the general weighted surface
\(
S_d\subseteq \mathbf{P}(a_0, \dots, \hat{a}_i, \dots, a_4)
\)
is well-formed.
\end{lemma}

\begin{proof}
We need to prove that the following conditions are satisfied:
\begin{itemize}
\item[1.] \(\gcd(a_0, \dots, \hat{a}_i, \dots, \hat{a}_j, \dots, a_4) = 1 \quad \forall\ i,j\),
\item[2.] \(\gcd(a_0, \dots, \hat{a}_i, \dots, \hat{a}_j, \dots, \hat{a}_k, \dots, a_4) \mid d \quad \forall\ i,j,k\).
\end{itemize}

Since we are assuming that \(\codim_{X_d} X_d^{\mathrm{sing}} \ge 3\), the following holds:
\[
\gcd(a_0, \dots, \hat{a}_i, \dots, \hat{a}_j, \dots, a_5) = 1 \quad \forall\ i,j \quad \text{and} \quad
\gcd(a_0, \dots, \hat{a}_i, \dots, \hat{a}_j, \dots, \hat{a}_k, \dots, a_5) \mid d \quad \forall\ i,j,k.
\]
Thus, we know that \(\gcd(a_0, \dots, \hat{a}_i, \dots, \hat{a}_j, \dots, \hat{a}_5) \mid d\). Moreover, by assumption, \(d = a_5 + a_i\), and hence \(\sum_{l \neq i, 5} a_l = d\).

It follows that
\[
\gcd(a_0, \dots, \hat{a}_i, \dots, \hat{a}_j, \dots, \hat{a}_5) \mid \sum_{l \neq i, 5} a_l,
\]
which implies
\begin{equation}\label{eq_a}
\sum_{l \neq i, 5} a_l = \gcd(a_0, \dots, \hat{a}_i, \dots, \hat{a}_j, \dots, \hat{a}_5) \cdot k
\end{equation}
for some integer \(k \in \mathbf{Z}_{\geq 1}\).

By construction, for all \(h \neq i, j, 5\), we have \(a_h = \gcd(a_0, \dots, \hat{a}_i, \dots, \hat{a}_j, \dots, \hat{a}_5) \cdot \alpha_h\). Then by \eqref{eq_a}, we also get
\[
a_j = \gcd(a_0, \dots, \hat{a}_i, \dots, \hat{a}_j, \dots, \hat{a}_5) \cdot \left(k - \sum_{h \neq i, j, 5} \alpha_h\right).
\]
This means that
\[
\gcd(a_0, \dots, \hat{a}_i, \dots, \hat{a}_j, \dots, \hat{a}_5) \mid a_j,
\]
and so
\[
\gcd(a_0, \dots, \hat{a}_i, \dots, \hat{a}_j, \dots, a_4) = \gcd(a_0, \dots, \hat{a}_i, \dots, \hat{a}_j, \dots, \hat{a}_5) = \gcd(a_0, \dots, \hat{a}_i, \dots, \hat{a}_5) = 1,
\]
where the last equality follows from our assumptions, and this completes the proof of item 1.

To prove item 2, we use the quasi-smoothness of the surface, which follows from \cref{lemma the K3 is quasi smooth}.

To simplify the notation, let \(\{s, t\}\) be the set of indices \(\{0, \dots, \hat{i}, \dots, \hat{j}, \dots, \hat{k}, \dots, 4\}\). We want to prove that \(\gcd(a_s, a_t) \mid d\). From \cref{lemma the K3 is quasi smooth}, we know that either condition (1) or (2) in \cref{prop quasismooth} is satisfied.

If (1) holds, then for \(I = \{s,t\}\), we have \(d = n_s a_s + n_t a_t\) for some \(n_s, n_t \in \mathbf{N}\), hence \(\gcd(a_s, a_t) \mid d\).

If (2) holds, then there exist two indices \(r, l\) different from \(s\) and \(t\), such that \(\{s, t, r, l\} = \{0, \dots, \hat{i}, \dots, 4\}\), and
\[
d - a_r = n_s a_s + n_t a_t, \quad d - a_l = m_s a_s + m_t a_t
\]
for some \(n_s, n_t, m_s, m_t \in \mathbf{N}\). Then \(\gcd(a_s, a_t) \mid (d - a_r)\) and \(\gcd(a_s, a_t) \mid (d - a_l)\), so
\[
\gcd(a_s, a_t) \mid (2d - a_r - a_l - a_s - a_t).
\]
Note that by construction \(a_r + a_l + a_s + a_t = a_0 + \dots + \hat{a}_i + \dots + a_4 = d\), hence \(\gcd(a_s, a_t) \mid d\), as required. This concludes the proof.
\end{proof}

\begin{corollary}\label{cor the k3 is a k3}
Let $X_d\subseteq\P(a_0,\dots,a_5)$ be a FK3 weighted fourfold such that $a_5>1$ and \(\dim X_d^{\mathrm{sing}}\le 1\), and call $i\in\{1,\dots,5\}$ an index such that $a_i+a_5=d$, given by \cref{thm ai+a5=d}. Then the general weighted surface
\(
S_d \subseteq \mathbf{P}(a_0, \dots, \hat{a}_i, \dots, a_4)
\)
is a quasi-smooth K3 surface that is not a linear cone.
\end{corollary}

\begin{proof}
The variety \( S_d \) is a well-formed, quasi-smooth weighted surface that is not a linear cone, by \cref{lemma the K3 is well-formed,lemma the K3 is quasi smooth}. As a consequence, we can apply the adjunction formula \eqref{eq adjuction} which gives
\[
\OO(K_{S_d}) \cong \mathcal{O}_{S_d} \bigl( \deg(S_d) - (a_0 + \dots + \hat{a}_i + \dots + a_4) \bigr) = \mathcal{O}_{S_d},
\]
since by hypothesis \( d - (a_0 + \dots + \hat{a}_i + \dots + a_4) = d - (2d - a_5 - a_i) = 0 \), because of \cref{conditions}.(3), hence it is a K3 surface.
\end{proof}

\begin{theorem}\label{thm list complete} 
There exist finitely many families of FK3 weighted fourfolds $X_d\subseteq\P(a_0, \ldots, a_5)$ with singular locus of dimension at most 1.
\end{theorem}

\begin{proof}
Let \(X_d\) be as in the statement. If \(a_0 = \dots = a_5 = 1\), then \(d = 3\) and \(X_d\) is a smooth cubic hypersurface in \(\P^5\). If \(a_5 > 1\), then we are in the hypotheses of \cref{cor the k3 is a k3} and to $X_d$ we can associate a weighted hypersurface $S_d$ that is a K3 surface. The claim is therefore a consequence of the finiteness of quasi-smooth codimension 1 weighted K3 surfaces proved in \cite[Section 4]{reid}: From any such K3 surface  \(S_d\), a FK3 as in \cref{cor the k3 is a k3} is constructed from a partition $a_i+a_5=d$, and there are finitely many of these. 
\end{proof}

\begin{remark} \label{rmk:unbounded}
The dimension hypothesis on the singular locus is necessary to have the finiteness result: if omitted, our setting would include the infinitely many examples $X_{2^d}\subseteq \P(2^{d-2},2^{d-2},2^{d-2},2^{d-2},2^{d-1}-e,2^{d+1}+e)$, with $d>>0$ and $e>0$, provided in \cite[Remark A.12]{LPZ18}.

Note that this hypothesis is weaker than having terminal singularities (\cite[Corollary 5.18]{KM98}), which would guarantee the finiteness of families \cite[Theorem 1.1]{Birkar21}; in fact, not all FK3 weighted fourfolds $X_d\subseteq\P(a_0, \ldots, a_5)$ with $\dim X_d^{\mathrm{sing}}\le 1$ have terminal singularities (see \cref{table sigularities}).
\end{remark}

\section{The examples} 
\label{section: the examples}

\subsection{The list of examples}\label{section list}
In \cref{table fourfolds} we list all families of FK3 weighted fourfolds \(X_d\subseteq \P(a_0,...,a_5)\) with \(\dim X_d^{\mathrm{sing}}\le 1\). The list is finite because of \cref{thm list complete}, and is produced by hand, as suggested by the proof of the theorem:
\begin{enumerate}
    \item consider the finite list of all possible quasi-smooth codimension 1 weighted K3 surfaces  given by \cite[Section 4]{reid} (see also \cite[Table 1]{Fletcher00} for the full list);
    \item for each such K3 surface $S_d\subseteq \P(a_0,a_1,a_2,a_3)$, take any partition $s+t=d$;
    \item control if the general $X_d\subseteq\P(a_0,a_1,a_2,a_3,s,t)$ is a FK3 weighted fourfold with \(\dim X_d^{\mathrm{sing}}\le 1\), and in case of a positive answer, add it to the list.
\end{enumerate}
The Hodge number $h^{2,2}$ of each example is computed using \cref{thm Hodge structure of Xd}.
\tiny
\renewcommand{\arraystretch}{1.3}
\rowcolors{2}{gray!10}{white}
\begin{longtable}{c||c|c|c} 
\caption{Families of FK3 weighted fourfolds $X_d\subseteq \P(\underline a)$ with \(\dim X_d^{\mathrm{sing}}\le 1\). They are listed in lexicographic order as the degree increases.
} 
\label{table fourfolds} \\
\rowcolor{white}
 \(\#\)
 & $(\underline a)$
 & $d$
 & $h^{2,2}$
 \\ \hline \hline \endfirsthead
 \rowcolor{white}
  \(\#\)
 & $(\underline a)$
 & $d$
 & $h^{2,2}$
 \\ \hline \hline \endhead 
 1
 & $(1^6)$
 & 3
 & 21
 \\ \hline
 2
 & $(1^5,3) $
 &  4
 & 20
 \\ \hline
 3
 & $(1^4,2^2)$
 & 4
 & 20
 \\ \hline
 4
 & $(1^4,2,4)$
 & 5
 & 19
 \\ \hline
 5
 & $(1^3,2^2,3)$
 & 5
 & 19
 \\ \hline
 6
 & $(1^4,3,5) $
 & 6
 & 20
 \\ \hline
 7
 & $(1^3,2^2,5)$ 
 & 6
 & 17
 \\ \hline
 8
 & $(1^3,2,3,4)$ 
 & 6
 & 20
 \\ \hline
 9
 & $(1^3,3^3)$
 & 6
 & 20
 \\ \hline
 10
 & $(1^2,2^2,3^2)$
 & 6
 & 17
 \\ \hline
 11
 & $(1^3,2,3,6)$
 & 7
 & 17
 \\ \hline
 12
 & $(1^2,2^2,3,5)$
 & 7
 & 17
 \\ \hline
 13
 & $(1^2,2,3^2,4)$
 & 7
 & 17
 \\ \hline
 14
 & $(1^3,2,4,7)$
 & 8
 & 18
 \\ \hline
 15
 & $(1^2,2^2,3,7)$
 & 8
 & 14
 \\ \hline
 16
 & $(1^2,2,3,4,5)$
 & 8
 & 18
 \\ \hline
 17
 & $(1,2^2,3^2,5)$  
 & 8
 & 14
 \\ \hline
 18
 & $(1^3,3,4,8)$
 & 9
 & 17
 \\ \hline
 19
 & $(1^2,2,3^2,8)$
 & 9
 & 13
 \\ \hline
 20
 & $(1^2,2,3,4,7)$
 & 9
 & 17
 \\ \hline
 21
 & $(1^2,3^2,4,6)$
 & 9
 & 17
 \\ \hline
 22
 & $(1^2,3,4^2,5)$
 & 9
 & 17
 \\ \hline
 23
 & $(1,2,3^2,4,5)$
 & 9
 & 13
 \\ \hline
 24
 & $(1,2^2,3^2,7)$
 & 9
 & 13
 \\ \hline
 25
 & $(1^3,3,5,9)$
 & 10
 & 18
 \\ \hline
 26
 & $(1^2,2^2,5,9)$
 & 10
 & 15
 \\ \hline
 27
 & $(1^2,2,3,4,9)$
 & 10
 & 13
 \\ \hline
 28
 & $(1^2,2,3,5,8)$
 & 10
 & 18
 \\ \hline
 29
 & $(1^2,3^2,5,7)$
 & 10
 & 18
 \\ \hline
 30
 & $(1^2,3,5^3)$
 & 10
 & 18
 \\ \hline
 31
 & $(1,2^2,3,5,7)$
 & 10
 & 15
 \\ \hline
 32
 & $(1,2^2,5^3)$
 & 10
 & 15
 \\ \hline
 33
 & $(1,2,3^2,4,7)$
 & 10
 & 13
 \\ \hline
 34
 & $(1,2,3,4,5^2)$
 & 10
 & 13
 \\ \hline
 35
 & $(1^2,2,3,5,10)$
 & 11
 & 13
 \\ \hline
 36
 & $(1,2^2,3,5,9)$
 & 11
 & 13
 \\ \hline
 37
 & $(1,2,3^2,5,8)$
 & 11
 & 13
 \\ \hline
 38
 & $(1,2,3,4,5,7)$
 & 11
 & 13
 \\ \hline
 39
 & $(1,2,3,5^2,6)$
 & 11
 & 13
 \\ \hline
 40
 & $(1^3,4,6,11)$
 & 12
 & 19
 \\ \hline
 41
 & $(1^2,2,3,6,11)$
 & 12
 & 14
 \\ \hline
 42
 & $(1^2,2,4,5,11)$
 & 12
 & 13
 \\ \hline
 43
 & $(1^2,3,4^2,11)$
 & 12
 & 11
 \\ \hline
 44
 & $(1^2,3,4,6,9)$
 & 12
 & 19
 \\ \hline
 45
 & $(1,2^2,3,5,11)$
 & 12
 & 10
 \\ \hline
 46
 & $(1,2,3^2,4,11)$
 & 12
 & 9
 \\ \hline
 47
 & $(1,2,3,4,5,9)$
 & 12
 & 13
 \\ \hline
 48
 & $(1,2,3,5,6,7)$
 & 12
 & 14
 \\ \hline
 49
 & $(1,2,4,5^2,7)$
 & 12
 & 13
 \\ \hline
 50
 & $(1,3^2,4^2,9)$
 & 12
 & 11
 \\ \hline
 51
 & $(2^2,3^2,5,9)$
 & 12
 & 10
 \\ \hline
 52
 & $(2^2,3,5^2,7)$
 & 12
 & 10
 \\ \hline
 53
 & $(2,3^2,4,5,7)$
 & 12
 & 9
 \\ \hline
 54
 & $(1^2,3,4,5,12)$
 & 13
 & 11
 \\ \hline
 55
 & $(1,2,3,4,5,11)$
 & 13
 & 11
 \\ \hline
 56
 & $(1,3^2,4,5,10)$
 & 13
 & 11
 \\ \hline
 57
 & $(1,3,4^2,5,9)$
 & 13
 & 11
 \\ \hline
 58
 & $(1,3,4,5^2,8)$
 & 13
 & 11
 \\ \hline
 59
 & $(1,3,4,5,6,7)$
 & 13
 & 11
 \\ \hline
 60
 & $(1^2,2,4,7,13)$
 & 14
 & 14
 \\ \hline
 61
 & $(1,2^2,3,7,13)$
 & 14
 & 11
 \\ \hline
 62
 & $(1,2,3,4,5,13)$
 & 14
 & 8
 \\ \hline
 63
 & $(1,2,3,4,7,11)$
 & 14
 & 14
 \\ \hline
 64
 & $(1,2,4,5,7,9)$
 & 14
 & 14
 \\ \hline
 65
 & $(1,2,4,7^3)$
 & 14
 & 14
 \\ \hline
 66
 & $(2^2,3^2,7,11)$
 & 14
 & 11
 \\ \hline
 67
 & $(2^2,3,7^3)$
 & 14
 & 11
 \\ \hline
 68
 & $(2,3^2,4,5,11)$
 & 14
 & 8
 \\ \hline
 69
 & $(2,3,4,5^2,9)$
 & 14
 & 8
 \\ \hline
 70 
 & $(2,3,4,5,7^2)$
 & 14
 & 6
 \\ \hline 
 71
 & $(1^2,2,5,7,14)$
 & 15
 & 13
 \\ \hline
 72
 & $(1^2,3,4,7,14)$
 & 15
 & 11
 \\ \hline
 73
 & $(1^2,3,5,6,14)$
 & 15
 & 11
 \\ \hline
 74
 & $(1,2^2,5,7,13)$
 & 15
 & 13
 \\ \hline
 75
 & $(1,2,3,4,7,13)$
 & 15
 & 11
 \\ \hline
 76
 & $(1,2,3,5^2,14)$
 & 15
 & 7
 \\ \hline
 77 
 & $(1,2,3,5,6,13)$
 & 15
 & 11
 \\ \hline
 78
 & $(1,2,3,5,7,12)$
 & 15
 & 13
 \\ \hline
 79
 & $(1,2,5^2,7,10)$
 & 15
 & 13
 \\ \hline
 80
 & $(1,2,5,7^2,8)$
 & 15
 & 13
 \\ \hline
 81
 & $(1,3^2,4,5,14)$
 & 15
 & 7
 \\ \hline
 82
 & $(1,3^2,4,7,12)$
 & 15
 & 11
 \\ \hline
 83
 & $(1,3,4^2,7,11)$
 & 15
 & 11
 \\ \hline
 84
 & $(1,3,4,5,6,11)$
 & 15
 & 11
 \\ \hline
 85
 & $(1,3,4,5,7,10)$
 & 15
 & 11
 \\ \hline
 86
 & $(1,3,4,6,7,9)$
 & 15
 & 10
 \\ \hline
 87
 & $(1,3,4,7^2,8)$
 & 15
 & 11
 \\ \hline
 88
 & $(1,3,5^2,6,10)$
 & 15
 & 11
 \\ \hline
 89
 & $(1,3,5,6,7,8)$
 & 15
 & 11
 \\ \hline
 90
 & $(3^2,4^2,5,11)$
 & 15
 & 7
 \\
 \hline
 91
 & $(3^2,4,5^2,10)$
 & 15
 & 7
 \\ \hline
 92
 & $(1^2,2,5,8,15)$
 & 16
 & 14
 \\ \hline
 93
 & $(1^2,3,4,8,15)$
 & 16
 & 12
 \\ \hline
 94
 & $(1^2,4,5,6,15)$
 & 16
 & 10
 \\ \hline
 95
 & $(1,2,3,4,7,15)$
 & 16
 & 8
 \\ \hline
 96
 & $(1,2,3,5,8,13)$
 & 16
 & 14
 \\ \hline
 97
 & $(1,2,5^2,8,11)$
 & 16
 & 14
 \\ \hline
 98
 & $(1,2,5,7,8,9)$
 & 16
 & 14
 \\ \hline
 99
 & $(1,3^2,4,8,13)$
 & 16
 & 12
 \\ \hline
 100
 & $(1,3,4,5,6,13)$
 & 16
 & 10
 \\ \hline
 101
 & $(1,3,4,5,8,11)$
 & 16
 & 12
 \\ \hline
 102
 & $(1,4,5^2,6,11)$
 & 16
 & 10
 \\ \hline
 103
 & $(1,2,3,5,7,16)$
 & 17
 & 7
 \\ \hline
 104
 & $(1^2,2,6,9,17)$
 & 18
 & 15
 \\ \hline
 105
 & $(1^2,3,5,9,17)$
 & 18
 & 12
 \\ \hline
 106
 & $(1,2,3,4,9,17)$
 & 18
 & 9
 \\ \hline
 107
 & $(1,2,3,5,8,17)$
 & 18
 & 7
 \\ \hline
 108
 & $(1,2,3,5,9,16)$
 & 18
 & 12
 \\ \hline
 109
 & $(1,3,5^2,9,13)$
 & 18
 & 12
 \\ \hline
 110
 & $(2,3^2,5,8,15)$
 & 18
 & 7
 \\ \hline
 111
 & $(2,3,4,5,9,12)$
 & 18
 & 5
 \\ \hline
 112
 & $(2,3,4,7,9,11)$
 & 18
 & 9
 \\ \hline
 113
 & $(2,3,5^2,8,13)$
 & 18
 & 7
 \\ \hline
 114
 & $(2,3,5,8,9^2)$
 & 18
 & 7
 \\ \hline
 115
 & $(1,3,4,5,7,18)$
 & 19
 & 5
 \\ \hline
 116
 & $(1^2,4,5,10,19)$
 & 20
 & 11
 \\ \hline
 117
 & $(1,2,3,5,10,19)$
 & 20
 & 8
 \\ \hline
 118
 & $(1,2,4,5,9,19)$
 & 20
 & 7
 \\ \hline
 119
 & $(1,2,5,6,7,19)$
 & 20
 & 6
 \\ \hline
 120
 & $(2,3^2,5,10,17)$
 & 20
 & 8
 \\ \hline
 121
 & $(2,3,4,5,9,17)$
 & 20
 & 7
 \\ \hline
 122
 & $(2,3,5,6,7,17)$
 & 20
 & 6
 \\ \hline
 123
 & $(2,3,5,9,10,11)$
 & 20
 & 8
 \\ \hline
 124
 & $(2,4,5^2,9,15)$
 & 20
 & 7
 \\ \hline
 125
 & $(2,4,5,9^2,11)$
 & 20
 & 7
 \\ \hline
 126
 & $(2,5^2,6,7,15)$
 & 20
 & 6
 \\ \hline
 127
 & $(2,5,6,7^2,13)$
 & 20
 & 6
 \\ \hline
 128
 & $(2,5,6,7,9,11)$
 & 20
 & 6
 \\ \hline
 129
 & $(1^2,3,7,10,20)$
 & 21
 & 11
 \\ \hline
 130
 & $(1^2,5,7,8,20)$
 & 21
 & 9
 \\ \hline
 131
 & $(1,2,3,7,9,10)$
 & 21
 & 17
 \\ \hline
 132
 & $(1,2,3,7,10,19)$
 & 21
 & 11
 \\ \hline
 133
 & $(1,2,5,7,8,19)$
 & 21
 & 9
 \\ \hline
 134
 & $(1,3^2,7,10,18)$
 & 21
 & 11
 \\ \hline
 135
 & $(1,3,4,7,10,17)$
 & 21
 & 11
 \\ \hline
 136
 & $(1,3,5,6,7,20)$
 & 21
 & 5
 \\ \hline
 137
 & $(1,3,5,7,8,18)$
 & 21
 & 9
 \\ \hline
 138
 & $(1,3,5,7,10,16)$
 & 21
 & 11
 \\ \hline
 139
 & $(1,3,6,7,10,15)$
 & 21
 & 11
 \\ \hline
 140
 & $(1,3,7^2,10,14)$
 & 21
 & 11
 \\ \hline
 141
 & $(1,3,7,9,10,12)$
 & 21
 & 11
 \\ \hline
 142
 & $(1,3,7,10^2,11)$
 & 21
 & 11
 \\ \hline
 143
 & $(1,4,5,7,8,17)$
 & 21
 & 9
 \\ \hline
 144
 & $(1,5^2,7,8,16)$
 & 21
 & 9
 \\ \hline
 145
 & $(1,5,7^2,8,14)$
 & 21
 & 9
 \\ \hline
 146
 & $(1,5,7,8^2,13)$
 & 21
 & 9
 \\ \hline
 147
 & $(1,5,7,8,10,11)$
 & 21
 & 9
 \\ \hline
 148
 & $(1^2,3,7,11,21)$
 & 22
 & 12
 \\ \hline
 149
 & $(1,2,3,7,11,20)$
 & 22
 & 12
 \\ \hline
 150
 & $(1,2,4,5,11,21)$
 & 22
 & 8
 \\ \hline
 151
 & $(1,3^2,7,11,19)$
 & 22
 & 12
 \\ \hline
 152
 & $(1,3,7^2,11,15)$
 & 22
 & 12
 \\ \hline
 153
 & $(1,3,7,11^3)$
 & 22
 & 12
 \\ \hline
 154
 & $(2,4,5^2,11,17)$
 & 22
 & 8
 \\ \hline
 155
 & $(2,4,5,9,11,13)$
 & 22
 & 8
 \\ \hline
 156
 & $(2,4,5,11^3)$
 & 22
 & 8
 \\ \hline
 157
 & $(1^2,3,8,12,23)$
 & 24
 & 13
 \\ \hline
 158
 & $(1^2,6,8,9,23)$
 & 24
 & 9
 \\ \hline
 159
 & $(1,2,3,7,12,23)$
 & 24
 & 8
 \\ \hline
 160
 & $(1,2,3,8,11,23)$
 & 24
 & 7
 \\ \hline
 161
 & $(1,3,4,5,12,23)$
 & 24
 & 6
 \\ \hline
 162
 & $(1,3,4,7,10,23)$
 & 24
 & 4
 \\ \hline
 163
 & $(1,3,5,7,8,23)$
 & 24
 & 1
 \\ \hline
 164
 & $(1,3,7,8,12,17)$
 & 24
 & 13
 \\ \hline
 165
 & $(1,4,5,6,9,23)$
 & 24
 & 4
 \\ \hline
 166
 & $(2,3^2,8,11,21)$
 & 24
 & 7
 \\ \hline
 167
 & $(2,3,7^2,12,17)$
 & 24
 & 8
 \\ \hline
 168
 & $(2,3,7,8,11,17)$
 & 24
 & 7
 \\ \hline
 169
 & $(2,3,7,11,12,13)$
 & 24
 & 8
 \\ \hline
 170
 & $(2,3,8,11^2,13)$
 & 24
 & 7
 \\ \hline
 171
 & $(3^2,4,7,10,21)$
 & 24
 & 4
 \\ \hline
 172
 & $(3,4,5^2,12,19)$
 & 24
 & 6 
 \\ \hline
 173
 & $(3,4,5,7,10,19)$
 & 24
 & 4
 \\ \hline
 174
 & $(3,4,5,7,12,17)$
 & 24
 & 6
 \\ \hline
 175
 & $(3,4,7^2,10,17)$
 & 24
 & 4
 \\ \hline
 176
 & $(3,6,7^2,8,17)$
 & 24
 & 5
 \\ \hline
 177
 & $(4,5^2,6,9,19)$
 & 24
 & 4
 \\ \hline
 178
 & $(1,4,5,7,9,24)$
 & 25
 & 3
 \\ \hline
 179
 & $(1^2,5,7,13,25)$
 & 26
 & 10
 \\ \hline
 180
 & $(1,2,3,8,13,25)$
 & 26
 & 8
 \\ \hline
 181
 & $(1,2,5,6,13,25)$
 & 26
 & 7
 \\ \hline
 182
 & $(1,2,5,7,13,24)$
 & 26
 & 10
 \\ \hline
 183
 & $(1,5^2,7,13,21)$
 & 26
 & 10
 \\ \hline
 184
 & $(1,5,7^2,13,19)$
 & 26
 & 10
 \\ \hline
 185
 & $(1,5,7,13^3)$
 & 26
 & 10
 \\ \hline
 186
 & $(2,3^2,8,13,23)$
 & 26
 & 8
 \\ \hline
 187
 & $(2,3,8,9,13,17)$
 & 26
 & 8
 \\ \hline
 188
 & $(2,3,8,13^3)$
 & 26
 & 8
 \\ \hline 
 189
 & $(1,2,5,9,11,26)$
 & 27
 & 5
 \\ \hline
 190
 & $(1^2,4,9,14,27)$
 & 28
 & 11
 \\ \hline
 191
 & $(1,3,4,7,14,27)$
 & 28
 & 5
 \\ \hline
 192
 & $(1,3,4,9,14,25)$
 & 28
 & 11
 \\ \hline
 193
 & $(1,4,6,7,11,27)$
 & 28
 & 3
 \\ \hline
 194
 & $(1,4,7,9,14,21)$
 & 28
 & 11
 \\ \hline
 195
 & $(1,4,9^2,14,19)$
 & 28
 & 11
 \\ \hline
 196
 & $(4,6,7^2,11,21)$
 & 28
 & 3
 \\ \hline
 197
 & $(4,6,7,11^2,17)$
 & 28
 & 3
 \\ \hline
 198
 & $(1^2,6,8,15,29)$
 & 30
 & 10
 \\ \hline
 199
 & $(1,2,3,10,15,29)$
 & 30
 & 9
 \\ \hline
 200
 & $(1,2,6,7,15,29)$
 & 30
 & 7
 \\ \hline
 201
 & $(1,5,6,8,11,29)$
 & 30
 & 2
 \\ \hline
 202
 & $(1,5,6,8,15,25)$
 & 30
 & 10
 \\ \hline
 203
 & $(1,6,8,11,15,19)$
 & 30
 & 10
 \\ \hline
 204
 & $(2,3,7,10,15,23)$
 & 30
 & 9
 \\ \hline
 205
 & $(2,5,6,7,15,25)$
 & 30
 & 7
 \\ \hline
 206
 & $(2,6,7^2,15,23)$
 & 30
 & 7
 \\ \hline
 207
 & $(3,5,6,8,11,27)$
 & 30
 & 2
 \\ \hline
 208
 & $(5^2,6,8,11,25)$
 & 30
 & 2
 \\ \hline
 209
 & $(5,6,8,11^2,19)$
 & 30
 & 2
 \\ \hline
 210
 & $(5,6,8,11,15^2)$
 & 30
 & 2
 \\ \hline
 211
 & $(1,2,5,9,16,31)$
 & 32
 & 6
 \\ \hline
 212
 & $(2,3,5,9,16,29)$
 & 32
 & 6
 \\ \hline
 213
 & $(2,5^2,9,16,27)$
 & 32
 & 6
 \\ \hline
 214
 & $(2,5,9^2,16,23)$
 & 32
 & 6
 \\ \hline
 215
 & $(2,5,9,15,16,17)$
 & 32
 & 6
 \\ \hline
 216
 & $(1,3,5,11,14,32)$
 & 33
 & 3
 \\ \hline
 217
 & $(3^2,5,11,14,30)$
 & 33
 & 3
 \\ \hline
 218
 & $(3,5^2,11,14,28)$
 & 33
 & 3
 \\ \hline
 219
 & $(3,5,6,11,14,27)$
 & 33
 & 3
 \\ \hline
 220
 & $(3,5,7,11,14,26)$
 & 33
 & 3
 \\ \hline
 221
 & $(3,5,10,11,14,23)$
 & 33
 & 3
 \\ \hline
 222
 & $(3,5,11^2,14,22)$
 & 33
 & 3
 \\ \hline
 223
 & $(3,5,11,14^2,19)$
 & 33
 & 3
 \\ \hline
 224
 & $(3,5,11,14,15,18)$
 & 33
 & 3
 \\ \hline
 225
 & $(1^2,5,12,18,35)$
 & 36
 & 11
 \\ \hline
 226
 & $(1,3,4,11,18,35)$
 & 36
 & 5 
 \\ \hline
 227
 & $(1,5^2,12,18,31)$
 & 36
 & 11
 \\ \hline
 228
 & $(1,5,7,12,18,29)$
 & 36
 & 11
 \\ \hline
 229
 & $(2,4,11^2,18,25)$
 & 36
 & 7
 \\ \hline
 230
 & $(2,3,5,11,19,36)$
 & 38
 & 4
 \\ \hline
 231
 & $(1,5,7,8,20,39)$
 & 40
 & 3
 \\ \hline
 232
 & $(5,7^2,8,20,33)$
 & 40
 & 3
 \\ \hline
 233
 & $(5,7,8,11,20,29)$
 & 40
 & 3
 \\ \hline
 234
 & $(1^2,6,14,21,41)$
 & 42
 & 11
 \\ \hline
 235
 & $(1,2,5,14,21,41)$
 & 42
 & 7
 \\ \hline
 236
 & $(2,3,5,14,21,39)$
 & 42
 & 7
 \\ \hline
 237
 & $(2,5^2,14,21,37)$
 & 42
 & 7
 \\ \hline
 238
 & $(1,4,5,13,22,43)$
 & 44
 & 3
 \\ \hline
 239
 & $(4,5^2,13,22,39)$
 & 44
 & 3
 \\ \hline
 240
 & $(4,5,11,13,22,33)$
 & 44
 & 3
 \\ \hline
 241
 & $(4,5,13^2,22,31)$
 & 44
 & 3
 \\ \hline
 242
 & $(1,3,5,16,24,47)$
 & 48
 & 5
 \\
 243
 & $(3,5^2,16,24,43)$
 & 48
 & 5
 \\ \hline
 244
 & $(1,5,6,22,33,65)$
 & 66
 & 3
 \end{longtable}
 \rowcolors{0}{}{}
\renewcommand{\arraystretch}{1.0}

\normalsize

\begin{remark} The table does not contain repetitions, as no isomorphism exists between hypersurfaces corresponding to different rows in the table; see \cite[Theorem 2.1]{Esser24}.
\end{remark}

\subsection{Geometric description}
\label{subse:geometricdescription}
We now provide a geometric description for any example in \cref{table fourfolds} but the cubic fourfold, explaining where the K3 structure of the FK3 weighted fourfold comes from. The proof makes use of weighted blow ups, and we refer the reader to \cite{ATW24}, \cite[Section 10]{KM92}, or \cite{Takayuki99} for a comprehensive treatment of the subject.
\begin{proposition}\label{prop blow-up}
\label{prop FK3 structure}
    Let \(X_d\subseteq\P\bigl(a_0,\dots,a_5)\) be a FK3 weighted fourfold with $\dim X_d^{\mathrm{sing}}\le 1$ that is not a cubic fourfold, and call $i$ an index such that $a_i+a_5=d$, given by \cref{thm ai+a5=d}. 
    Then there exists an isomorphism 
    \[(a_0,\dots,a_4)\mathrm{Bl}_p X_d\cong \mathrm{Bl}_{S_d}\P(a_0,\dots,a_4), 
    \]
    $S_d$ being the weighted K3 surface $S_d\subseteq \P(a_0,\dots,a_4)$ defined by $V(x_i,g_d(x_0,\dots,x_4))$, where $g_d$ is a general homogeneous polynomial of degree $d$.

    
    
    
    Furthermore, there exists an isomorphism of Hodge structures $H^4(X_d,\mathbf Q)\cong H^2(S_d,\mathbf Q)$, where $H^4(X_d,\mathbf Q)$ has the weight 2 Hodge structure given by the K3 structure in cohomology.

\end{proposition}
\begin{proof}
Let $x_0,\dots,x_5$ be graded homogeneous coordinates for $\P(a_0,\dots,a_5)$, with $\deg x_i=a_i$. From the hypotheses on weights and degree it follows that the polynomial defining $X_d$ is linear in $x_5$, i.e.\ it is of the form $f_d=x_5h_{a_i}(x_0,\dots,x_4)+g_d(x_0,\dots,x_4)$, where $\deg(h_{a_i})=a_i$ and $\deg(g_d)=d$. Up to a homogeneous transformation of the graded ring defining the weighted projective space, and hence up to an automorphism of the weighted projective space, we can assume that $h_{a_i}=x_i$, i.e.\ $f_d=x_ix_5+g_d(x_0,\dots,x_4)$. We define $S_d:=V(x_i,g_d)\subseteq \P(a_0,\dots,a_4)$, that is a K3 surface by \cref{cor the k3 is a k3}.

Let $\pi:\P(a_0,\dots,a_5) \dashrightarrow \P\bigl(a_0,\dots,a_4)$ be the projection from $p=(0:\dots:0:1)$. We can resolve this rational map by performing a weighted blow up at the point $p$, with weights $(a_0,\dots,a_4)$. Here the blow ups should be rather considered in the category of Deligne-Mumford stacks. 
 The morphism resolving $\pi$ is then the morphism $\widetilde{\pi}$, which is the restriction of the projection $\P(a_0,\dots,a_4)\times \P(a_0,\dots,a_5) \to \P(a_0,\dots,a_4)$ to $(a_0,\dots,a_4)\mathrm{Bl}_p\P(a_0,\dots,a_5)$. We have a commutative diagram
\begin{center}
\begin{tikzcd}[column sep=0.8cm]
    			&(a_0,\dots,a_4)\mathrm{Bl}_p\P(a_0,\dots,a_5) \ar[dr, "\widetilde{\pi}"]  \arrow["h"' ,dl]&
    			\\ 
    			\P(a_0,\dots,a_5) \arrow[rr, dashed, "\pi"] && \P(a_0,\dots,a_4) 
\end{tikzcd}
\end{center}
Now, consider the diagram
\begin{center}
\begin{tikzcd}[column sep=0.8cm]
    			& \widetilde{X_d} \ar[dr, "{\widetilde{\pi}}_{|\widetilde{X_d}}"]  \arrow[dl, "h_{|\widetilde{X_d}}"']&
    			\\ 
    			X_d \arrow[rr, dashed, "\pi_{|X_d}"] && \P(a_0,\dots,a_4) 
\end{tikzcd},
\end{center}
where $\widetilde{X_d}$ is the strict transform of $X_d$ via $h$. The morphism ${\widetilde{\pi}}_{|\widetilde{X_d}}$ is birational, and the dimension of the fibers jumps exactly at the common zero locus of $x_i$ and $g_d$ in $\P(a_0,\dots,a_4)$ (that is, on $S_d$). We recall that $(a_0,\dots,a_4)\mathrm{Bl}_p\P(a_0,\dots,a_5)$ is a toric variety. Then, let $z$ be the defining section of the exceptional divisor $E\cong \P(a_0,\dots,a_4)$ of $h$. It follows that the total transform $h^*(X_d)$ of $X_d$ via $h$ is given by $\{z^{a_i}x_ix_5+g_d(z^{a_0}x_0,\dots,z^{a_4}x_4)=0\}$, and hence an equation of $\widetilde{X_d}$ inside $(a_0,\dots,a_4)\mathrm{Bl}_p\P(a_0,\dots,a_5)$ is obtained by taking out $z^{a_i}$ from the equation above. Now, we have $$(a_0,\dots,a_4)\mathrm{Bl}_p\P(a_0,\dots,a_5) \cong \mathcal{P}roj(\mathrm{Sym}^{\bullet}(\mathcal{O}\oplus \mathcal{O}(-a_5)))=:\mathbf{P}(\mathcal{O}\oplus\mathcal{O}(-a_5)),$$ where $\mathcal{P}roj(\mathrm{Sym}^{\bullet}(\mathcal{O}\oplus \mathcal{O}(-a_5)))$ denotes the stack-theoretic Proj of the symmetric algebra $\mathrm{Sym}^{\bullet}(\mathcal{O}\oplus \mathcal{O}(-a_5))$ (see for example \cite[Subsection 1.2]{QR22}, or \cite[Subsection 3.1]{ATW24}).  But then the strict transform $\widetilde{X}_d$ is the zero locus of a section of $\mathcal{O}_{\mathbf{P}(\mathcal{O}\oplus \mathcal{O}(-a_5))}(1)\otimes \pi^*\mathcal{O}(-a_i)$, where $\pi$ is the structural morphism of $\mathbf{P}(\mathcal{O}\oplus\mathcal{O}(-a_5))$, and this coincides with the blow up $\mathrm{Bl}_{S_d}\P(a_0,\dots,a_4)$. On the other hand, by construction we also have $\widetilde{X_d} \cong (a_0,\dots,a_4)\mathrm{Bl}_p X_d$, hence the claim.

The claim on the Hodge structure follows from the blow up formula on cohomology, which adapts to our singular setting by \cite[Theorem 1.21]{blache1996chern}
\end{proof}

\begin{remark} The isomorphism of Hodge structures in \cref{prop blow-up} can also be deduced (over $\C$) with some computations, via Griffiths' residue \cref{thm Hodge structure of Xd}. Indeed, consider \(X_d= V(f_d)\), where \(f_d=x_5x_i+g_d( \hat{x_5})\),  \(i \in \{0,\dots,4\}\), as in the statement. Without loss of generality, we can assume that \(x_i=x_0\). Then the associated K3 surface is \(S_d=V(x_0, g_d(x_0, \ldots, x_4))\), i.e.\ it is the weighted hypersurface \(S=V(g_d)\subseteq \P(a_1,\dots,a_4)\). We have
\begin{align*}
   h^{2,2}_{\mathrm{prim}}(X_d) &= \dim (S(a_0,\dots,a_5)/J_{f_d})_{d}  \\
   h^{1,1}_{\mathrm{prim}}(S)&= \dim (S(a_1,a_2,a_3,a_4)/J_{g_d})_{d}
\end{align*}
and 
\begin{align*}
  h^{3,1}_{\mathrm{prim}}(X_d)&= \dim (S(a_0,\dots,a_5)/J_{f_d})_{0}  \\
  h^{2,0}_{\mathrm{prim}}(S)&= \dim (S(a_1,a_2,a_3,a_4)/J_{g_d})_{0}.
\end{align*}

We compute \(J_{f_d}=\langle x_0, \frac{\partial g_d}{\partial x_1}, \frac{\partial g_d}{\partial x_2}, \frac{\partial g_d}{\partial x_3}, \frac{\partial g_d}{\partial x_4}, \frac{\partial f}{\partial x_0} \rangle\) and \(J_{g_d}=\langle \frac{\partial g_d}{\partial x_1}, \frac{\partial g_d}{\partial x_2}, \frac{\partial g_d}{\partial x_3}, \frac{\partial g_d}{\partial x_4}\rangle \). Note that in \(S(a_0,\dots,a_5)/J_{f_d}\) we are imposing \(\frac{\partial{f_d}}{\partial{x_0}}=0\), which means \(x_5+\frac{\partial{g_d}}{\partial{x_0}}=0\). This implies that \(x_5\) depends on the variables \(x_1,x_2,x_3,x_4\) in the quotient ring \(S(a_0,\dots,a_5)/J_{f_d}\). In particular, the rings \(S(a_0,\dots,a_5)/J_{f_d}\) and \(S(a_1,a_2,a_3,a_4)/J_{g_d}\) are isomorphic, and hence also the subspaces of homogeneous polynomials of a fixed degree are.
\end{remark}

\begin{remark} \label{rmk:defo}

We have already noticed how, in our setting, the first-order deformations of our Fano are controlled by $H^1(T_{X_d})$. Thanks to the previous computation, this is isomorphic to $H^{2,2}_{\textrm{prim}}(X_d)$, which in turn is isomorphic to $H^{1,1}_{\textrm{prim}}(S_d)$. The latter can be identified with the Kernel of the map $H^1(T_{S_d}) \to H^2(\mathcal{O}_{S_d})$ given by the multiplication with $c_1(L)$ from the Atiyah extension. In other words, we are parametrizing the deformations of the couple $(S,L)$, where $L\cong \mathcal O_S(\text{lcm}\{a_i\}_i)$ is the ample generator of the Picard group. This observation is especially interesting in light of the above \cref{prop blow-up}. In fact, normally, for $Z \subset Y$, the deformation theory of $\widetilde{Y}= \Bl_ZY$ would be computed via the exact sequence
\[
0 \to f^*(\Omega^1_Y) \to \Omega^1_{\widetilde{Y}} \to j_* \Omega^1_{E/Z} \to 0,
\]
with $f$ the blow up map and $j: E \to \widetilde{Y}$ the inclusion of the exceptional divisor $E$. In our case, the situation is in fact much simpler, with the K3 $S_d$ being the only responsible \emph{a posteriori} for the deformations of $X_d$.

\end{remark}

\begin{corollary}\label{rationality} Every FK3 weighted fourfold $X_d$ with singular locus of dimension at most 1 that is not a cubic fourfold is rational. 
\end{corollary}
\begin{proof}
This is a straightforward consequence of \cref{thm ai+a5=d,prop blow-up}. In fact, by \cref{thm ai+a5=d} for any $X_d$ as in the statement there exists an index $i$ such that $a_i+a_5=d$, hence by \cref{prop blow-up} it is birational to the blow up of some weighted projective space, which gives the rationality. 
\end{proof}
See \cite[Proposition 3.1]{Esser25} for a different (but related) proof of \cref{rationality} in the case $2a_4=2a_5=d$, and \cite[Theorem 1.2]{Okada19} for an analogous statement for weighted Fano threefolds.

The results of \cref{rationality} can be compared with the derived categorical approach. In fact, our approach can be also found in the work of Kuznetsov, and later on Perry. As in \cite[\S 4.4]{Kuz19}, extended by Perry in \cite[Example 6.11]{Per22}, the derived category of perfect complexes $\text{D}_{\text{perf}}(X)$ of our FK3 contains a CY2 category as the orthogonal to a collection of exceptional objects, which is commutative via the blow up map. This holds for all of the examples in Table \cref{table fourfolds}. 

As we are going to see, if we drop the hypothesis on the codimension of the singular locus, there are two examples which do dot fall into this pattern, see \cref{table fourfolds from del Pezzo}.

\begin{remark}\label{remark double suspension}
In the case where $X_d\subseteq\P(a_0,\dots,a_5)$ is such that $a_4=a_5=\frac{d}{2}$, we have another interesting geometric description for $X_d$. Indeed, up to a change of coordinates in $x_4,x_5$ we can transform the equation of $X_d$ given in the proof of \cref{prop blow-up} to $\tilde f_d=x_5^2+x_4^2+g_d(x_0,\dots,x_3)$. We remark that this is a consequence of the quasi-smoothness of $X_d$, that forces the quadratic part of the equation of $X_d$, in the coordinates $x_4,x_5$, to have non-zero eigenvalues; furthermore, observe that we did not change the equation defining the K3 surface $S=V(g_d)$. Geometrically, we obtain that the projection $\P\bigl(a_0,\dots,a_3,\frac{d}{2},\frac{d}{2}\bigl)\dashrightarrow\P\bigl(a_0,\dots,a_3,\frac{d}{2}\bigl)$ restricts to a double cover $X_d\xrightarrow{\phi_1} \P\bigl(a_0,\dots,a_3,\frac{d}{2}\bigl)$, with branch locus $Y:=V(x_4^2+g_d(x_0,\dots,x_3))$, and that the further projection $\P\bigl(a_0,\dots,a_3,\frac{d}{2}\bigl)\dashrightarrow\P\bigl(a_0,\dots,a_3\bigl)$ restricts to a double cover $Y\xrightarrow{\phi_2} \P(a_0,\dots,a_3)$, with branch locus  $S$, i.e.\ $X_d$ is described by the following diagram:
\[
        \begin{tikzcd}
        & & X_d \ar[d, "\phi_1" left, "2:1" near start] \ar[r, hook] & \P\Bigl(a_0,\dots,a_3,\frac{d}{2},\frac{d}{2}\Bigl) \\
        & Y \ar[r, hook] \ar[d, "\phi_2" left, "2:1" near start] & \P\Bigl(a_0,\dots,a_3,\frac{d}{2}\Bigl) &  \\
        S \ar[r, hook] & \P(a_0,\dots,a_3). &
        \end{tikzcd}
        \]

It is worth noting that the same construction is performed in \cite[Theorem C]{FRZ19}, where the Hodge-theoretical behavior was already described.
\end{remark}

\section{Singularities}
\label{section:sing}
In this section, we determine the type of singularities of the FK3 weighted fourfolds. Since \( X_d \subset \P(a_0, \ldots, a_4) \) is general, the cyclic quotient singularities of the hypersurfaces can be described by following \cite[\S 2.5]{Brown_Kasprzyk} and \cite[\S 10]{Fletcher00}. We then apply the Reid--Shepherd-Barron--Tai criterion, as stated in \cite[p.\ 376, Theorem]{Reid_criterion}.
\medskip

The next paragraph is taken from \cite[\S 2.5]{Brown_Kasprzyk}, and we provide it for the reader's convenience. 
\medskip

The orbifold strata $\Pi$ in the ambient space \( \P(a_0, \ldots, a_s) \) correspond to subsets of indices \( I \subset \{1, \ldots, s\} \) for which 
\[
r_I := \gcd\{a_i \mid i \in I\} > 1.
\]
We only need to work with the maximal \( I \) for any given \( r = r_I \), so we always assume that this is the case.

Given a hypersurface \( X_d = V(f_d) \subset \P(a_0, \ldots, a_s) \), consider an ambient orbifold stratum \( \Pi = \P(a_{i_1}, \ldots, a_{i_t}) \) of transverse type \( \frac{1}{r}(b_1, \ldots, b_{s-t}) \). Then
\[
\{a_{i_1}, \ldots, a_{i_t}, b_1, \ldots, b_{s-t}\} = \{a_0, \ldots, a_s\}, 
\]
and one of the following two things can happen:

\begin{itemize}
  \item[(i)] The polynomial \( f_d\) vanishes on \( \Pi \), so that \( \Pi \subset X \). In this case, at any point \( P \in X \cap \Pi^\circ \), \( X \) has a transverse quotient singularity of type \( \frac{1}{r}(b_1, \ldots, \hat{b}_j, \ldots, b_{s-t}) \), where \( x_j \) is a tangent variable to \( X \) at \( P \), i.e.\ $\mathrm{deg}(x_j)=d-rk$ for some positive integer $k$. Here $\Pi^\circ$ consists of the points in $\Pi$ where all the coordinates are non-vanishing. Note that there may be several tangent variables at \( P \), but their weights are congruent modulo \( r \), and so any one may be used.

  \item[(ii)] \( V(f_d) \) cuts a codimension-one locus transversely inside \( \Pi \). In this case, at any point \( P \in X \cap \Pi^\circ \), the hypersurface \( X \) has a transverse quotient singularity of type \( \frac{1}{r}(b_1, \ldots, b_{s-t}) \).
\end{itemize}

This is enough to determine the singularities of \( X \).
\medskip

We recall now the Reid--Shepherd-Barron--Tai criterion.

\begin{theorem}{\cite[p.\ 376, Theorem]{Reid_criterion}}\label{thm Reid terminal criterion}
Let \(\frac{1}{r}(b_1, \ldots, b_s)\) be a well-formed cyclic quotient singularity. The singularity is canonical (resp. terminal) if and only if
\[ 
\sum_{j=1}^{s} \left[ \frac{kb_j}{r}
\right ] \geq 1
\]
(resp.\( > 1\)) for all \(k=1, \ldots, r-1.\) Here, for any rational number $q$, the number $[q]$ denotes the fractional part of $q$.
\end{theorem}

In the following \cref{table sigularities}, since we are only interested in the type of singularities, we will not write the zeros arising from the reduction modulo $r_I$. In the “type” column of the table, the notation is as follows: terminal means that all singularities are terminal, canonical means that all singularities are canonical and at least one is not terminal, and klt indicates that at least one singularity is strictly klt. 

We leave the computations in the table to the reader and focus instead on providing an example that highlights the phenomena associated with the singularities inherited by the hypersurface from the ambient weighted projective space.

\begin{example}
Consider the FK3 \#8 in \cref{table sigularities}, i.e.\ a general hypersurface \(X_6\subseteq \P(1^3,2,3,4)\). The singular locus of \(\P(1^3,2,3,4)\) contains the weighted projective line \(\P(2,4)\), defined by $V(x_0,x_1,x_2,x_4)$, and the type of singularity along $\P(2,4) \setminus \{(0:\cdots:0:1)=:P\}$ is \(\frac{1}{2}(1,1,1,3)=\frac{1}{2}(1,1,1,1)\), whereas the singularity $P$ is of type  \(\frac{1}{4}(1,1,1,2,3)\). The weighted projective line \(\P(2,4)\) intersects the hypersurface \(X_6\) transversely in finitely many points. The points of $X_6$ lying in $\P(2,4) \setminus \{P\}$ are singularities of type \(\frac{1}{2}(1,1,1,1)\) as well. Instead, since the variable $x_3$ is tangent to the hypersurface $X_6$ at the point $P$, it follows that $P$ is a singularity of type \(\frac{1}{4}(1,1,1,3)\) on $X_6$. To conclude, the singular locus of the weighted projective space \(\P(1^3,2,3,4)\) contains the isolated singularity $(0: \cdots :0:1:0)$, which is of type \(\frac{1}{3}(1,1,1,2,4)=\frac{1}{3}(1,1,1,2,1)\), and does not belong to \(X_6\).

\end{example}

\tiny
\renewcommand{\arraystretch}{1.3}
\rowcolors{2}{gray!10}{white}
\begin{longtable}{c||c|c|c|c}
\caption{Singularities of FK3 weighted fourfolds $X_d\subseteq \P(\underline a)$ with \(\dim X_d^{\mathrm{sing}}\le 1\).
} 
\label{table sigularities}
\\
 \rowcolor{white}
 \(\#\)
 & $(\underline a)$
 & $d$
 & singularities
 & type 
 \\ \hline \hline \endfirsthead
 
 \rowcolor{white}
  \(\#\)
 & $(\underline a)$
 & $d$
 & singularities
 & type 
 \\ \hline \hline \endhead 

 1
 & $(1^6)$
 & 3
 & 
 & 
 \\ \hline

 2
 & $(1^5,3) $
 & 4
 & \(\frac{1}{3}(1,1,1,1)\)
 & terminal
 \\ \hline

 3
 & $(1^4,2^2)$
 & 4
 & \(\frac{1}{2}(1,1,1,1)\)
 & terminal
 \\ \hline
 
 4
 & $(1^4,2,4)$
 & 5
 & \(\frac{1}{2}(1,1,1) \ \ \frac{1}{4}(1,1,1,2)\)
 & terminal
 \\ \hline
 
 5
 & $(1^3,2^2,3)$
 & 5
 & \(\frac{1}{2}(1,1,1) \ \ \frac{1}{3}(1,1,1,2)\) 
 & terminal
 \\ \hline
 
 6
 & $(1^4,3,5) $
 & 6
 & \( \frac{1}{5}(1,1,1,3) \)
 & terminal
 \\ \hline
 7
 & $(1^3,2^2,5)$ 
 &  6
 &  \( \frac{1}{2}(1,1,1,1) \ \  \frac{1}{5}(1,1,2,2) \)
 & terminal
 \\ \hline
 8
 & $(1^3,2,3,4)$ 
 & 6
 & $\frac{1}{4}(1,1,1,3) \ \ \frac{1}{2}(1,1,1,1)$
 & terminal
 \\ \hline
 9
 & $(1^3,3^3)$
 & 6
 & \( \frac{1}{3}(1,1,1) \)
 & canonical
 \\ \hline
 10
 & $(1^2,2^2,3^2)$
 & 6
 & $\frac{1}{3}(1,1,2,2) \ \ \frac{1}{2}(1,1,1,1)$
 & terminal
 \\ \hline
 11
 & $(1^3,2,3,6)$
 & 7
 & $\frac{1}{2}(1,1,1) \ \ \frac{1}{3}(1,1,2) \ \  \frac{1}{6}(1,1,2,3)$
 & terminal
 \\ \hline
 12
 & $(1^2,2^2,3,5)$
 & 7
 & $\frac{1}{2}(1,1,1) \ \ \frac{1}{3}(1,1,2,2,2) \ \  \frac{1}{5}(1,1,2,3)$
 & terminal
 \\ \hline 
 13
 & $(1^2,2,3^2,4)$
 & 7
 & $\frac{1}{2}(1,1,1) \ \ \frac{1}{3}(1,2,1)) \ \  \frac{1}{4}(1,1,2,3)$
 & terminal
 \\ \hline
 14
 & $(1^3,2,4,7)$
 & 8
 & $\frac{1}{2}(1,1,1,1) \ \ \frac{1}{7}(1,1,2,4)$
 & terminal
 \\ \hline
 15
 & $(1^2,2^2,3,7)$
 & 8
 & $\frac{1}{2}(1,1,1,1) \ \ \frac{1}{3}(1,1,2,1)) \ \ \frac{1}{7}(1,2,2,3)$
 & terminal
 \\ \hline
 16
 & $(1^2,2,3,4,5)$
 & 8
 & $\frac{1}{2}(1,1,1,1) \ \ \frac{1}{3}(1,1,1,2)) \ \ \frac{1}{5}(1,1,2,4)$
 & terminal
 \\ \hline
 17
 & $(1,2^2,3^2,5)$  
 & 8
 & $\frac{1}{2}(1,1,1,1) \ \ \frac{1}{3}(1,2,2)) \ \ \frac{1}{5}(1,2,2,3)$
 & terminal
 \\ \hline
 18
 & $(1^3,3,4,8)$
 & 9
 & $\frac{1}{4}(1,1,3) \ \ \frac{1}{8}(1,1,3,4) $
 & terminal
 \\ \hline
 19
 & $(1^2,2,3^2,8)$
 & 9
 & $\frac{1}{2}(1,1,1) \ \ \frac{1}{3}(1,1,2,2) \ \ \frac{1}{8}(1,2,3,3)$
 & terminal
 \\ \hline
 20
 & $(1^2,2,3,4,7)$
 & 9
 & $\frac{1}{2}(1,1,1,1) \ \ \frac{1}{4}(1,2,3,3) \ \ \frac{1}{7}(1,1,3,4)$
 & terminal
 \\ \hline
 21
 & $(1^2,3^2,4,6)$
 & 9
 & $\frac{1}{2}(1,1,1) \ \ \frac{1}{4}(1,3,3,2) \ \ \frac{1}{6}(1,1,3,4) \ \ \frac{1}{3}(1,1,1)$ 
 & klt
 \\ \hline
 22
 & $(1^2,3,4^2,5)$
 & 9
 & $\frac{1}{4}(1,3,1) \ \ \frac{1}{5}(1,1,3,4)$ 
 & terminal
 \\ \hline
 23
 & $(1,2,3^2,4,5)$
 & 9
 & $\frac{1}{2}(1,1,1) \ \ \frac{1}{3}(1,2,1,2) \ \ \frac{1}{4}(1,2,3,3) \ \ \frac{1}{5}(1,2,3,3)$ 
 & terminal
 \\ \hline 
 24
 & $(1,2^2,3^2,7)$
 & 9
 & $\frac{1}{2}(1,1,1) \ \ \frac{1}{3}(1,2,2,1) \ \ \frac{1}{7}(1,2,3,3)$ 
 & terminal
 \\ \hline
 25
 & $(1^3,3,5,9)$
 & 10
 & $\frac{1}{3}(1,1,2) \ \ \frac{1}{9}(1,1,3,5)$
 & terminal
 \\ \hline
 26
 & $(1^2,2^2,5,9)$
 & 10
 &  $\frac{1}{2}(1,1,1,1) \ \ \frac{1}{9}(1,2,2,5)$
 & terminal
 \\ \hline
 27
 & $(1^2,2,3,4,9)$
 & 10
 & $\frac{1}{2}(1,1,1,1) \ \ \frac{1}{3}(1,2,1) \ \ \frac{1}{9}(1,1,2,3,4) \ \ \frac{1}{4}(1,1,3,1)$
 & terminal
 \\ \hline
 28
 & $(1^2,2,3,5,8)$
 & 10
 & $ \frac{1}{2}(1,1,1,1) \ \ \frac{1}{3}(1,2,2,2) \ \ \frac{1}{9}(1,1,2,3,4) \ \ \frac{1}{8}(1,1,3,5)$
 & terminal
 \\ \hline
 29
 & $(1^2,3^2,5,7)$
 & 10
 & $ \frac{1}{3}(1,2,1) \ \ \frac{1}{7}(1,1,3,5)$
 & terminal
 \\ \hline
 30
 & $(1^2,3,5^3)$
 & 10
 & $ \frac{1}{3}(1,2,2,2) \ \ \frac{1}{5}(1,1,3)$
 & canonical
 \\ \hline
 31
 & $(1,2^2,3,5,7)$
 & 10
 & $ \frac{1}{2}(1,1,1,1) \ \ \frac{1}{3}(2,2,2,2)\ \ \frac{1}{7}(1,2,2,5)$
 & terminal
 \\ \hline
 32
 & $(1,2^2,5^3)$
 & 10
 & $ \frac{1}{2}(1,1,1,1) \ \ \frac{1}{5}(1,2,2)$
 &  klt 
 \\ \hline
 33
 & $(1,2,3^2,4,7)$
 & 10
 & $ \frac{1}{2}(1,1,1,1) \ \ \frac{1}{3}(2,1,1) \ \ \frac{1}{4}(1,3,3,3) \ \ \frac{1}{7}(1,2,3,4)$
 & terminal
 \\ \hline
 34
 & $(1,2,3,4,5^2)$
 & 10
 & $ \frac{1}{2}(1,1,1,1) \ \ \frac{1}{3}(2,1,2,2) \ \ \frac{1}{4}(1,3,1,1) \ \ \frac{1}{5}(1,2,3,4)$
 & terminal
 \\ \hline
 35
 & $(1^2,2,3,5,10)$
 & 11
 & $ \frac{1}{2}(1,1,1) \ \ \frac{1}{3}(1,1,2,1) \ \ \frac{1}{5}(1,2,3) \ \ \frac{1}{10}(1,2,3,5)$
 & terminal
 \\ \hline
 36
 & $(1,2^2,3,5,9)$
 & 11
 & $ \frac{1}{2}(1,1,1) \ \ \frac{1}{3}(2,2,2) \ \ \frac{1}{5}(2,2,3,4) \ \ \frac{1}{9}(1,2,3,5)$
 & klt
 \\ \hline
 37
 & $(1,2,3^2,5,8)$
 & 11
 & $ \frac{1}{2}(1,1,1) \ \ \frac{1}{3}(1,2,2) \ \ \frac{1}{5}(2,3,3,3) \ \ \frac{1}{8}(2,3,3,5)$
 & terminal
 \\ \hline
 38
 & $(1,2,3,4,5,7)$
 & 11
 & $ \frac{1}{2}(1,1,1) \ \ \frac{1}{3}(1,1,2,1) \ \ \frac{1}{4}(1,2,1,2) \ \ \frac{1}{5}(2,3,4,2) \ \ \frac{1}{7}(1,2,3,5)$
 & terminal
 \\ \hline
 39
 & $(1,2,3,5^2,6)$
 & 11
 & $ \frac{1}{2}(1,1,1) \ \ \frac{1}{3}(1,2,2) \ \ \frac{1}{5}(2,3,1) \ \ \frac{1}{6}(2,3,5,5)$
 & terminal
 \\ \hline
 40
 & $(1^3,4,6,11)$
 & 12
 & $ \frac{1}{2}(1,1,1,1) \ \ \frac{1}{11}(1,1,4,6)$
 & terminal
 \\ \hline
 41
 & $(1^2,2,3,6,11)$
 & 12
 & $ \frac{1}{2}(1,1,1,1) \ \ \frac{1}{3}(1,1,2,2) \ \ \frac{1}{11}(1,2,3,6)$
 & terminal
 \\ \hline
 42
 & $(1^2,2,4,5,11)$
 & 12
 & $ \frac{1}{2}(1,1,1,1) \ \ \frac{1}{5}(1,1,4,1) \ \ \frac{1}{11}(1,2,4,5)$
 & terminal
 \\ \hline
 43
 & $(1^2,3,4^2,11)$
 & 12
 & $ \frac{1}{4}(1,1,3,3) \ \ \frac{1}{11}(1,3,4,4)$
 & terminal
 \\ \hline
 44
 & $(1^2,3,4,6,9)$
 & 12
 & $ \frac{1}{3}(1,1,1) \ \ \frac{1}{2}(1,1,1,1) \ \ \frac{1}{9}(1,1,4,6)$
 & klt
 \\ \hline
 45
 & $(1,2^2,3,5,11)$
 & 12
 & $\frac{1}{2}(1,1,1,1) \ \ \frac{1}{5}(1,2,3,1) \ \ \frac{1}{11}(2,2,3,5)$
 & terminal
 \\ \hline
 46
 & $(1,2,3^2,4,11)$
 & 12
 & $\frac{1}{2}(1,1,1,1) \ \ \frac{1}{3}(1,2,1,2) \ \ \frac{1}{11}(2,3,3,4)$
 & terminal
 \\ \hline
 47
 & $(1,2,3,4,5,9)$
 & 12
 & $\frac{1}{2}(1,1,1,1) \ \ \frac{1}{3}(1,2,1,2) \ \ \frac{1}{5}(1,3,4,4) \ \ \frac{1}{9}(1,2,4,5)$
 & terminal
 \\ \hline
 48
 & $(1,2,3,5,6,7)$
 & 12
 & $\frac{1}{2}(1,1,1,1) \ \ \frac{1}{3}(1,2,2,1) \ \ \frac{1}{5}(1,3,1,2) \ \ \frac{1}{7}(1,2,3,6)$
 & terminal
 \\ \hline
 49
 & $(1,2,4,5^2,7)$
 & 12
 & $\frac{1}{2}(1,1,1,1) \ \ \frac{1}{5}(1,4,2) \ \ \frac{1}{7}(1,2,4,5)$
 &  terminal
 \\ \hline
 50
 & $(1,3^2,4^2,9)$
 & 12
 & $\frac{1}{3}(1,1,1) \ \ \frac{1}{4}(1,3,3,1) \ \ \frac{1}{9}(1,3,4,4)$
 & klt
 \\  \hline
 51
 & $(2^2,3^2,5,9)$
 & 12
 & $\frac{1}{2}(1,1,1,1) \ \ \frac{1}{3}(2,2,2) \ \ \frac{1}{5}(2,3,3,4) \ \ \frac{1}{9}(2,2,3,5)$
 & klt
 \\ \hline
 52
 & $(2^2,3,5^2,7)$
 & 12
 & $\frac{1}{2}(1,1,1,1) \ \ \frac{1}{5}(2,3,2) \ \ \frac{1}{7}(2,2,3,5)$
 & terminal
 \\ \hline
 53
 & $(2,3^2,4,5,7)$
 & 12
 & $\frac{1}{2}(1,1,1,1) \ \ \frac{1}{3}(2,4,2,1) \ \
 \frac{1}{5}(3,3,4,2) \ \ \frac{1}{7}(2,3,3,4)$
 & terminal
 \\ \hline
 54
 & $(1^2,3,4,5,12)$
 & 13
 & $\frac{1}{3}(1,1,2) \ \ \frac{1}{4}(1,3,1) \ \
 \frac{1}{5}(1,1,4,2) \ \ \frac{1}{12}(1,3,4,5)$
 & terminal
 \\ \hline
 55
 & $(1,2,3,4,5,11)$
 & 13
 & $\frac{1}{2}(1,1,1) \ \ \frac{1}{3}(2,1,2,2) \ \
 \frac{1}{4}(2,3,1,3) \ \ \frac{1}{5}(1,2,4,1) \ \ \frac{1}{11}(1,3,4,5)$
 &  terminal
 \\ \hline
 56
 & $(1,3^2,4,5,10)$
 & 13
 & $\frac{1}{3}(1,2,1) \ \ \frac{1}{4}(3,3,1,2) \ \
 \frac{1}{5}(1,3,4) \ \ \frac{1}{10}(1,3,4,5)$
 & terminal
 \\ \hline
 57
 & $(1,3,4^2,5,9)$
 & 13
 & $\frac{1}{3}(1,1,2) \ \ \frac{1}{4}(3,1,1) \ \
 \frac{1}{5}(1,4,4,4) \ \ \frac{1}{9}(1,3,4,5)$
 &  terminal
 \\ \hline
 58
 & $(1,3,4,5^2,8)$
 & 13
 & $\frac{1}{3}(1,2,2,2) \ \ \frac{1}{4}(3,1,1) \ \
 \frac{1}{5}(1,4,3) \ \ \frac{1}{8}(1,3,4,5)$
 & terminal
 \\ \hline
 59 
 & $(1,3,4,5,6,7)$
 & 13
 & $\frac{1}{3}(1,2,1) \ \ \frac{1}{4}(3,1,2,3) \ \
 \frac{1}{5}(1,4,1,2) \ \ \frac{1}{6}(3,4,5,1) \ \ \frac{1}{7}(1,3,4,5)$
 & terminal
 \\ \hline
 60
 & $(1^2,2,4,7,13)$
 & 14
 & $\frac{1}{2}(1,1,1,1) \ \ \frac{1}{4}(1,1,3,1) \ \
 \frac{1}{13}(1,2,4,7)$
 & terminal
 \\ \hline
 61
 & $(1,2^2,3,7,13)$
 & 14
 & $\frac{1}{2}(1,1,1,1) \ \ \frac{1}{3}(1,2,1,1) \ \
 \frac{1}{13}(2,2,3,7)$
 & terminal
 \\ \hline
 62
 & $(1,2,3,4,5,13)$
 & 14
 & $\frac{1}{2}(1,1,1,1) \ \ \frac{1}{3}(2,1,2,1) \ \
 \frac{1}{4}(1,3,1,1) \ \ \frac{1}{5}(1,2,3,3) \ \ \frac{1}{13}(2,3,4,5)$ 
 & terminal
 \\ \hline
 63
 & $(1,2,3,4,7,11)$
 & 14
 &  $\frac{1}{2}(1,1,1,1) \ \ \frac{1}{3}(1,1,1,2) \ \
 \frac{1}{4}(1,3,3,3)  \ \ \frac{1}{11}(1,2,4,7)$ 
 & terminal
 \\ \hline
 64
 & $(1,2,4,5,7,9)$
 & 14
 & $\frac{1}{2}(1,1,1,1) \ \ \frac{1}{4}(1,1,3,1) \ \
 \frac{1}{5}(1,2,2,4)  \ \ \frac{1}{9}(1,2,4,7)$
 & terminal
 \\ \hline
 65
 & $(1,2,4,7^3)$
 & 14
 & $\frac{1}{2}(1,1,1,1) \ \ \frac{1}{4}(1,3,3,3) \ \
 \frac{1}{7}(1,2,4)$
 & klt
 \\ \hline
 66
 & $(2^2,3^2,7,11)$
 & 14
 & $\frac{1}{2}(1,1,1,1) \ \ \frac{1}{3}(2,1,2) \ \
 \frac{1}{11}(2,2,3,7)$
 & terminal
 \\ \hline
 67
 & $(2^2,3,7^3)$
 & 14
 & $\frac{1}{2}(1,1,1,1) \ \ \frac{1}{3}(2,1,1,1) \ \
 \frac{1}{7}(2,2,3)$
 & klt
 \\ \hline
 68
 & $(2,3^2,4,5,11)$
 & 14
 & $\frac{1}{2}(1,1,1,1) \ \ \frac{1}{3}(1,2,2) \ \
 \frac{1}{4}(3,3,1,3) \ \ \frac{1}{5}(2,3,3,1) \ \ \frac{1}{11}(2,3,4,5)$
 & terminal
 \\ \hline
 69
 & $(2,3,4,5^2,9)$
 & 14
 &  $\frac{1}{2}(1,1,1,1) \ \ \frac{1}{3}(1,2,2) \ \
 \frac{1}{4}(3,1,1,1) \ \ \frac{1}{5}(2,3,4) \ \ \frac{1}{9}(2,3,4,5)$
 & terminal
 \\ \hline
 70 
 & $(2,3,4,5,7^2)$
 & 14
 &  $\frac{1}{3}(1,2,1,1) \ \ \frac{1}{4}(3,1,3,3) \ \
 \frac{1}{5}(2,3,7,7) \ \ \frac{1}{7}(2,3,4,5) $
 & terminal
\\ \hline
 71 
 & $(1,1,2,5,7,14)$
 & 15
 & $\frac{1}{2}(1,1,1) \ \ \frac{1}{7}(1,2,5) \ \
 \frac{1}{14}(1,2,5,7) $
 & terminal
\\ \hline
 72 
 & $(1,1,3,4,7,14)$
 & 15
 & $\frac{1}{2}(1,1,1) \ \ \frac{1}{4}(1,1,3,2) \ \
 \frac{1}{7}(1,3,4) \ \ \frac{1}{14}(1,3,4,7) $
 & terminal
\\ \hline
 73 
 & $(1,1,3,5,6,14)$
 & 15
 & $\frac{1}{3}(1,1,2,2) \ \ \frac{1}{6}(1,1,5,2) \ \
 \frac{1}{14}(1,3,5,6)$
 & terminal
 \\ \hline 
74
 & $(1,2^2,5,7,13)$
 & 15
 &  $\frac{1}{2}(1,1,1) \ \ \frac{1}{7}(2,2,5,6) \ \
 \frac{1}{13}(1,2,5,7)$
 & terminal
 \\ \hline
 75
 & $(1,2,3,4,7,13)$
 & 15
 &  $\frac{1}{2}(1,1,1) \ \ \frac{1}{4}(1,2,3,1) \ \
 \frac{1}{7}(2,3,4,6) \ \ \frac{1}{13}(1,3,4,7)$
 & terminal
 \\ \hline
 76
 & $(1,2,3,5^2,14)$
 & 15
 & $\frac{1}{2}(1,1,1) \ \ \frac{1}{5}(1,1,3,4) \ \
 \frac{1}{14}(2,3,5,5)$
 & terminal
 \\ \hline
 77 
 & $(1,2,3,5,6,13)$
 & 15
 & $\frac{1}{2}(1,1,1) \ \ \frac{1}{3}(1,2,2,1) \ \
 \frac{1}{6}(1,2,5,1) \ \ \frac{1}{13}(1,3,5,6)$
 & terminal
 \\ \hline
 78
 & $(1,2,3,5,7,12)$
 & 15
 & $\frac{1}{2}(1,1,1) \ \ \frac{1}{3}(1,2,2,1) \ \
 \frac{1}{7}(2,3,5,5) \ \ \frac{1}{12}(1,2,5,7)$
 & terminal
 \\ \hline
 79
 & $(1,2,5^2,7,10)$
 & 15
 & $\frac{1}{2}(1,1,1) \ \ \frac{1}{5}(1,2,2) \ \
 \frac{1}{7}(2,5,5,3) \ \ \frac{1}{10}(1,2,5,7)$
 & klt
 \\ \hline
 80
 & $(1,2,5,7^2,8)$
 & 15
 &  $\frac{1}{2}(1,1,1) \ \ \frac{1}{7}(2,5,1) \ \ \frac{1}{8}(1,2,5,7)$
 & terminal
 \\ \hline
 81
 & $(1,3^2,4,5,14)$
 & 15
 & $\frac{1}{3}(1,1,2,2) \ \ \frac{1}{2}(1,1,1) \ \ \frac{1}{4}(1,3,1,2) \ \ \frac{1}{14}(3,3,4,5)$
 & terminal
 \\ \hline
 82
 & $(1,3^2,4,7,12)$
 & 15
 & $\frac{1}{3}(1,1,1) \ \ \frac{1}{4}(1,3,3) \ \ \frac{1}{7}(3,3,4,5) \ \ \frac{1}{12}(1,3,4,7)$
 & canonical
 \\ \hline
 83
 & $(1,3,4^2,7,11)$
 & 15
 & $\frac{1}{4}(1,3,3) \ \ \frac{1}{7}(3,4,4,4) \ \ \frac{1}{11}(1,3,4,7)$
 & terminal
 \\ \hline
 84
 & $(1,3,4,5,6,11)$
 & 15
 & $\frac{1}{3}(1,1,2,2) \ \ \frac{1}{4}(1,1,2,3) \ \ \frac{1}{6}(1,4,5,5) \ \ \frac{1}{11}(1,3,5,6)$
 & terminal
 \\ \hline
 85
 & $(1,3,4,5,7,10)$
 & 15
 & $\frac{1}{2}(1,1,1) \ \ \frac{1}{4}(1,1,3,2) \ \ \frac{1}{5}(1,3,4,2) \ \ \frac{1}{7}(3,4,5,3) \ \ \frac{1}{10}(1,3,4,7)$
 & terminal
 \\ \hline
 86
 & $(1,3,4,6,7,9)$
 & 15
 & $\frac{1}{3}(1,1,1) \ \ \frac{1}{2}(1,1,1) \ \ \frac{1}{4}(1,2,3,1) \ \ \frac{1}{6}(1,4,1,3) \ \ \frac{1}{7}(3,4,6,2) \ \ \frac{1}{9}(1,3,4,7)$
 & klt
 \\ \hline
 87
 & $(1,3,4,7^2,8)$
 & 15
 & $\frac{1}{4}(1,3,3) \ \ \frac{1}{7}(3,4,1) \ \ \frac{1}{8}(1,3,4,7) $
 & terminal
 \\ \hline
 88
 & $(1,3,5^2,6,10)$
 & 15
 & $\frac{1}{5}(1,3,1) \ \ \frac{1}{3}(1,2,2,1) \ \ \frac{1}{6}(1,5,5,4) \ \ \frac{1}{10}(1,3,5,6) $
 & klt
 \\ \hline
 89
 & $(1,3,5,6,7,8)$
 & 15
 & $\frac{1}{3}(1,2,1,2) \ \ \frac{1}{6}(1,5,1,2) \ \ \frac{1}{7}(3,5,6,1) \ \ \frac{1}{8}(1,3,5,6) $
 & terminal
 \\ \hline
 90
 & $(3^2,4^2,5,11)$
 & 15
 & $\frac{1}{3}(1,1,2,2) \ \ \frac{1}{4}(3,1,3) \ \ \frac{1}{11}(3,3,4,5)$
 & terminal
 \\ \hline
 91
 & $(3^2,4,5^2,10)$
 & 15
 & $\frac{1}{3}(1,2,2,1) \ \ \frac{1}{2}(1,1,1) \ \ \frac{1}{4}(3,1,1,2) \ \ \frac{1}{5}(3,3,4) \ \ \frac{1}{10}(3,3,4,5)$
 & canonical
 \\ \hline
 92
 & $(1^2,2,5,8,15)$
 & 16
 & $\frac{1}{2}(1,1,1,1) \ \ \frac{1}{5}(1,2,3) \ \ \frac{1}{15}(1,2,5,8)$
 & terminal
 \\  \hline
 93
 & $(1^2,3,4,8,15)$
 & 16
 & $\frac{1}{3}(1,1,2) \ \ \frac{1}{4}(1,1,3,3) \ \ \frac{1}{15}(1,3,4,8)$
 & terminal
 \\  \hline
 94
 & $(1^2,4,5,6,15)$
 & 16
 & $\frac{1}{4}(1,1,1,3) \ \ \frac{1}{5}(1,4,1) \ \ \frac{1}{6}(1,1,5,3) \ \ \frac{1}{15}(1,4,5,6)$
 & terminal
 \\  \hline
 95
 & $(1,2,3,4,7,15)$
 & 16
 & $\frac{1}{2}(1,1,1,1) \ \ \frac{1}{3}(2,1,1) \ \ \frac{1}{7}(1,3,4,1) \ \ \frac{1}{15}(2,3,4,7)$
 & terminal
 \\  \hline
 96 
 & $(1,2,3,5,8,13)$
 & 16
 &  $\frac{1}{2}(1,1,1,1) \ \ \frac{1}{3}(2,2,2,1) \ \ \frac{1}{5}(2,3,3,3) \ \ \frac{1}{13}(1,2,5,8)$
 & terminal
 \\  \hline
 97
 & $(1,2,5^2,8,11)$
 & 16
 & $\frac{1}{2}(1,1,1,1) \ \ \frac{1}{5}(2,3,1) \ \ \frac{1}{11}(1,2,5,8) $
 & terminal
 \\  \hline
 98
 & $(1,2,5,7,8,9)$
 & 16
 & $\frac{1}{2}(1,1,1,1) \ \ \frac{1}{5}(2,2,3,4) \ \ \frac{1}{7}(1,5,1,2) \ \ \frac{1}{9}(1,2,5,8)$
 & terminal
 \\  \hline
 99
 & $(1,3^2,4,8,13)$
 & 16
 &  $\frac{1}{3}(1,2,1) \ \ \frac{1}{4}(1,3,3,1) \ \ \frac{1}{13}(1,3,4,8)$
 & terminal
 \\  \hline
 100
 & $(1,3,4,5,6,13)$
 & 16
 &  $\frac{1}{3}(1,2,1) \ \ \frac{1}{5}(3,4,1,3) \ \ \frac{1}{6}(1,3,5,1) \ \ \frac{1}{13}(1,4,5,6)$
 & terminal
 \\  \hline
 101
 & $(1,3,4,5,8,11)$
 & 16
 & $\frac{1}{3}(1,2,2,2) \ \ \frac{1}{4}(1,3,1,3) \ \ \frac{1}{5}(3,4,3,1) \ \ \frac{1}{11}(1,3,4,8)$
 & terminal
 \\  \hline
 102
 & $(1,4,5^2,6,11)$
 & 16
 & $\frac{1}{2}(1,1,1,1) \ \ \frac{1}{4}(4,1,1) \ \ \frac{1}{6}(1,5,5,5) \ \ \frac{1}{11}(1,4,5,6)$
 & terminal
 \\  \hline
 103
 & $(1,2,3,5,7,16)$
 & 17
 & $\frac{1}{2}(1,1,1) \ \ \frac{1}{3}(1,2,1,1) \ \ \frac{1}{5}(1,3,2,1) \ \ \frac{1}{7}(1,2,5,2) \ \ \frac{1}{16}(2,3,5,7)$
 & terminal
 \\  \hline
 104
 & $(1^2,2,6,9,17)$
 & 18
 & $\frac{1}{2}(1,1,1,1) \ \ \frac{1}{3}(1,1,2,2) \ \ \frac{1}{17}(1,2,6,9) $
 & terminal
 \\  \hline
 105
 & $(1^2,3,5,9,17)$
 & 18
 & $\frac{1}{3}(1,1,2,2) \ \ \frac{1}{5}(1,1,4,2) \ \ \frac{1}{17}(1,3,5,9) $
 & terminal
 \\  \hline
 106
 & $(1,2,3,4,9,17)$
 & 18
 & $\frac{1}{2}(1,1,1,1) \ \ \frac{1}{3}(1,2,1,2) \ \ \frac{1}{4}(1,3,1,1) \ \ \frac{1}{17}(2,3,4,9) $
 & terminal
 \\  \hline
 107
 & $(1,2,3,5,8,17)$
 & 18
 & $\frac{1}{2}(1,1,1,1) \ \ \frac{1}{5}(1,2,3,2) \ \ \frac{1}{8}(1,3,5,1) \ \ \frac{1}{17}(2,3,5,8) $
 & terminal
 \\  \hline
 108
 & $(1,2,3,5,9,16)$
 & 18
 & $\frac{1}{2}(1,1,1,1) \ \ \frac{1}{3}(1,2,2,1) \ \ \frac{1}{5}(1,2,4,1) \ \ \frac{1}{16}(1,3,5,9) $
 & terminal
 \\  \hline
 109
 & $(1,3,5^2,9,13)$
 & 18
 & $\frac{1}{3}(1,2,2,1) \ \ \frac{1}{5}(1,4,3) \ \ \frac{1}{13}(1,3,5,9) $
 & terminal
 \\  \hline
 110
 & $(2,3^2,5,8,15)$
 & 18
 & $\frac{1}{2}(1,1,1,1) \ \ \frac{1}{3}(2,2,2) \ \ \frac{1}{5}(2,3,3) \ \ \frac{1}{8}(3,3,5,7) \ \ \frac{1}{15}(2,3,5,8) $
 & klt
 \\  \hline
 111
 & $(2,3,4,5,9,12)$
 & 18
 & $\frac{1}{2}(1,1,1) \ \ \frac{1}{3}(2,1,2) \ \ \frac{1}{4}(3,1,1) \ \ \frac{1}{5}(2,4,4,2) \ \ \frac{1}{12}(2,3,4,5,9) $
 & terminal
 \\  \hline
 112
 & $(2,3,4,7,9,11)$
 & 18
 & $\frac{1}{2}(1,1,1,1) \ \ \frac{1}{3}(2,1,1,2) \ \ \frac{1}{4}(3,3,1,3) \ \ \frac{1}{7}(2,3,2,4) \ \ \frac{1}{11}(2,3,4,9) $
 & terminal
 \\  \hline
 113
 & $(2,3,5^2,8,13)$
 & 18
 &  $\frac{1}{2}(1,1,1,1) \ \ \frac{1}{5}(2,3,3) \ \ \frac{1}{8}(3,5,5,5) \ \ \frac{1}{13}(2,3,5, 8)  $
 & terminal
 \\  \hline
 114
 & $(2,3,5,8,9^2)$
 & 18
 &  $\frac{1}{2}(1,1,1,1) \ \ \frac{1}{3}(2,2,2) \ \ \frac{1}{5}(2,3,4,4) \ \ \frac{1}{8}(3,5,1,1) \ \ \frac{1}{9}(2,3,5,8) $
 & klt
 \\  \hline
 115
 & $(1,3,4,5,7,18)$
 &  19
 & $\frac{1}{3}(1,2,1) \ \ \frac{1}{2}(1,1,1) \ \ \frac{1}{4}(1,1,3,2) \ \ \frac{1}{5}(1,3,2,3) \ \ \frac{1}{7}(1,3,4,4) \ \ \frac{1}{18}(3,4,5,7) $
 & terminal
 \\  \hline
 116
 & $(1^2,4,5,10,19)$
 & 20
 & $\frac{1}{2}(1,1,1,1) \ \ \frac{1}{5}(1,1,4,4) \ \ \frac{1}{19}(1,4,5,10) $
 & terminal
 \\  \hline
 117
 & $(1,2,3,5,10,19)$
 & 20
 & $\frac{1}{2}(1,1,1,1) \ \ \frac{1}{3}(1,2,1,1) \ \ \frac{1}{5}(1,2,3,4) \ \ \frac{1}{19}(2,3,5,10) $
 & terminal
 \\  \hline
 118
 & $(1,2,4,5,9,19)$
 & 20
 & $\frac{1}{2}(1,1,1,1) \ \ \frac{1}{9}(1,4,5,1) \ \ \frac{1}{19}(2,4,5,9)$
 & terminal
 \\  \hline
 119
 & $(1,2,5,6,7,19)$
 & 20
 & $\frac{1}{2}(1,1,1,1) \ \ \frac{1}{6}(1,5,1,2) \ \ \frac{1}{7}(1,2,5,5) \ \ \frac{1}{19}(2,5,6,7)$
 & terminal
 \\  \hline
 120
 & $(2,3^2,5,10,17)$
 & 20
 & $\frac{1}{2}(1,1,1,1) \ \ \frac{1}{3}(2,1,2) \ \ \frac{1}{5}(2,3,3,2) \ \ \frac{1}{17}(2,3,5,10)$
 & terminal
 \\  \hline
 121
 & $(2,3,4,5,9,17)$
 & 20
 & $\frac{1}{2}(1,1,1,1) \ \ \frac{1}{3}(1,2,2) \ \ \frac{1}{9}(3,4,5,8) \ \ \frac{1}{17}(2,4,5,9)$
 & terminal
 \\  \hline
 122
 & $(2,3,5,6,7,17)$
 & 20
 & $\frac{1}{2}(1,1,1,1) \ \ \frac{1}{3}(2,1,2) \ \ \frac{1}{6}(3,5,1,5) \ \ \frac{1}{7}(2,3,5,3) \ \ \frac{1}{17}(2,5,6,7)$
 & terminal
 \\  \hline
 123
 & $(2,3,5,9,10,11)$
 & 20
 & $\frac{1}{2}(1,1,1,1) \ \ \frac{1}{3}(2,1,2) \ \ \frac{1}{9}(3,5,1,2) \ \ \frac{1}{11}(2,3,5,10)$
 & terminal
 \\  \hline
 124
 & $(2,4,5^2,9,15)$
 & 20
 & $\frac{1}{2}(1,1,1,1) \ \ \frac{1}{5}(2,4,4) \ \ \frac{1}{9}(4,5,5,6) \ \ \frac{1}{15}(2,4,5,9)$
 & klt
 \\  \hline
 125
 & $(2,4,5,9^2,11)$
 & 20
 & $\frac{1}{2}(1,1,1,1) \ \ \frac{1}{9}(4,5,2) \ \ \frac{1}{11}(2,4,5,9)$
 & terminal
 \\  \hline
 126
 & $(2,5^2,6,7,15)$
 & 20
 & $\frac{1}{2}(1,1,1,1) \ \ \frac{1}{5}(2,1,2) \ \ \frac{1}{6}(5,5,1,3) \ \ \frac{1}{7}(2,5,5,1) \ \ \frac{1}{15}(2,5,6,7)$
 & klt
 \\  \hline
 127
 & $(2,5,6,7^2,13)$
 & 20
 & $\frac{1}{6}(5,1,1,1) \ \ \frac{1}{7}(2,5,6) \ \ \frac{1}{13}(2,5,6,7)$
 & terminal
 \\  \hline
 128
 & $(2,5,6,7,9,11)$
 & 20
 &  $\frac{1}{2}(1,1,1,1) \ \ \frac{1}{3}(2,1,2) \ \ \frac{1}{6}(5,1,3,5) \ \ \frac{1}{7}(2,5,2,4) \ \ \frac{1}{9}(5,6,7,2) \ \ \frac{1}{11}(2,5,6,7)$
 & terminal
 \\  \hline
 129
 & $(1^2,3,7,10,20)$
 & 21
 & $\frac{1}{10}(1,3,7) \ \ \frac{1}{20}(1,3,7,10)$
 & terminal
 \\  \hline
 130
 & $(1^2,5,7,8,20)$
 & 21
 & $\frac{1}{5}(1,2,3) \ \ \frac{1}{8}(1,1,7,4) \ \ \frac{1}{20}(1,5,7,8)$
 & terminal
 \\  \hline
 131
 & $(1,2,3,7,9,10)$
 & 21
 & $\frac{1}{2}(1,1,1) \ \ \frac{1}{3}(1,2,1,1) \ \ \frac{1}{9}(1,2,7,1) \ \ \frac{1}{10}(2,3,7,9)$
 & terminal
 \\  \hline
 132
 & $(1,2,3,7,10,19)$
 & 21
 & $\frac{1}{2}(1,1,1) \ \ \frac{1}{10}(2,3,7,9) \ \ \frac{1}{19}(1,3,7,10)$
 & terminal
 \\  \hline
 133
 & $(1,2,5,7,8,19)$
 & 21
 & $\frac{1}{2}(1,1,1) \ \ \frac{1}{5}(2,2,3,4) \ \ \frac{1}{8}(1,2,7,3) \ \ \frac{1}{19}(1,5,7,8)$
 & terminal
 \\  \hline
 134
 & $(1,3^2,7,10,18)$
 & 21
 & $\frac{1}{3}(1,1,1) \ \ \frac{1}{10}(3,3,7,8) \ \ \frac{1}{18}(1,3,7,10)$
 & canonical
 \\  \hline
 135
 & $(1,3,4,7,10,17)$
 & 21
 & $\frac{1}{4}(3,3,2,1) \ \ \frac{1}{10}(3,4,7,7) \ \ \frac{1}{17}(1,3,7,10)$
 & terminal
 \\  \hline
 136
 & $(1,3,5,6,7,20)$
 & 21
 & $\frac{1}{5}(3,1,2) \ \ \frac{1}{6}(1,5,1,2) \ \ \frac{1}{20}(3,5,6,7)$
 & terminal
 \\  \hline
 137
 & $(1,3,5,7,8,18)$
 & 21
 & $\frac{1}{3}(1,2,1,2) \ \ \frac{1}{5}(3,2,3,3) \ \ \frac{1}{8}(1,3,7,2) \ \ \frac{1}{18}(1,5,7,8)$
 & terminal
 \\  \hline
 138
 & $(1,3,5,7,10,16)$
 & 21
 & $\frac{1}{5}(3,2,1) \ \ \frac{1}{10}(3,5,7,6) \ \ \frac{1}{16}(1,3,7,10) $
 & terminal
 \\  \hline
 139
 & $(1,3,6,7,10,15)$
 & 21
 & $ \frac{1}{3}(1,1,1) \ \ \frac{1}{6}(1,1,4,3) \ \ \frac{1}{10}(3,6,7,5) \ \ \frac{1}{15}(1,3,7,10) $
 & klt
 \\  \hline
 140
 & $(1,3,7^2,10,14)$
 & 21
 & $ \frac{1}{7}(1,3,3) \ \ \frac{1}{10}(3,7,7,4) \ \ \frac{1}{14}(1,3,7,10) $
 & klt
 \\  \hline
 141
 & $(1,3,7,9,10,12)$
 & 21
 & $ \frac{1}{3}(1,1,1) \ \ \frac{1}{9}(1,7,1,3) \ \ \frac{1}{10}(3,7,9,2) \ \ \frac{1}{12}(1,3,7,10) $
 & canonical
 \\  \hline
 142
 & $(1,3,7,10^2,11)$
 & 21
 & $ \frac{1}{10}(3,7,1) \ \ \frac{1}{11}(1,3,7,10) $
 & terminal
 \\  \hline
 143
 & $(1,4,5,7,8,17)$
 & 21
 & $ \frac{1}{4}(1,3,1) \ \ \frac{1}{5}(4,2,3,2) \ \ \frac{1}{8}(1,4,7,1) \ \ \frac{1}{17} (1,5,7,8) $
 & terminal
 \\  \hline
 144
 & $(1,5^2,7,8,16)$
 & 21
 &  $ \frac{1}{5}(2,3,1) \ \ \frac{1}{5}(4,2,3,2) \ \ \frac{1}{8}(1,4,7,1) \ \ \frac{1}{17} (1,5,7,8) $
 & terminal
 \\  \hline
 145
 & $(1,5,7^2,8,14)$
 & 21
 & $ \frac{1}{5}(2,2,3,4) \ \ \frac{1}{7}(1,5,1) \ \ \frac{1}{8}(1,7,7,6) \ \ \frac{1}{14} (1,5,7,8) $
 & klt
 \\  \hline
 146
 & $(1,5,7,8^2,13)$
 & 21
 & $ \frac{1}{5}(2,3,3,3) \ \ \frac{1}{8}(1,7,5) \ \ \frac{1}{13}(1,5,7,8) $
 & terminal
 \\  \hline
 147
 & $(1,5,7,8,10,11)$
 & 21
 & $ \frac{1}{5}(2,3,1) \ \ \frac{1}{8}(1,7,2,3) \ \ \frac{1}{10}(5,7,8,1) \ \ \frac{1}{11}(1,5,7,8) $
 & terminal
 \\  \hline
 148
 & $(1^2,3,7,11,21)$
 & 22
 &  $ \frac{1}{3}(1,1,2) \ \ \frac{1}{7}(1,3,4) \ \ \frac{1}{21}(1,3,7,11)  $
 & terminal
 \\  \hline
 149
 & $(1,2,3,7,11,20)$
 &  22
 & $ \frac{1}{2}(1,1,1,1) \ \ \frac{1}{3}(2,1,2,2) \ \ \frac{1}{7}(2,3,4,6) \ \ \frac{1}{20}(1,3,7,11) $
 & terminal
 \\  \hline
 150
 & $(1,2,4,5,11,21)$
 & 22
 & $ \frac{1}{2}(1,1,1,1) \ \ \frac{1}{4}(1,1,3,1) \ \ \frac{1}{5}(1,4,1,1) \ \ \frac{1}{21}(2,4,5,11) $
 & terminal
 \\  \hline
 151
 & $(1,3^2,7,11,19)$
 & 22
 &  $ \frac{1}{3}(1,2,1) \ \ \frac{1}{7}(3,3,4,5) \ \ \frac{1}{19}(1,3,7,11) $
 & terminal
 \\  \hline
 152
 & $(1,3,7^2,11,15)$
 & 22
 &  $ \frac{1}{3}(1,1,2) \ \ \frac{1}{7}(3,4,1) \ \ \frac{1}{15}(1,3,7,11) $
 & terminal
 \\  \hline
 153
 & $(1,3,7,11^3)$
 & 22
 & $ \frac{1}{11}(1,3,7) \ \ \frac{1}{3}(1,2,2,2) \ \ \frac{1}{7}(3,4,4,4) $
 & klt
 \\  \hline
 154
 & $(2,4,5^2,11,17)$
 & 22
 & $ \frac{1}{2}(1,1,1,1) \ \ \frac{1}{4}(1,1,3,1) \ \ \frac{1}{5}(4,1,2) \ \ \frac{1}{17}(2,4,5,11) $
 & terminal
 \\  \hline
 155
 & $(2,4,5,9,11,13)$
 & 22
 & $ \frac{1}{2}(1,1,1,1) \ \ \frac{1}{4}(1,1,3,1) \ \ \frac{1}{5}(4,4,1,3) \ \ \frac{1}{9}(2,5,2,4) \ \ \frac{1}{13}(2,4,5,11) $
 & terminal
 \\  \hline
 156
 & $(2,4,5,11^3)$
 & 22
 & $ \frac{1}{2}(1,1,1,1) \ \ \frac{1}{4}(1,3,3,3) \ \ \frac{1}{5}(4,1,1,1) \ \ \frac{1}{11}(2,4,5) $
 & klt
 \\  \hline
 157
 & $(1^2,3,8,12,23)$
 &  24
 & $ \frac{1}{3}(1,1,2,2) \ \ \frac{1}{4}(1,1,3,3) \ \ \frac{1}{23}(1,3,8,12) $
 & terminal
 \\  \hline
 158
 & $(1^2,6,8,9,23)$
 & 24
 & $ \frac{1}{3}(1,1,2,2) \ \ \frac{1}{2}(1,1,1,1) \ \ \frac{1}{9}(1,1,8,5) \ \ \frac{1}{23}(1,6,8,9) $
 & terminal
 \\  \hline
 159
 & $(1,2,3,7,12,23)$
 & 24
 & $ \frac{1}{2}(1,1,1,1) \ \ \frac{1}{7}(1,2,5,2) \ \ \frac{1}{23}(2,3,7,12) $
 & terminal
 \\  \hline
 160
 & $(1,2,3,8,11,23)$
 & 24
 & $ \frac{1}{2}(1,1,1,1) \ \ \frac{1}{11}(1,3,8,7) \ \ \frac{1}{23}(2,3,8,11) $
 & terminal
 \\  \hline
 161
 & $(1,3,4,5,12,23)$
 & 24
 & $ \frac{1}{3}(1,1,2,2) \ \ \frac{1}{4}(1,3,1,3) \ \ \frac{1}{5}(1,3,2,3) \ \ \frac{1}{23}(3,4,5,12) $
 & terminal
 \\  \hline
 162
 & $(1,3,4,7,10,23)$
 & 24
 & $ \frac{1}{7}(1,4,3,2) \ \ \frac{1}{10}(1,3,7,3) \ \ \frac{1}{23}(3,4,7,10) $
 & terminal
 \\  \hline
 163
 & $(1,3,5,7,8,23)$
 & 24
 & $ \frac{1}{5}(1,3,2,3,3) \ \ \frac{1}{7}(1,5,1,2) \ \ \frac{1}{23}(3,5,7,8) $
 & canonical
 \\  \hline
 164
 & $(1,3,7,8,12,17)$
 & 24
 & $ \frac{1}{3}(1,1,2,2) \ \ \frac{1}{7}(1,1,5,3) \ \ \frac{1}{4}(1,3,3,1) \ \ \frac{1}{17}(1,3,8,12) $
 & terminal
 \\  \hline
 165
 & $(1,4,5,6,9,23)$
 & 24
 & $ \frac{1}{2}(1,1,1,1) \ \ \frac{1}{5}(1,1,4,3) \ \ \frac{1}{9}(1,4,5,5) \ \ \frac{1}{23}(4,5,6,9) $
 & terminal
 \\  \hline
 166
 & $(2,3^2,8,11,21)$
 & 24
 & $ \frac{1}{2}(1,1,1,1) \ \ \frac{1}{3}(2,2,2) \ \ \frac{1}{11}(3,3,8,10) \ \ \frac{1}{21}(2,3,8,11) $
 & klt
 \\  \hline
 167
 & $(2,3,7^2,12,17)$
 & 24
 & $ \frac{1}{2}(1,1,1,1) \ \ \frac{1}{3}(2,1,1,2) \ \ \frac{1}{7}(2,5,3) \ \ \frac{1}{17}(2,3,7,12) $
 & terminal
 \\  \hline
 168
 & $(2,3,7,8,11,17)$
 & 24
 &  $ \frac{1}{2}(1,1,1,1) \ \ \frac{1}{7}(2,1,4,3) \ \ \frac{1}{11}(3,7,8,6) \ \ \frac{1}{17}(2,3,8,11) $
 & terminal
 \\  \hline
 169
 & $(2,3,7,11,12,13)$
 & 24
 & $ \frac{1}{2}(1,1,1,1) \ \ \frac{1}{3}(2,1,2,1) \ \ \frac{1}{7}(2,4,5,6) \ \ \frac{1}{11}(3,7,1,2) \ \ \frac{1}{13}(2,3,7,12) $
 & terminal
 \\  \hline
 170
 & $(2,3,8,11^2,13)$
 & 24
 & $\frac{1}{2}(1,1,1,1) \ \  \frac{1}{11}(3,8,2) \ \ \frac{1}{13}(2,3,8,11)$
 & terminal
 \\  \hline
 171
 & $(3^2,4,7,10,21)$
 & 24
 & $\frac{1}{3}(1,1,1) \ \  \frac{1}{2}(1,1,1,1) \ \ \frac{1}{7}(3,4,3) \ \ \frac{1}{10}(3,3,7,1) \ \  \frac{1}{21}(3,4,7,10)$
 & canonical
 \\  \hline
 172
 & $(3,4,5^2,12,19)$
 & 24
 & $\frac{1}{4}(3,1,1,3) \ \  \frac{1}{3}(1,2,2,1) \ \ \frac{1}{5}(3,2,4) \ \ \frac{1}{19}(3,4,5,12)$
 & terminal
 \\  \hline
 173
 & $(3,4,5,7,10,19)$
 & 24
 & $\frac{1}{2}(1,1,1,1) \ \  \frac{1}{5}(3,2,4) \ \ \frac{1}{10}(3,5,7,9) \ \ \frac{1}{19}(3,4,7,10)$
 & terminal
 \\  \hline
 174
 & $(3,4,5,7,12,17)$
 & 24
 & $\frac{1}{4}(3,1,3,1) \ \  \frac{1}{3}(1,2,1,2) \ \ \frac{1}{5}(3,2,2,2) \ \ \frac{1}{17}(3,4,5,12) \ \  \frac{1}{7}(4,5,5,3)$
 & terminal
 \\  \hline
 175
 & $(3,4,7^2,10,17)$
 & 24
 & $\frac{1}{2}(1,1,1,1) \ \  \frac{1}{10}(3,7,7,7) \ \ \frac{1}{7}(4,3,3) \ \ \frac{1}{17}(3,4,7,10) $
 & terminal
 \\  \hline
 176
 & $(3,6,7^2,8,17)$
 & 24
 & $\frac{1}{2}(1,1,1,1) \ \  \frac{1}{7}(6,1,3) \ \ \frac{1}{3}(1,1,2,2) \ \ \frac{1}{17}(3,6,7,8)$
 & terminal
 \\  \hline
 177
 & $(4,5^2,6,9,19)$
 & 24
 & $\frac{1}{2}(1,1,1,1) \ \  \frac{1}{5}(1,4,4) \ \ \frac{1}{3}(1,2,2,1) \ \ \frac{1}{9}(4,5,5,1) \ \  \frac{1}{19}(4,5,6,9)$
 & terminal
 \\  \hline
 178
 & $(1,4,5,7,9,24)$
 & 25
 & $\frac{1}{4}(1,3,1) \ \  \frac{1}{7}(1,5,2,3) \ \ \frac{1}{3}(1,2,1) \ \ \frac{1}{24}(4,5,7,9)$
 & terminal
 \\  \hline
 179
 & $(1^2,5,7,13,25)$
 & 26
 & $\frac{1}{5}(1,2,3) \ \  \frac{1}{7}(1,1,6,4) \ \ \frac{1}{25}(1,5,7,13)$
 & terminal
 \\  \hline
 180
 & $(1,2,3,8,13,25)$
 & 26
 & $\frac{1}{2}(1,1,1,1) \ \  \frac{1}{3}(1,2,1,1) \ \ \frac{1}{8}(1,3,5,1) \ \ \frac{1}{25}(2,3,8,13)$
 & terminal
 \\  \hline
 181
 & $(1,2,5,6,13,25)$
 & 26
 & $\frac{1}{2}(1,1,1,1) \ \  \frac{1}{5}(2,1,3) \ \ \frac{1}{6}(1,5,1,1) \ \ \frac{1}{25}(2,5,6,13)$
 & terminal
 \\  \hline
 182
 & $(1,2,5,7,13,24)$
 & 26
 & $\frac{1}{2}(1,1,1,1) \ \  \frac{1}{5}(2,2,3,4) \ \ \frac{1}{7}(1,2,6,3) \ \ \frac{1}{24}(1,5,7,13)$
 & terminal
 \\  \hline
 183
 & $(1,5^2,7,13,21)$
 & 26
 & $\frac{1}{5}(2,3,1) \ \  \frac{1}{7}(1,5,6) \ \ \frac{1}{21}(1,5,7,13) $
 & terminal
 \\  \hline
 184
 & $(1,5,7^2,13,19)$
 & 26
 & $\frac{1}{5}(2,2,3,4) \ \  \frac{1}{7}(1,6,5) \ \ \frac{1}{19}(1,5,7,13)$
 & terminal
 \\  \hline
 185
 & $(1,5,7,13^3)$
 & 26
 & $\frac{1}{5}(2,3,3,3) \ \  \frac{1}{7}(1,6,6,6) \ \ \frac{1}{13}(1,5,7)$
 & klt
 \\  \hline
 186
 & $(2,3^2,8,13,23)$
 & 26
 & $\frac{1}{2}(1,1,1,1) \ \  \frac{1}{3}(2,1,2) \ \ \frac{1}{8}(3,3,5,7) \ \ \frac{1}{23}(2,3,8,13)$
 & terminal
 \\  \hline
 187
 & $(2,3,8,9,13,17)$
 & 26
 & $\frac{1}{2}(1,1,1,1) \ \  \frac{1}{3}(2,1,2) \ \ \frac{1}{8}(3,1,5,1) \ \ \frac{1}{9}(2,3,4,8) \ \ \frac{1}{17}(2,3,8,13)$
 & terminal
 \\  \hline
 188
 & $(2,3,8,13^3)$
 & 26
 & $\frac{1}{2}(1,1,1,1) \ \  \frac{1}{13}(2,3,8) \ \ \frac{1}{8}(3,5,5,5) \ \ \frac{1}{3}(2,1,1,1)$
 & klt
 \\  \hline
 189
 & $(1,2,5,9,11,26)$
 & 27
 & $\frac{1}{2}(1,1,1,1) \ \  \frac{1}{5}(1,4,1,1) \ \ \frac{1}{11}(1,2,9,4) \ \ \frac{1}{26}(2,5,9,11)$
 & terminal
 \\  \hline
 190
 & $(1^2,4,9,14,27)$
 & 28
 & $\frac{1}{2}(1,1,1,1) \ \  \frac{1}{9}(1,4,5) \ \ \frac{1}{27}(1,4,9,14)$
 & terminal
 \\  \hline
 191
 & $(1,3,4,7,14,27)$
 & 28
 & $\frac{1}{2}(1,1,1,1) \ \  \frac{1}{3}(1,1,2) \ \ \frac{1}{7}(1,3,4,6) \ \ \frac{1}{27}(3,4,7,14)$
 & terminal
 \\  \hline
 192
 & $(1,3,4,9,14,25)$
 & 28
 & $\frac{1}{3}(1,2,1) \ \  \frac{1}{2}(1,1,1,1) \ \ \frac{1}{9}(3,4,5,7) \ \ \frac{1}{25}(1,4,9,14) $
 & terminal
 \\  \hline
 193
 & $(1,4,6,7,11,27)$
 & 28
 & $\frac{1}{2}(1,1,1,1) \ \  \frac{1}{3}(1,1,2) \ \ \frac{1}{6}(1,1,5,3) \ \ \frac{1}{11}(1,4,7,5) \ \ \frac{1}{27}(4,6,7,11)$
 & terminal
 \\  \hline
 194
 & $(1,4,7,9,14,21)$
 & 28
 & $\frac{1}{2}(1,1,1,1) \ \  \frac{1}{3}(1,1,2) \ \ \frac{1}{9}(4,7,5,3) \ \ \frac{1}{7}(1,4,2) \ \ \frac{1}{21}(4,6,7,11)$
 & klt
 \\  \hline
 195
 & $(1,4,9^2,14,19)$
 & 28
 & $\frac{1}{2}(1,1,1,1) \ \  \frac{1}{9}(4,5,1) \ \ \frac{1}{19}(1,4,9,14)$
 & terminal
 \\  \hline
 196
 & $(4,6,7^2,11,21)$
 & 28
 & $\frac{1}{2}(1,1,1,1) \ \  \frac{1}{7}(4,6,4) \ \ \frac{1}{3}(4,7,11) \ \ \frac{1}{11}(4,7,7,10) \ \ \frac{1}{21}(4,6,7,11)$
 & klt
 \\  \hline
 197
 & $(4,6,7,11^2,17)$
 & 28
 & $\frac{1}{2}(1,1,1,1) \ \  \frac{1}{6}(1,5,5,5) \ \ \frac{1}{11}(4,7,6) \ \ \frac{1}{17}(4,6,7,11)$
 & terminal
 \\  \hline
 198
 & $(1^2,6,8,15,29)$
 & 30
 & $\frac{1}{2}(1,1,1,1) \ \  \frac{1}{8}(1,1,7,5) \ \ \frac{1}{3}(1,1,2,2) \ \ \frac{1}{29}(1,6,8,15)$ 
 & terminal
 \\  \hline
 199
 & $(1,2,3,10,15,29)$
 & 30
 & $\frac{1}{2}(1,1,1,1) \ \  \frac{1}{3}(1,2,1,2) \ \ \frac{1}{5}(1,2,3,4) \ \ \frac{1}{29}(2,3,10,15)$ 
 & terminal
 \\  \hline
 200
 & $(1,2,6,7,15,29)$
 & 30
 & $\frac{1}{2}(1,1,1,1) \ \  \frac{1}{3}(1,2,1,2) \ \ \frac{1}{7}(1,6,1,1) \ \ \frac{1}{29}(2,6,7,15)$ 
 & terminal
 \\  \hline
 201
 & $(1,5,6,8,11,29)$
 & 30
 & $\frac{1}{8}(1,5,3,5) \ \  \frac{1}{29}(5,6,8,11) \ \ \frac{1}{11}(1,6,4,8) \ \ \frac{1}{19}(1,6,8,15) \ \ \frac{1}{2}(1,1,1,1)$ 
 & terminal
 \\  \hline
 202
 & $(1,5,6,8,15,25)$
 & 30
 & $\frac{1}{5}(1,1,3) \ \  \frac{1}{2}(1,1,1,1) \ \ \frac{1}{8}(1,5,7,1) \ \ \frac{1}{25}(1,6,8,15) \ \ \frac{1}{3}(1,2,2,1)$ 
 & canonical
 \\  \hline
 203
 & $(1,6,8,11,15,19)$
 & 30
 & $\frac{1}{2}(1,1,1,1) \ \  \frac{1}{3}(1,2,2,1) \ \ \frac{1}{8}(1,3,7,3) \ \ \frac{1}{11}(1,6,4,8) \ \ \frac{1}{19}(1,6,8,15)$ 
 & terminal
 \\  \hline
 204
 & $(2,3,7,10,15,23)$
 & 30
 & $\frac{1}{2}(1,1,1,1) \ \  \frac{1}{3}(2,1,1,2) \ \ \frac{1}{7}(3,3,1,2) \ \ \frac{1}{5}(2,3,1,3) \ \ \frac{1}{23}(2,3,10,15)$ 
 & terminal
 \\  \hline
 205
 & $(2,5,6,7,15,25)$
 & 30
 & $\frac{1}{2}(1,1,1,1) \ \  \frac{1}{5}(2,1,2) \ \ \frac{1}{3}(2,2,1,4) \ \ \frac{1}{7}(5,6,1,4) \ \ \frac{1}{25}(2,6,7,15)$ 
 & klt
 \\  \hline
 206
 & $(2,6,7^2,15,23)$
 & 30
 & $\frac{1}{2}(1,1,1,1) \ \  \frac{1}{7}(6,1,2) \ \ \frac{1}{3}(2,1,1,2) \ \ \frac{1}{23}(2,6,7,15)$ 
 & terminal
 \\  \hline
 207
 & $(3,5,6,8,11,27)$
 & 30
 & $\frac{1}{3}(2,3,2) \ \  \frac{1}{27}(5,6,8,11) \ \ \frac{1}{2}(1,1,1,1) \ \ \frac{1}{11}(3,5,6,5) \ \ \frac{1}{15}(5,9,1,2) \ \ \frac{1}{17}(2,5,9,16)$ 
 & klt
 \\  \hline
 208
 & $(5^2,6,8,11,25)$
 & 30
 & $\frac{1}{5}(1,3,1) \ \  \frac{1}{25}(5,6,8,11) \ \ \frac{1}{8}(5,5,3,1) \ \ \frac{1}{2}(1,1,1,1) \ \ \frac{1}{11}(5,5,6,3)$ 
 & canonical
 \\  \hline
 209
 & $(5,6,8,11^2,19)$
 & 30
 & $\frac{1}{2}(1,1,1,1) \ \  \frac{1}{8}(5,3,3,3) \ \ \frac{1}{11}(5,6,8) \ \ \frac{1}{19}(5,6,8,11)$ 
 & terminal
 \\  \hline
 210
 & $(5,6,8,11,15^2)$
 & 30
 & $\frac{1}{5}(1,3,1) \ \  \frac{1}{2}(1,1,1,1) \ \ \frac{1}{8}(5,3,7,7) \ \ \frac{1}{3}(5,8,11) \ \ \frac{1}{15}(5,6,8,11) \ \ \frac{1}{11}(5,6,4,4)$ 
 & klt
 \\  \hline
 211
 & $(1,2,5,9,16,31)$
 & 32
 & $\frac{1}{2}(1,1,1,1) \ \  \frac{1}{5}(1,4,1,1) \ \ \frac{1}{9}(1,2,7,4) \ \ \frac{1}{31}(2,5,9,16) $ 
 & terminal
 \\  \hline
 212
 & $(2,3,5,9,16,29)$
 & 32
 & $\frac{1}{2}(1,1,1,1) \ \  \frac{1}{3}(2,1,2) \ \ \frac{1}{5}(3,4,1,4) \ \ \frac{1}{9}(2,3,7,2) \ \ \frac{1}{29}(2,5,9,16)$ 
 & terminal
 \\  \hline
 213
 & $(2,5^2,9,16,27)$
 & 32
 & $\frac{1}{2}(1,1,1,1) \ \  \frac{1}{5}(4,1,2) \ \ \frac{1}{9}(2,5,7) \ \ \frac{1}{27}(2,5,9,16)$ 
 & terminal
 \\  \hline
 214
 & $(2,5,9^2,16,23)$
 & 32
 & $\frac{1}{2}(1,1,1,1) \ \  \frac{1}{5}(4,4,1,3) \ \ \frac{1}{9}(2,7,5) \ \ \frac{1}{23}(2,5,9,16)$ 
 & terminal
 \\  \hline
 215
 & $(2,5,9,15,16,17)$
 & 32
 &  $\frac{1}{2}(1,1,1,1) \ \  \frac{1}{5}(4,1,2) \ \ \frac{1}{3}(2,1,2) \ \ \frac{1}{9}(2,6,7,8) \ \ \frac{1}{15}(5,9,1,2) \ \ \frac{1}{17}(2,5,9,16)$ 
 & terminal
 \\  \hline
 216
 & $(1,3,5,11,14,32)$
 & 33
 & $\frac{1}{5}(1,1,4,2) \ \  \frac{1}{2}(1,1,1) \ \ \frac{1}{14}(1,3,11,4) \ \ \frac{1}{32}(3,5,11,14)$ 
 & terminal
 \\  \hline
 217
 & $(3^2,5,11,14,30)$
 & 33
 & $\frac{1}{3}(2,2,2) \ \  \frac{1}{5}(3,1,4) \ \ \frac{1}{30}(3,5,11,14) \ \ \frac{1}{2}(1,1,1)  \ \ \frac{1}{14}(3,3,11,2) \ \ \frac{1}{27}(3,5,11,14)$ 
 & klt
 \\  \hline
 218
 & $(3,5^2,11,14,28)$
 & 33
 & $\frac{1}{5}(1,4,3) \ \  \frac{1}{14}(3,5,11) \ \ \frac{1}{28}(3,5,11,14)$ 
 & terminal
 \\  \hline
 219
 & $(3,5,6,11,14,27)$
 & 33
 & $\frac{1}{3}(2,2,2) \ \  \frac{1}{6}(5,5,2,3) \ \ \frac{1}{27}(3,5,11,14) \ \ \frac{1}{5}(1,1,4,2)  \ \ \frac{1}{14}(3,6,11,13) \ \ \frac{1}{27}(3,5,11,14)$ 
 & klt
 \\  \hline
 220
 & $(3,5,7,11,14,26)$
 & 33
 & $\frac{1}{5}(2,1,4,1) \ \  \frac{1}{7}(3,4,5) \ \ \frac{1}{14}(3,7,11,12) \ \ \frac{1}{2}(1,1,1)  \ \ \frac{1}{26}(3,5,11,14)$ 
 & canonical
 \\  \hline
 221
 & $(3,5,10,11,14,23)$
 & 33
 & $\frac{1}{5}(3,1,4) \ \  \frac{1}{10}(3,5,1,4) \ \ \frac{1}{2}(1,1,1) \ \ \frac{1}{14}(3,10,11,9)  \ \ \frac{1}{23}(3,5,11,14)$ 
 & terminal
 \\  \hline
 222
 & $(3,5,11,11,14,22)$
 & 33
 & $\frac{1}{5}(1,1,4,2) \ \  \frac{1}{11}(3,5,3) \ \ \frac{1}{22}(3,5,11,14) \ \ \frac{1}{2}(1,1,1)  \ \ \frac{1}{14}(3,11,11,8)$ 
 & klt
 \\  \hline
 223
 & $(3,5,11,14,14,19)$
 & 33
 & $\frac{1}{5}(1,4,4,4) \ \  \frac{1}{14}(3,5,11) \ \ \frac{1}{19}(3,5,11,14)$
 & terminal
 \\  \hline
 224
 & $(3,5,11,14,15,18)$
 & 33
 & $\frac{1}{3}(2,2,2) \ \  \frac{1}{15}(5,11,14,3) \ \ \frac{1}{18}(3,5,11,14) \ \ \frac{1}{5}(1,4,3) \ \ \frac{1}{14}(3,11,1,4)$ 
 & klt
 \\  \hline
 225
 & $(1^2,5,12,18,35)$
 & 36
 & $\frac{1}{5}(1,2,3) \ \  \frac{1}{3}(1,1,2,2) \ \ \frac{1}{35}(1,5,12,18)$
 & terminal
 \\  \hline
 226
 & $(1,3,4,11,18,35)$
 & 36
 & $\frac{1}{3}(1,1,2,2) \ \ \frac{1}{2}(1,1,1,1) \ \
 \frac{1}{11}(1,4,7,2) \ \ \frac{1}{35}(3,4,11,18)$ 
 & terminal
 \\  \hline
 227 
 & $(1,5^2,12,18,31)$
 & 36
 & $\frac{1}{5}(1,2,3) \ \ \frac{1}{6}(1,5,5,1) \ \
 \frac{1}{31}(1,5,12,18)$
 & terminal
 \\ \hline 
 228 
 & $(1,5,7,12,18,29)$
 & 36
 & $\frac{1}{5}(2,2,3,4) \ \ \frac{1}{7}(5,5,4,1) \ \
 \frac{1}{29}(11,5,12,18)\ \ \frac{1}{6}(1,5,1,5)  $
 & terminal
 \\ \hline 
 229
 & $(2,4,11,11,18,25)$
 & 36
 & $\frac{1}{2}(1,1,1) \ \ \frac{1}{11}(2,4,7) \ \
 \frac{1}{25}(2,4,11,18)$
 & klt
 \\ \hline 
 230
 & $(2,3,5,11,19,36)$
 & 38
 & $\frac{1}{2}(1,1,1,1) \ \ \frac{1}{36}(3,5,11,19) \ \
   \frac{1}{3}(5,11,19) \ \ \frac{1}{5}(2,11,19,36)$
 & terminal
 \\ \hline 
 231
 & $(1,5,7,8,20,39)$
 & 40
 & $\frac{1}{4}(1,1,3,3) \ \ \frac{1}{5}(1,2,3,4) \ \
   \frac{1}{7}(1,1,6,4) \ \ \frac{1}39(5,7,8,20)$
 & terminal
 \\ \hline 
 232
 & $(5,7,7,8,20,33)$
 & 40
 & $\frac{1}{4}(1,1,3,3) \ \ \frac{1}{5}(2,2,3,3) \ \ \frac{1}{7}(1,1,6,4) \ \
   \frac{1}{39}(5,11,19) \ \ \frac{1}{39}(5,7,8,20)$
 & terminal
 \\ \hline 
 233
 & $(5,7,8,11,20,29)$
 & 40
 & $\frac{1}{5}(2,3,1,4) \ \ \frac{1}{7}(1,4,6,1) \ \ \frac{1}{5}(5,7,8,9) \ \
   \frac{1}{29}(5,7,8,20) \ \ \frac{1}{4}(1,3,3,1)$
 & terminal
 \\ \hline 
 234
 & $(1,1,6,14,21,41)$
 & 40
 & $\frac{1}{7}(1,1,6,6) \ \ \frac{1}{41}(1,6,14,21) \ \ \frac{1}{2}(1,1,1,1) \ \
   \frac{1}{3}(1,1,2,2)$
 & terminal
 \\ \hline 
 235
 & $(1,2,5,14,21,41)$
 & 42
 & $\frac{1}{2}(1,1,1,1) \ \ \frac{1}{5}(1,4,1,1) \ \ \frac{1}{7}(1,2,5,6) \ \
   \frac{1}{41}(2,5,14,21)$
 & terminal
 \\ \hline 
 236
 & $(2,3,5,14,21,39)$
 & 42
 & $\frac{1}{2}(1,1,1,1) \ \ \frac{1}{3}(2,2,2) \ \ \frac{1}{39}(2,5,14,21) \ \
   \frac{1}{5}(3,4,1,4) \ \ \frac{1}{7}(2,3,5,4)$
 & klt
 \\ \hline 
 237
 & $(2,5,5,14,21,37)$
 & 42
 & $\frac{1}{2}(1,1,1,1) \ \ \frac{1}{5}(2,4,1) \ \ \frac{1}{37}(2,5,14,21) \ \
   \frac{1}{7}(2,5,5,2)$
 & terminal
\\ \hline 
 238
 & $(1,4,5,13,22,43)$
 & 44
 & $\frac{1}{2}(1,1,1,1) \ \ \frac{1}{5}(1,4,2,3) \ \ \frac{1}{13}(1,4,9,4) \ \
   \frac{1}{43}(4,5,13,22)$
 & terminal
 \\ \hline 
 239
 & $(4,5,5,13,22,39)$
 & 44
 & $\frac{1}{2}(1,1,1,1) \ \ \frac{1}{5}(4,13,2) \ \ \frac{1}{13}(4,5,5,9) \ \
   \frac{1}{39}(4,5,13,22)$
 & terminal
 \\ \hline 
 240
 & $(4,5,11,13,22,33)$
 & 44
 & $\frac{1}{2}(1,1,1,1) \ \ \frac{1}{5}(1,3,2,3) \ \ \frac{1}{11}(4,5,2) \ \
   \frac{1}{33}(4,5,13,22) \ \ \frac{1}{13}(4,11,9,7)$
 & klt
 \\ \hline 
 241
 & $(4,5,13,13,22,31)$
 & 44
 & $\frac{1}{2}(1,1,1,1) \ \ \frac{1}{5}(3,3,2,1) \ \ \frac{1}{13}(4,5,9) \ \
   \frac{1}{31}(4,5,13,22)$
 & terminal
 \\ \hline 
 242
 & $(1,3,5,16,24,47)$
 & 48
 & $\frac{1}{3}(1,5,16,47) \ \ \frac{1}{5}(1,1,4,2) \ \ \frac{1}{4}(1,3,1,3) \ \
   \frac{1}{47}(3,5,16,24)$
 & terminal
 \\ \hline 
 243
 & $(3,5,5,16,24,43)$
 & 48
 & $\frac{1}{5}(3,1,4) \ \ \frac{1}{4}(3,1,1,3) \ \ \frac{1}{43}(3,5,16,24)$ 
 & terminal
 \\ \hline 
 244
 & $(1,5,16,22,33,65)$
 & 66
 & $\frac{1}{5}(1,2,3) \ \ \frac{1}{65}(5,6,22,33) \ \ \frac{1}{2}(1,1,1,1) \ \
   \frac{1}{11}(1,5,6,10) \ \ \frac{1}{3}(1,2,1,2)$
 & terminal
\end{longtable}
\rowcolors{0}{}{}
\renewcommand{\arraystretch}{1.0}
\normalsize

\begin{corollary} \label{cor:terminality}
    There are exactly 197 families of terminal FK3 weighted Fano fourfold hypersurfaces.
\end{corollary}

\section{Other interesting cases}
\label{sec:others}
If we drop the hypothesis on the dimension of the singular locus of the FK3 weighted fourfold, we obtain new interesting examples, such as those in the following table.

\small
\renewcommand{\arraystretch}{1.3}
\rowcolors{2}{gray!10}{white}
\begin{longtable}{c||c|c|c}
\caption{Families of FK3 weighted fourfolds $X_d\subseteq \P\bigl(a_0,a_1,a_2,a_3,a_4,\frac{d}{2}\bigl)$ such that $a_4|d$, $a_4\neq\frac{d}{2}$. 
}
\label{table fourfolds from del Pezzo}
\\ 
 \rowcolor{white}
 \(\#\) 
 & $(\underline a)$
 & $d$
 & $h^{2,2}$
 \\ \hline \hline \endfirsthead
 1
 & $(1,2^4,3)$
 & 6
 & 15
 \\
 2
 & $(3^2,4^3,6)$
 & 12
 & 3
 \renewcommand{\arraystretch}{1.0}
\end{longtable}
\normalsize

\begin{remark}
    Observe that $a_4|d$ is not a redundant hypothesis: if we drop the hypothesis on the dimension of the singular locus, we cannot use \cref{prop a4 divide d} to conclude $a_4|d$ from $2a_5=d$, as we crucially used \cref{conditions}.(5) in the proof of the proposition.
\end{remark}

\begin{proposition}\label{prop table 3 is complete}
    \cref{table fourfolds from del Pezzo} lists all families of FK3 weighted fourfolds $X_d\subseteq \P\bigl(a_0,a_1,a_2,a_3,a_4,\frac{d}{2}\bigl)$ such that $a_4|d$, $a_4\neq\frac{d}{2}$.
\end{proposition}
\begin{proof}
A FK3 weighted fourfold as in the statement satisfies \cref{conditions}.(1)-(4). A straightforward computation in $a_i,d$ (see the first lines of the proof of \cref{prop a4=a5=d/2}) gives that in the above assumptions $a_4=\frac{d}{3}$, hence \cref{claim 2a5=3a4=2} gives
\begin{equation}\label{eq (a2,...,d)=(2k,...,6k)}
    (a_2,a_3,a_4,a_5,d)=(2k,2k,2k,3k,6k),\hspace{3mm}k\ge1.
\end{equation}
Combining \eqref{eq (a2,...,d)=(2k,...,6k)} with \cref{conditions}.(3) we get $a_0+a_1=3k$, and $a_1\le a_2=2k$ gives $a_1\ge\frac{3}{2}k$ and $a_0\ge k$. 

Assume $a_1=a_2$; we have $\gcd(a_1,\dots,a_5)=k$, hence by \cref{conditions}.(1) we get $k=1$ and $X_d$ is example $\#$1 in \cref{table fourfolds from del Pezzo}.

Assume $a_1<a_2$, hence $a_0>k$. The quasi-smoothness criterion of \cref{prop quasismooth} on $I=\{1\}$ gives the following options.
\begin{enumerate}
    \item $d=na_1$ for some $n\in\Z_{\ge 1}$; from $na_1=d=3a_2>3a_1$ and $6k=d=na_1\ge\frac{3}{2}nk$ we get $3<n\le 4$, hence $n=4$. We obtain $2a_1=3k$, hence $k=2\alpha$ for some $\alpha\in\Z_{\ge1}$ and $a_1=3\alpha$, $a_2=a_3=a_4=4\alpha$, $a_5=6\alpha$ hence $\gcd(a_1,\dots,a_5)=\alpha$, then $\alpha=1$ by \cref{conditions}.(1) and $X_d$ is example $\#$2 in \cref{table fourfolds from del Pezzo}. 
    \item $d=na_1+a_j$ for some $j\neq1$ and $n\in\Z_{\ge1}$. Assume $j=5$: we have $\frac{d}{2}=d-a_5=na_1$, then $d=2na_1$ and we are in case (1) here above. Assume $j=2,3$ or 4: we have $4k=d-a_j=na_1>nk$ and $4k=na_1<na_2=2nk$, then $n=3$ and $4k=3a_1$; we obtain $k=3\alpha$ for some $\alpha\in\Z_{\ge 1}$, hence $a_1=4\alpha$ and $a_0=3k-a_1=5\alpha$, which is a contradiction. Finally, assume $j=0$: we have $6k=na_1+a_0\ge(n+1)a_0>(n+1)k$ and $6k=na_1+a_0<(n+1)a_2=2(n+1)k$, hence $2<n<5$. If $n=3$, then $6k=d=3a_1+a_0$, which combined with $a_0+a_1=3k$ gives $2a_1=3k$, and we conclude, as in case (1) above, that $X_d$ is the example $\#$2 in \cref{table fourfolds from del Pezzo}. If $n=4$, then $6k=d=4a_1+a_0$, which combined with $a_0+a_1=3k$ gives $a_1=k$ and $a_2=2k$, which is a contradiction.
\end{enumerate}
We conclude that all possible cases are those written in \cref{table fourfolds from del Pezzo}, which is our statement.
\end{proof}

\begin{remark}
Examples in \cref{table fourfolds from del Pezzo} are examples N2, N3 in \cite[Table 1]{LPZ18}. Observe that every other example in \cite[Table 1]{LPZ18} is of the same type as the ones explained in \cref{remark double suspension}; those having singular locus of dimension at most 1, i.e.\ examples N4, N5, N8, N9, are examples \#3, 9, 10, 32 in \cref{table fourfolds} respectively.
\end{remark}

A natural aim is to give a geometric description of the examples in \cref{table fourfolds from del Pezzo} that explains their K3 structure, similarly to what was done for the examples in \cref{table fourfolds}. Observe that for the examples in \cref{table fourfolds from del Pezzo} it is not possible to find two indices $i,j$ such that $a_i+a_j=d$ (see \cref{prop blow-up}); on the other hand, in some special cases we can mimic the construction in \cref{remark double suspension}, which surprisingly in this case does not end up with a K3 surface, as explained in the following remark.

\begin{remark}[Cyclic examples]
Let $X_d\subseteq\P(a_0,\dots,a_4,\frac{d}{2})$ be one of the weighted hypersurfaces in \cref{table fourfolds from del Pezzo}. The defining equation $f_d$ of $X_d$ can always be assumed to be cyclic in $x_5$, by completing the square. We further assume that it is cyclic also in  $x_4$, i.e.\ that $f_d=x_5^2+x_4^3+g_d(x_0,\dots,x_3)$. 

Under this hypothesis the projection $\P\bigl(a_0,\dots,a_4,\frac{d}{2}\bigl)\dashrightarrow\P(a_0,\dots,a_4)$ restricts to a double cover $X_d\xrightarrow{\phi_1} \P\bigl(a_0,\dots,a_3,\frac{d}{2}\bigl)$ with branch locus $Y:=V(x_4^3+g_d)$, and the further projection $\P(a_0,\dots,a_4)\dashrightarrow\P(a_0,\dots,a_3)$ restricts to a $3:1$ cover $Y\xrightarrow{\phi_2} \P(a_0,\dots,a_3)$, with branch locus  $S_d:=V(g_d)$. 
Summarizing, we have the following diagram:
\[
\begin{tikzcd}
        & & X_d \ar[d, "\phi_1" left, "2:1" near start] \ar[r, hook] & \P\Bigl(a_0,\dots,a_3,a_4,\frac{d}{2}\Bigl) \\
        & Y \ar[r, hook] \ar[d, "\phi_2" left, "3:1" near start] & \P(a_0,\dots,a_4) &  \\
        S_d \ar[r, hook] & \P(a_0,\dots,a_3). &
\end{tikzcd}
\]
The surface $S_d\subseteq \P(a_0,\dots,a_3)$ is a weighted hypersurface: it is quasi-smooth, as $a_i|d$ for all i, and is well-formed and not a linear cone by direct computations. The adjuction formula \eqref{eq adjuction} on $S_d$ gives: 
\[\OO(K_{S_d})\cong \mathcal O_{S_d}\bigl(\text{deg}(S_d)-(a_0+\dots+a_3)\bigl)=\mathcal O_{S_d}\bigl(d-(2d-a_5-a_4)\bigl)=\mathcal O_{S_d}\Bigl(-\frac{d}{6}\Bigl)\]
where the central equality follows from \cref{conditions}.(3) and the last one from $3a_4=2a_5=d$. We find that $K_{S_d}$ is antiample, i.e.\ $S_d$ is a del Pezzo surface. 

We remark that the above construction is special in moduli, since the general element of the families in \cref{prop table 3 is complete} will not be cyclic in the variable $x_4$.

We also remark the similarity between this construction and the \emph{cyclic cubic fourfolds}, studied for example in \cite{BHS22,BCL23}, where both $3a_4=3a_5=d$.
\end{remark}

The examples in \cref{table fourfolds from del Pezzo} are quite interesting from the point of view of rationality. In particular, these are the only two families of the list in \cite{LPZ18} that do not satisfy the condition $d=2a_5=2a_4$, which would ensure their rationality by \cref{rationality}. In general, we do not see a way to check their rationality in an immediate way. On the other hand, we are not able to identify any \emph{obvious} K3 surface which would be responsible for the K3 structure in their Hodge theory or derived category of perfect complexes. 
In fact, the two families of fourfolds in \cref{table fourfolds from del Pezzo} are still covered of \cite[Corollary 4.2]{Kuz19}, and it would be interesting to understand if the relevant CY2 category in the sense of \cite[Example 6.11]{Per22} turned out to be commutative or not.

In particular, 
if they turned out to be rational, it would be interesting in light of extending Kuznetsov's rationality conjecture from the cubic fourfold to other (weighted) fourfolds of K3 type. We therefore end our paper with the following question.

\begin{question}
    Let $X$ be one of the two weighted fourfolds contained in \cref{prop table 3 is complete}. Is the general member of these families rational?
\end{question}

\nocite{*}
\printbibliography

\end{document}